\documentclass[12pt]{article}
\usepackage[latin9]{inputenc}
\usepackage{geometry}
\usepackage[toc,page]{appendix}
\usepackage{amssymb}
\usepackage[unicode=true,
bookmarks=false,
breaklinks=false,pdfborder={0 0 1},colorlinks=false]
{hyperref}
\usepackage{amssymb}
\usepackage{amsmath}
\usepackage{amsthm}
\usepackage{amsfonts}
\usepackage{mathrsfs}
\usepackage{authblk}
\theoremstyle{definition}
\newtheorem{theorem}{Theorem} 
\newtheorem{lemma}{Lemma}
\newtheorem{corollary}{Corollary}
\newtheorem{proposition}{Proposition}
\theoremstyle{remark}
\newtheorem{remark}{Remark}
\newtheorem{assumption}{Assumption}

\makeatletter

\usepackage{titlesec}
\titleformat{\section}
{\normalfont\normalsize\bfseries}{\thesection.}{1em}{}
\titleformat{\subsection}
{\normalfont\normalsize\bfseries\itshape}{\thesubsection.}{1em}{}
\titleformat{\subsubsection}
{\normalfont\normalsize\itshape}{\thesubsubsection.}{1em}{}

\usepackage{amsfonts}
\usepackage{graphicx}
\usepackage{amsthm}
\makeatother

\begin{document}
	\title{\textbf{Estimating non-linear functionals of trawl processes}}
	\author{Orimar Sauri}
	\affil{Department of Mathematical Sciences, Aalborg University, Denmark}
	
	\maketitle

	\begin{abstract}
		Trawl processes are a family of continuous-time, infinitely divisible, stationary processes whose correlation structure is entirely characterized by their so-called trawl function. This paper investigates the problem of estimating non-linear functionals of a trawl function under in-fill and long-span sampling schemes. Specifically, building on the work of \cite{SauriVeraart23}, we introduce non-parametric estimators for functionals of the type $\Psi_{t}(g)=\int_{0}^{t}g(a(s))\mathrm{d}s$ and $ \Lambda_t(g)=\int_{t}^{\infty}g(a(s))\mathrm{d}s$, where $a$ represents the trawl function of interest and $g$ a non-linear test function. We show that our estimator for $\Psi_{t}(g)$ is consistent and asymptotically Gaussian regardless of the memory of the process. We further demonstrate that the same phenomenon occurs for the estimation of $\Lambda_t(g)$ as long as $g(x)= \mathrm{O} (\lvert x\rvert^p)$, as $x\to0$, for some $p>3$. Additionally, we illustrate how our results can be used to construct a test statistic robust to memory effects for the presence of $T$-dependent.
	\end{abstract}
	
	\noindent%

	

		
		
		\section{Introduction}
		
		Trawl processes form a subclass of continuous-time, strictly stationary, and infinitely divisible processes. Specifically, a trawl process $X = (X_t)_{t \in \mathbb{R}}$ is constructed by evaluating a L\'evy basis -- also known as an infinitely divisible independently scattered random measure -- $L$ over a set $A_t$, known as a \textit{trawl set}, which is typically a subset of $\mathbb{R}^2$. The distribution of $X_t$ is completely determined by $L$,  while the autocorrelation function of $X$ is  fully described by the so-called \textit{trawl function} $a$. This relationship is expressed through the identity
		\[ \rho_X(h)=\frac{\int_{h}^{\infty}a(s)\mathrm{d}s}{\int_{0}^{\infty}a(s)\mathrm{d}s},\,\,\, h\geq0 .\]
		By design, a trawl process offers a highly flexible autocorrelation structure and produces a wide range of marginal distributions within the class of infinitely divisible distributions. These features make trawl processes suitable for modeling data with stylized facts such as non-Gaussianity, heavy tails, jumps, and persistence. This is reflected in their increasing popularity for modeling complex temporal phenomena across disciplines such as finance (\cite{BNLundShepVerr14,BenLundeShepVer23,ShepYan17,Veraart19,Veraart24}) and physics (\cite{HedSch13,MarquezSch16,Sch05}). 
		
		Regarding statistical inference for trawl processes, various parametric approaches have been explored in the literature. See for instance  \cite{Bacro20,BNLundShepVerr14,BenLundeShepVer23,ShephardYang16}. In this context, the work in \cite{SauriVeraart23} is the first to address non-parametric estimation of the trawl function and serves as the starting point for the present study. Specifically, that paper proves that, given $n$ equidistant observations of $X$, say $(X_{i\Delta_{n}})_{i=0}^{n}$,  the trawl function $a$  can be consistently estimated by
		\begin{align}
			\check{a}(t) & =-\Delta_{n}^{-1}\left[\hat{\varGamma}_{l+1}-\hat{\varGamma}_{l}\right],\,\,\,\text{if }\Delta_{n}l\leq t<(l+1)\Delta_{n},\label{eq:defahat}
		\end{align}
		assuming that as $n \to \infty$, $\Delta_n \to 0$, with $n \Delta_n \to \infty$, and provided the L\'evy seed of $L$ (see Section \ref{sec:Preliminaries}) has unit variance.  $\hat{\varGamma}_{l}$ in \eqref{eq:defahat} denotes the sample auto-covariance function of $X$ at lag $l\Delta_{n}$. Furthermore, under suitable conditions, the estimator
		$\check{a}(t)$ is asymptotically normal with asymptotic variance given by
		\begin{align*}
			\sigma_{a}^{2}(t):= & \mathfrak{K}_{4}a(t)+2\int_{0}^{\infty}a(s)^{2}\mathrm{d}s+2\int_{0}^{t}a(t-s)a(t+s)\mathrm{d}s\\
			& -2\int_{t}^{\infty}a(s-t)a(t+s)\mathrm{d}s,
		\end{align*}
		for some non-negative constant $\mathfrak{K}_{4}$. In the
		same work, we proposed to use  $\check{a}$ once again along with a ``Riemann sum'' approach to estimate
		the quadratic functionals of $a$ appearing
		in the definition of $\sigma_{a}^{2}(t)$, e.g. $\int_{0}^{\infty}a(s)^{2}\mathrm{d}s$. This procedure has some limitations. First, it relies on a tuning parameter whose choice and finite-sample performance are unclear. Second, the estimation error of $\int_{0}^{\infty}a(s)^{2}\mathrm{d}s$ is dependent on the behaviour of $\check{a}(t)$ for large $t$, which could result on a statistic that is very sensitive to the memory of the underlying stochastic process. In this paper, we address these issues by estimating functionals of the type
		\[
		\Psi_{t}(g)=\int_{0}^{t}g(a(s))\mathrm{d}s;\,\,\Lambda(g):=\int_{t}^{\infty}g(a(s))\mathrm{d}s,
		\]
		where $g$ is a continuous non-linear test function, based on the sample $(X_{i\Delta_{n}})_{i=0}^{n}$. 
		
		\subsection*{Main contributions}
		By slightly modifying the definition of $\check{a}$ in \eqref{eq:defahat}, which we denote by $\hat{a}$ (see \eqref{eq:def_ahat_modified} below), we derive limit theorems for the functionals 
		\begin{align*}\Psi_{t}^{n}(g):= & \Delta_{n}\sum_{l=0}^{[t/\Delta_{n}]-1}g(\hat{a}(l\Delta_{n})),\,\,\,\,0\leq t\leq(n-1)\Delta_{n};\\
			\Lambda_{t}^{n}(g):= & \Delta_{n}\sum_{l=[t/\Delta_{n}]}^{n-1}g(\hat{a}(l\Delta_{n}));\,\,0\leq t\leq(n-1)\Delta_{n},
		\end{align*}
		using a double-asymptotic scheme based on both long-span and high-frequency observations of the underlying trawl process $X$. The function $g$ is assumed to be continuous. 
		
		Under minimal integrability conditions on $X$ and a growth condition on $g$, we provide a functional law of large numbers for $\Psi^{n}(g)$. With stronger assumptions on the trawl function, we establish a functional Central Limit Theorem (CLT hereafter)) at the rate $\sqrt{n\Delta_n}$, applicable to test functions that are twice continuously differentiable and of polynomial growth. Additionally, in line with \cite{SauriVeraart23}, we show that this convergence is independent of the memory  of $X$. In contrast, the asymptotic behaviour of $\Lambda^{n}(g)$ is more subtle. First, we demonstrate that, in general, $\Lambda_{t}^{n}(g)$ is not consistent for $\int_{t}^{\infty}g(a(s))\mathrm{d}s$ unless $g(x)= \mathrm{O} (\lvert x\rvert^p)$, as $x\to0$, for some $p>2$. In the latter situation with $p>3$,  $\Lambda_{t}^{n}(g)$ also fulfils a CLT analogous to that of  $\Psi^{n}(g)$. Secondly, to mitigate the inconsistency of $\Lambda^{n}(g)$ we alter its definition (see equation \eqref{def_Lambdabar}) following the approach suggested in \cite{SauriVeraart23}. The resulting estimator $\bar{\Lambda}^{n}(g)$ is consistent and satisfies a CLT at rate $\sqrt{n\Delta_n}$. However, this modification requires a tuning parameter and its second-order asymptotics are dependent on the memory of the process and the sample scheme. Finally, to illustrate the usefulness of our results, we propose a method for testing the presence of $T$-dependent data using a statistic that is robust to long memory.
		\subsection*{Outline of the paper}
		
		The remainder of the paper is organised as follows. Section 2 introduces the main mathematical concepts and notation used throughout the paper and recalls the definition and basic properties of trawl processes. Section 3 presents the main results concerning the asymptotics of the functionals $\Psi^{n}(g)$ and $\Lambda^{n}(g)$, along with our new test for $T$-dependence. Due to the technical nature of the arguments, most proofs are deferred to Section 4. The paper also contains an appendix where we carry out the computations required to approximate the asymptotic variance appearing in our central limit theorems. The arguments done there are primarily algebraic and technical in nature, and the appendix may therefore be skipped on a first reading.
		
		\section{Preliminaries\label{sec:Preliminaries}}
		
		This part introduces the basic notation and concepts that
		will be used in this paper. Throughout the following sections
		$\left(\Omega,\mathcal{F},\mathbb{P}\right)$ denotes a complete
		probability space. As usual, a function $f:\mathbb{R}\rightarrow\mathbb{R}$
		is said to be of class $\mathcal{C}^{N}$, $N\in\mathbb{N}$,
		if it is $N$-times continuously differentiable. In this case, $f^{(k)}$ will denote the $k$th derivative of $f$, with
		the convention that $f^{(0)}=f$. Often, we will write $f^{\prime}$
		instead of $f^{(1)}$. The \textsl{symmetrization of a function} $f:\mathbb{R}^{2}\rightarrow\mathbb{R}$
		is defined and denoted as $\tilde{f}(x,y):=f(x,y)+f(y,x)$.
		
		The symbols $\overset{\mathbb{P}}{\rightarrow}$ and $\overset{d}{\rightarrow}$
		stand, respectively, for convergence in probability and in distribution
		of random vectors (r.v.'s for short). For a sequence of random vectors $(\xi_n)_{n\ge1}$ defined on $(\Omega,\mathcal{F},\mathbb{P})$, we write
		$\xi_n=\mathrm{o}_{\mathbb{P}}(1)$ if $\xi_n \overset{\mathbb{P}}{\rightarrow} 0$, and
		$\xi_n=\mathrm{O}_{\mathbb{P}}(1)$ if $(\xi_n)_{n\ge1}$ is bounded in probability. For $T>0$, let $(Z_{t}^{n})_{0\leq t\leq T,n\in\mathbb{N}}$
		be a sequence of c\`adl\`ag processes defined on $\left(\Omega,\mathcal{F},\mathbb{P}\right)$.
		We will write $Z^{n}\overset{u.c.p}{\longrightarrow}Z$ if $Z^{n}$
		converges uniformly on compacts in probability to $Z$. Similarly,
		$Z^{n}\overset{\mathcal{D}([0,T])}{\Longrightarrow}Z$ stands for weak
		convergence of $Z^{n}$ towards $Z$ in the Skorokhod topology.
		
		We recall that a real-valued random field $L=\{L\left(A\right):A\in\mathcal{B}_{b}(\mathbb{R}^{d})\}$
		on $\left(\Omega,\mathcal{F},\mathbb{P}\right)$ is called a \textit{homogeneous
			L\'evy basis} if it is an infinitely divisible (ID for short) independently
		scattered random measure such that
		\begin{equation}\label{Levy-Kintchine_form}
			\mathbb{E}\big(\exp(\mathbf{i}zL(A))\big)
			=
			\exp\big(\mathrm{Leb}(A)\psi(z)\big),\,\,\forall\,A\in\mathcal{B}_{b}(\mathbb{R}^{d}),z\in\mathbb{R},
		\end{equation}
		where $\mathrm{Leb}$ and $\mathcal{B}_{b}(\mathbb{R}^{d})$  denote the Lebesgue measure on $\mathbb{R}^{d}$ and the Borel sets of $\mathbb{R}^{d}$ with finite Lebesgue measure, respectively. Furthermore,
		\begin{equation}
			\psi(z):=\mathbf{i}z\gamma-\frac{1}{2}b^{2}z^{2}+\int_{\mathbb{R}\backslash\{0\}}(e^{\mathbf{i}zx}-1-\mathbf{i}zx\mathbf{1}_{\rvert x\rvert\leq1})\nu(\mathrm{d}x),\label{chexpdef}
		\end{equation}
		with $\gamma\in\mathbb{R},$ $b\geq0$ and $\nu$ a L\'evy measure,
		i.e. $\int_{\mathbb{R}\backslash\{0\}}(1\land\rvert x\rvert^{2})\nu(\mathrm{d}x)<\infty$ and $\nu$ does not charge ${0}$.
		We will refer to the ID random variable with characteristic triplet
		$\left(\gamma,b,\nu\right)$ as the \textit{L\'evy seed} of $L$, and
		it will be denoted by $L'$. As it is customary, $\left(\gamma,b,\nu\right)$
		will be called the \textit{characteristic triplet} of $L$ and $\psi$
		its \textit{characteristic exponent}. We will say that \textit{$L$
			has unit variance} if $\mathrm{Var}(L')=1$.
		
		Let $L$ be a homogeneous L\'evy basis on $\mathbb{R}^{2}$ with characteristic
		triplet $(\gamma,b,\nu)$. In addition, we set 
		\begin{equation}
			A=\left\{ (r,y):r\leq0,0\leq y\leq a(-r)\right\} .\label{eq:trawlset}
		\end{equation}
		where $a:\mathbb{R}^{+}\rightarrow\mathbb{R}^{+}$ is a non-increasing,
		continuous, and integrable function. The process following the dynamics
		\begin{equation}
			X_{t}:=L(A_{t}),\,\,\,t\in\mathbb{R},\label{trawldef-1}
		\end{equation}
		where $A_{t}:=A+(t,0)$, is called a \textit{trawl process}. To avoid trivial situations, we will always assume that $\mathrm{Leb}(A)>0$, or equivalently that $a$ is not equal to $0$ almost everywhere. From now
		on, we will allude to $A$ and $a$ as \textit{trawl set and trawl
			function}, respectively. Furthermore, $L$ is termed the \textit{background driving L\'evy basis}.
		
		It is well known that $X$ is strictly stationary
		and, in the case when $L^{\prime}$ is square integrable, its auto-covariance
		function is given by 
		\begin{equation}
			\varGamma_{X}(h):=\mathrm{Var}(L')\int_{\left|h\right|}^{\infty}a(u)\mathrm{d}u,\,\,\,h\in\mathbb{R}.\label{TrawlACF-1}
		\end{equation}
		For a detailed exposition on the basic properties of trawl processes
		and L\'evy bases we refer the reader to \cite{BNBEnthVeraat18} and
		references therein.
		
		\section{Estimating non-linear functionals of trawl functions}
		
		In what is left of this work, $X$ will denote a trawl process with trawl function $a$ and whose background driving L\'evy basis has unit variance. As discussed in the introduction, our
		main goal is to estimate the functionals
		\begin{equation}
			\Psi_{t}(g)=\int_{0}^{t}g(a(s))\mathrm{d}s;\,\,\Lambda_{t}(g):=\int_{t}^{\infty}g(a(s))\mathrm{d}s,\label{eq:deffunctionals}
		\end{equation}
		when equidistant observations of $X$, say $(X_{i\Delta_{n}})_{i=0}^{n}$,
		are available. The sample scheme considered in this work is an in-fill and long-span,
		i.e. 
		\begin{equation}
			n\uparrow\infty;\,\,\Delta_{n}\downarrow0;\,\,n\Delta_{n}\rightarrow+\infty.\label{eq:highfreq_longspan_scheme}
		\end{equation}
		To do so, we slightly modify the definition of $\check{a}$ in (\ref{eq:defahat})
		as follows:
		\begin{equation}
			\hat{a}(t):=-\frac{1}{n\Delta_{n}}\sum_{k=l}^{n-1}(X_{(k-l)\Delta_{n}}-\bar{X}_{n})\delta_{k}X,\,\,l\Delta_{n}\leq t<(l+1)\Delta_{n},\label{eq:def_ahat_modified}
		\end{equation}
		for some $l=0,\ldots,n-1$ and where $\delta_{k}X:=X_{(k+1)\Delta_{n}}-X_{k\Delta_{n}}$. It is
		not difficult to see that $\hat{a}(t)$ and $\check{a}(t)$ are asymptotically
		equivalent. The estimators for the functionals appearing in (\ref{eq:deffunctionals})
		are constructed as 
		\begin{equation}
			\begin{aligned}\Psi_{t}^{n}(g):= & \Delta_{n}\sum_{l=0}^{[t/\Delta_{n}]-1}g(\hat{a}(l\Delta_{n})),\,\,\,\,0\leq t\leq(n-1)\Delta_{n};\\
				\Lambda_{t}^{n}(g):= & \Delta_{n}\sum_{l=[t/\Delta_{n}]}^{n-1}g(\hat{a}(l\Delta_{n}));\,\,0\leq t\leq(n-1)\Delta_{n},
			\end{aligned}
			\label{eq:estimators_def}
		\end{equation}
		where $[x]$ denotes the integer part of $x\in\mathbb{R}$.
		
		\subsection{Limit theorems for $\Psi^{n}(g)$ and $\Lambda^{n}(g)$}
		
		In this part we describe first and second order asymptotics for the
		estimators introduced in (\ref{eq:estimators_def}). We start by discussing the
		consistency of these statistics. We recall that a real-valued
		function $f:\mathbb{R}\rightarrow\mathbb{R}$ is said to be of \textit{polynomial
			growth} of order $q\geq0$ if there is some $C>0$, such that 
		\[
		\rvert f(x)\rvert\leq C(1+\rvert x\rvert^{q}),\,\,\forall\,x\in\mathbb{R}.
		\]
		The reader should keep in mind that our sampling scheme satisfies
		(\ref{eq:highfreq_longspan_scheme}) and this is assumed throughout the rest of the paper. Under this set-up we have the
		following consistency result for $\Psi^{n}(g)$.
		
		\begin{theorem}\label{consitency_Phin}Let $g:\mathbb{R}\rightarrow\mathbb{R}$
			be a continuous function of polynomial growth of order $q\geq0$.
			Suppose that $\mathbb{E}(\rvert L^{\prime}\rvert^{2q\lor4})<\infty$.
			Then, $\Psi^{n}(g)\overset{u.c.p}{\rightarrow}\Psi(g).$\end{theorem}
		
		Let us now turn our attention to $\Lambda^{n}(g)$. To guarantee integrability
		of the mapping $s\mapsto g(a(s))$, and hence the well-definedness of
		$\Lambda_{t}(g)$, we need to restrict the type of test functions. For a given $p,q\geq0$, $d\in\mathbb{N}$, we will
		write $\mathfrak{C}_{p,q}^{d}$ for the family of functions of class
		$\mathcal{C}^{d}$ satisfying that $g^{(j)}(0)=0$ for $j=0,1,\ldots,d-1$
		and that $g^{(d)}(x)=\mathrm{O}(\rvert x\rvert^{p})$ as $x\rightarrow0$
		and $g^{(d)}(x)=\mathrm{O}(\rvert x\rvert^{q})$ as $x\rightarrow+\infty$.
		For example, the power function $g(x)=\rvert x\rvert^{\alpha}$, for
		$\alpha>1$ belongs to $\mathfrak{C}_{p,q}^{[\alpha]}$ with $p=q=\alpha-[\alpha]$.
		
		\begin{remark}\label{integrability_C2_pq}Note that if $g\in\mathfrak{C}_{p,q}^{d}$,
			an induction argument along with the Mean-Value Theorem imply that
			$g^{(j)}(x)=\mathrm{O}(\rvert x\rvert^{d-j+p})$, as $x\rightarrow0$,
			for all $j=1,\ldots,d$. Thus, in view of the fact $a$ is decreasing and bounded by $a(0),$
			we can find a constant $C=C(a(0))>0$, such that for all $s\geq0$
			\begin{equation}
				\rvert g^{(j)}(a(s))\rvert\leq Ca(s)^{d-j+p},\,\,\,j=1,\ldots,d.\label{eq:boundsforg_gprime}
			\end{equation}
			It follows that the mappings $s\in[0,+\infty)\mapsto g^{(j)}(a(s))$
			are continuous, bounded, and integrable for $j=0,1,\ldots,d\land(d-1+[p])$.
			In particular, $\Lambda_{t}(g)$ is well-defined for every $t\geq0$.
		\end{remark}
		
		\begin{theorem}\label{consitency_Lambdan} Let $g\in\mathfrak{C}_{p,q}^{1}$
			with $p,q>1$ and assume that $\mathbb{E}(\rvert L^{\prime}\rvert^{2((q\lor p)+1)})<\infty$.
			Then, as $n\rightarrow\infty$, $\Lambda^{n}(g)\overset{u.c.p}{\rightarrow}\Lambda(g).$\end{theorem}
		The previous result does not cover the quadratic case, i.e. when $p\land q=1$.
		This is because in this situation $\Lambda^{n}(g)$ might be biased,
		as the following result shows.
		\begin{theorem}\label{bias_quadratic}Let $g(x)=x^{2}$ and assume
			that $\mathbb{E}(\rvert L^{\prime}\rvert^{4})<\infty$. Then, as $n\rightarrow\infty$,
			$\Lambda_{0}^{n}(g)\overset{\mathbb{P}}{\rightarrow}2\Lambda_{0}(g).$
		\end{theorem}
		
		To deal with the previous situation, we proceed as in \cite{SauriVeraart23} and  modify $\Lambda^{n}(g)$ as
		follows: We introduce a tuning parameter $N_{n}\in\mathbb{N}$, $N_{n}\leq n$,
		such that 
		\begin{equation}	i)\,N_{n}\uparrow\infty;\,\,ii)\,N_{n}/n\rightarrow0;\,\,iii)\,N_{n}\Delta_{n}\rightarrow+\infty.\label{eq:window_assumption}
		\end{equation}  
		We set
		\begin{equation}\label{def_Lambdabar}
			\bar{\Lambda}_{t}^{n}(g):=\Delta_{n}\sum_{l=[t/\Delta_{n}]}^{N_{n}-1}g(\hat{a}(l\Delta_{n}));\,\,0\leq t\leq(N_{n}-1)\Delta_{n}.
		\end{equation}

		The sequence of processes $\bar{\Lambda}^{n}$, unlike $\Lambda^{n}$,
		is consistent in the quadratic case. More precisely:
		\begin{theorem}\label{consitency_Lambdan_bar}Consider $N_{n}\in\mathbb{N}$
			as in (\ref{eq:window_assumption}). Let $p,q\geq1$ and assume
			that $\mathbb{E}(\rvert L^{\prime}\rvert^{2((q\lor p)+1)})<\infty$.
			Then for all $g\in\mathfrak{C}_{p,q}^{1}$, $	\bar{\Lambda}_{t}^{n}(g)\overset{u.c.p}{\rightarrow}\Lambda(g)$, as $n\rightarrow\infty.$
		\end{theorem}
		We now discuss second-order asymptotics for our statistics. As usual, this requires stronger assumptions. We essentially follow the set-up of \cite{SauriVeraart23}.
		
		\begin{assumption}\label{as:trawl} The sampling scheme satisfies
			that $n\Delta_{n}^{3}\rightarrow0$ as $n\rightarrow\infty$. Furthermore,
			the trawl function admits the representation $a(s)=\int_{s}^{\infty}\phi(y)dy$,
			for $s\geq0$, where $\phi$ is a strictly positive c\`adl\`ag  function such that $\phi(s)=\mathrm{O}(s^{-\alpha-1})$ as $s\rightarrow+\infty$
			for some $\alpha>1$.
			
		\end{assumption} 
		
		Before presenting the Central Limit Theorems for $\Psi^{n}, \Lambda^{n}$, and $\bar{\Lambda}^{n}$ let
		us introduce some functions that will serve to describe their asymptotic
		variance. For $s,r\geq0$ and $\mathfrak{K}_{4}:=\int x^{4}\nu(dx)$,
		we put
		\begin{align*}
			\sigma_{1}(s,r) & :=\mathfrak{K}_{4}a(s\lor r);\\
			\sigma_{2}(s,r) & :=\int_{0}^{\infty}a(u)a(\rvert u-(s-r)\rvert)\mathrm{sgn}(u-(s-r))\mathrm{d}u;\\
			\sigma_{3}(s,r) & :=\int_{0}^{\infty}a(u+r)a(\rvert s-u\rvert)\mathrm{sgn}(s-u)\mathrm{d}u.
		\end{align*}
		
		\begin{theorem}\label{CLT_Phin}Let Assumption \ref{as:trawl} hold.
			If $g:\mathbb{R}\rightarrow\mathbb{R}$ is of class $\mathcal{C}^{2}$
			with $g^{(2)}$ being of polynomial growth of order $q\geq0$,
			and $\mathbb{E}(\rvert L^{\prime}\rvert^{2(q+2)\lor r_{0}})<\infty$
			for some $r_{0}>4$, then as $n\rightarrow\infty$,
			\[
			\sqrt{n\Delta_{n}}(\Psi^{n}(g)-\Psi(g))\overset{\mathcal{D}([0,T])}{\Longrightarrow}Z,
			\]
			where $(Z_{t})_{t\geq0}$ is a centred Gaussian process with 
			\[
			\mathbb{E}(Z_{t}Z_{s})=\int_{0}^{t}\int_{0}^{s}g^{\prime}(a(u))\Sigma_{a}(u,r)g^{\prime}(a(r))\mathrm{d}r\mathrm{d}u,\,\,\,t,s\geq0,
			\]
			in which
			\begin{equation}
				\begin{aligned}\Sigma_{a}(u,r) & :=\sigma_{1}(u,r)+\tilde{\sigma}_{2}(u,r)+\tilde{\sigma}_{3}(u,r),\end{aligned}
				\label{eq:AVAR_ALL}
			\end{equation}
			with $\tilde{\sigma}_{2}$ and $\tilde{\sigma}_{3}$ the symmetrizations (see Section \ref{sec:Preliminaries}) of $\sigma_{2}$ and $\sigma_{3}$, respectively.
		\end{theorem}
		The second order asymptotics for the functionals $\Lambda^{n}$ and $\bar{\Lambda}^{n}$
		will be much alike to those presented in the previous theorem. However, we need to impose certain restrictions on the window $N_n$ to control
		the discretization error associated with $\bar{\Lambda}^{n}$, as well as the memory of $X$. Specifically, we assume that:
		
		\begin{assumption}\label{as:samplescheme} There are $1<\varpi<3,\,\theta>0$,
			and $0<\kappa<1$, such that $n\Delta_{n}^{\varpi}$ is bounded from
			above and below, and $N_{n}\sim\theta n^{\kappa}$, as $n\rightarrow\infty$.
		\end{assumption} 
		The CLT for the functionals $\Lambda^{n}$ and $\bar{\Lambda}^{n}$
		reads as follows.
		
		\begin{theorem}\label{CLT_Lambdan_hatLambdan} Let Assumption
			\ref{as:trawl} hold. Suppose that for some $r_0>4$, $\mathbb{E}(\rvert L^{\prime}\rvert^{2((q\lor p)+2)\lor r_0})<\infty$, and that  $g\in\mathfrak{C}_{p,q}^{2}$.
			Then the following holds:
			\begin{enumerate}
				\item If $p,q>1$, then as $n\rightarrow\infty$
				\begin{equation}
					\sqrt{n\Delta_{n}}(\Lambda^{n}(g)-\Lambda(g))\overset{\mathcal{D}([0,T])}{\Longrightarrow}Z,\label{eq:Lambda_CLT}
				\end{equation}
				where $(Z_{t})_{t\geq0}$ is a centred Gaussian process with 
				\[
				\mathbb{E}(Z_{t}Z_{s})=\int_{t}^{\infty}\int_{s}^{\infty}g^{\prime}(a(u))\Sigma_{a}(u,r)g^{\prime}(a(r))\mathrm{d}r\mathrm{d}u,
				\]
				with $\Sigma_{a}$ as in (\ref{eq:AVAR_ALL}). 
				\item If in addition Assumption \ref{as:samplescheme} is satisfied with
				\begin{equation}
					0<\frac{1}{\varpi}+\frac{\varpi-1}{2\varpi((2+p)\alpha-1)}<\kappa<\frac{1}{2}(1+\frac{1}{\varpi}),\label{eq:restriction_windows_thm}
				\end{equation}
				and $p,q\geq0$ arbitrary, then \eqref{eq:Lambda_CLT} also holds if we replace $\Lambda^{n}$ by $\bar{\Lambda}^{n}$.
			\end{enumerate}
		\end{theorem}
		\begin{remark}\label{Rmk_cont_Z} The following remarks are in order:
			\begin{enumerate}
				
				\item 		Let $Z$ be as in Theorems \ref{CLT_Phin}
				and \ref{CLT_Lambdan_hatLambdan}. From Remark \ref{integrability_C2_pq},
				it follows that for every $0\leq s\leq t$
				\[
				\mathbb{E}(\rvert Z_{t}-Z_{s}\rvert^{2})\leq C(t-s)^{2}.
				\]
				Therefore, by Kolmogorov's criterion, we deduce that $Z$ has $\lambda$-H\"older
				continuous paths for every $\lambda<1.$
				\item We are able to show that in the limiting case $p\wedge q=1$ -- recall that when $p\wedge q<1$, in general, $\Lambda^{n}(g)$ is not consistent --  the sequence	$\sqrt{n\Delta_{n}}(\Lambda^{n}(g)-\Lambda(g))$ is tight on $\mathcal{D}([0,T])$  but we were not able to identify its limit. However, we believe that, just as in the case $g(x)=x^2$, the limit process is Gaussian but no longer centred.
				\item 	The lower bound in  \eqref{eq:restriction_windows_thm}, which depends on the memory of $X$, prevents the tail $\int_{N_n\Delta_{n}}^{\infty}g(a(s))\mathrm{d}s$ from dominating the asymptotics. Thus, unlike Theorem \ref{CLT_Phin} and the first part of Theorem \ref{CLT_Lambdan_hatLambdan}, the CLT for $\bar{\Lambda}^{n}(g)$ is no longer robust to long memory.
			\end{enumerate}
			
		\end{remark}
		\subsection{A test for $T$-dependent trawls}
		In this part we briefly illustrate how our results can be used to test for $T$-dependence in a non-parametric framework. We recall that a stochastic process $(Y_t)_{t\in\mathbb{R}}$ is said to be \textit{$T$-dependent}, with $T>0$, if for a given $t_0\in\mathbb{R}$, the processes $(Y_t)_{t\leq t_0}$ and $(Y_t)_{t\geq T+ t_0}$ are independent (in the sense of finite-dimensional distributions). We have the following characterization of $T$-dependence for trawl processes.
		\begin{proposition}\label{T_depe} Let $X$ be a trawl process. Then, unless $L$ is deterministic, $X$ is $T$-dependent if, and only if, for all $t\geq T$, $\mathrm{Leb}(A_t\cap A)=0$.\end{proposition}
		\begin{proof} Since $X$ is an infinitely divisible process, $T$-dependence is equivalent to the condition that, for a given $t_0$, $X_s$ and $X_u$ are independent for all $s\leq t_0$ and $u\geq t_0+T$ (see, for instance, Exercise E 12.9 in \cite{Sato99}). Furthermore, by stationarity, the former condition is equivalent to the independence of $X_t$ and $X_0$ for all $t\geq T$.  Now note that, in view of  \eqref{Levy-Kintchine_form} and the independently scattered property of $L$, the characteristic exponent of the vector $(X_0,X_t)$ is equal to
			\[ [\psi(z_1)+\psi(z_2)]\mathrm{Leb}(A)+\mathrm{Leb}(A_t\cap A)[\psi(z_1+z_2)-\psi(z_1)-\psi(z_2)], \,\, \forall \,(z_1,z_2)\in\mathbb{R}^2.\]
			Consequently, $X_t$ and $X_0$ are independent  if, and only if, 
			\begin{equation}\label{independence_cond}
				\mathrm{Leb}(A_t\cap A)[\psi(z_1+z_2)-\psi(z_1)-\psi(z_2)]=0, \,\, \forall \,(z_1,z_2)\in\mathbb{R}^2.
			\end{equation}
			Thus, if $\mathrm{Leb}(A_t\cap A)=0$, for all $t\geq T$, the previous relation is trivially fulfilled. Reciprocally, if \eqref{independence_cond} is satisfied for all $t\geq T$, then either $\mathrm{Leb}(A_t\cap A)=0$ or $\psi(z_1+z_2)=\psi(z_1)+\psi(z_2)$ for all $ (z_1,z_2)\in\mathbb{R}^2$. If the latter holds, we would have that $\psi$ is continuous and additive, so linear (c.f. Cauchy's functional equation) which is the case only when $L$ is deterministic.
		\end{proof}
		The preceding result shows that, under our set-up,  $X$ is $T$-dependent if, and only if, $\int_{T}^{\infty}a(s)\mathrm{d}s=0$. Note that this condition is in turn equivalent to $\mathrm{Cov}(X_T,X_0)=0$ and thus one can use the sample autocovariance function to test $T$-dependence of $X$. However, it is well known that such a statistic is heavily dependent on the memory of the process. See for instance \cite{CohLind13,Drapatz17,HorvKoko08,Spangenberg21}. To avoid this problem, we propose to use another characterization of $T$-dependence for trawl processes, namely 
		\begin{equation}\label{H0_T_dep}
			\int_{T}^{\infty}\rvert a(s)\lvert^p\mathrm{d}s=0, \text{ for some (and hence all)  } p>0,
		\end{equation}
		which follows from the fact that $a$ is non-negative and non-increasing. Therefore, under the null hypothesis of $T$-dependence, we deduce, since $\mathrm{Leb}(A)>0$, that
		\[ \frac{\Lambda_{0}(g)}{\Psi_{T}(g)}=\frac{\int_{0}^{\infty}\rvert a(s)\lvert^p\mathrm{d}s}{\int_{0}^{T}\rvert a(s)\lvert^p\mathrm{d}s}=1, \]
		where $g(x)=\lvert x\rvert^p$. Based on the previous discussion, we propose to use the  statistic 
		\[\tau_{n,p}=\frac{\Lambda^n_{0}(g)}{\Psi_{T}^n(g)}-1=\frac{\Delta_{n}\sum_{l=0}^{n-1}\rvert\hat{a}(l\Delta_{n})\rvert^{p}}{\Delta_{n}\sum_{l=0}^{[T/\Delta_{n}]-1}\rvert\hat{a}(l\Delta_{n})\rvert^{p}}-1,\,\,  p>3, \]
		to test the hypothesis \eqref{H0_T_dep}. The asymptotic behaviour of $\tau_{n,p}$ is described in the following result.
		
		\begin{corollary}
			Let Assumption \ref{as:trawl} hold and assume that  $\mathbb{E}(\rvert L^{\prime}\rvert^{2p})<\infty$ for some $p>3$. Then, 
			\begin{equation}\label{asymp_test}
				\sqrt{n\Delta_{n}}\tau_{n,p}\overset{\mathbb{P}}{\to}\begin{cases}
					0 & \text{if } \eqref{H0_T_dep} \text{ holds} ;\\
					+\infty & \text{otherwise.}
				\end{cases} 
			\end{equation}
		\end{corollary}
		\begin{proof}
			Plainly, $\tau_{n,p}=\Lambda^n_{T}(g)/\Psi_{T}^n(g)$, with $g(x)=\lvert x\rvert^p$. Thus, if \eqref{H0_T_dep} is satisfied, Theorem \ref{CLT_Lambdan_hatLambdan} implies that $\sqrt{n\Delta_{n}}\Lambda^n_{T}(g)\overset{d}{\to}\zeta$, where $\zeta$ is a centred Gaussian random variable with variance
			\[ p^{2}\int_{T}^{\infty}\int_{T}^{\infty}\rvert a(u)\rvert^{p-1}\Sigma_{a}(u,r)\rvert a(r)\rvert^{p-1}\mathrm{d}u\mathrm{d}r=0,\]
			thanks to \eqref{H0_T_dep}. This shows the first part of \eqref{asymp_test}. In contrast, if \eqref{H0_T_dep} does not hold then necessarily $\int_{0}^{T}\rvert a(s)\lvert^p\mathrm{d}s>0$. Theorem \ref{consitency_Lambdan} now gives $\tau_{n,p}\overset{\mathbb{P}}{\to}\frac{\int_{T}^{\infty}\rvert a(s)\lvert^p\mathrm{d}s}{\int_{0}^{T}\rvert a(s)\lvert^p\mathrm{d}s}>0$, from which the second part of  \eqref{asymp_test} follows.
		\end{proof}
		The previous result shows that, unlike the sample autocovariance function, the limiting behaviour of $\tau_{n,p}$ under the alternative of no $T$-dependence is independent of the memory of the process. Consequently, large values of $\sqrt{n\Delta_n}\tau_{n,p}$ indicate that the data is unlikely to be $T$-dependent.
		\begin{remark} As pointed out in Remark \ref{Rmk_cont_Z}, our proposed statistic for testing \eqref{H0_T_dep} is robust to the memory of $X$. Moreover, in this setting $\tau_{n,p}$ has a smaller estimation error. However, it still suffers from some weaknesses. First, to rigorously use $\tau_{n,p}$ for testing $T$-dependence, a CLT under the null hypothesis \eqref{H0_T_dep} is required. Second, it remains to be seen how sensitive $\tau_{n,p}$ is in finite samples for different values of $p$. These issues require additional analysis, which we postpone to future research.
		\end{remark} 
		
		\section{Proofs}
		The proofs are organised in three parts. We first derive a collection of preliminary estimates and decompositions for the estimator $\hat a$, which provide the basic tools used throughout the section. We then use these ingredients to prove the consistency results and to analyse the behaviour of $\Lambda_{t}^{n}(g)$ in the quadratic case. Finally, for our central limit theorems, we use the decompositions obtained in the first part to isolate the leading term, show that all remainder terms are asymptotically negligible, and reduce the problem to a martingale central limit theorem. The technical computations required to approximate the asymptotic variance are deferred to the appendix. As mentioned in the introduction, the appendix may be skipped on a first reading.
		
		Throughout this section, non-random positive constants are denoted by the generic symbol $C>0$, which may change from line to line. We write $a\lesssim_{\theta}b$ whenever $a\leq C_{\theta}b$
		where $C_{\theta}$ depends only on the parameter
		$\theta$. For $i,j=0,1,\ldots$ with $j\geq i$, define 
		\begin{equation}
			\mathcal{P}_{A}^{n}(i,j):=\{(r,s):a(t_{j+1}-s)<r\leq a(t_{j}-s),t_{i-1}<s\leq t_{i}\},\label{Pijsetdef}
		\end{equation}
		where $t_{i}=i\Delta_{n}$ with the convention that $t_{-1}=-\infty$.
		Note that for all $k=0,1,\ldots,$
		\begin{equation}
			\mathrm{Leb}(A_{t_{k}}\backslash\bigcup_{i=0}^{k}\bigcup_{j\geq k}\mathcal{P}_{A}^{n}(i,j))=0.\label{eq:reprA}
		\end{equation}
		Using these sets, we introduce the filtrations: 
		\begin{equation}
			\mathscr{F}_{k}^{n}:=\sigma\left(\{L(\mathcal{P}_{A}^{n}(m,j))\}_{j\geq m}:m=0,\ldots,k\right),\,\,\,\,k\geq0,\label{filtrationovertime}
		\end{equation}
		and 
		\begin{equation}
			\mathscr{G}_{j}^{n}:=\sigma\left(\{L(\mathcal{P}_{A}^{n}(i,m))\}_{0\leq i\leq m}:m=0,\ldots,j\right),\,\,\,j\geq0.\label{filtrationoverspace}
		\end{equation}
		In the remainder of this section we set $\mathfrak{t}_{L}=(\gamma,\sigma,\nu)$,
		the characteristic triplet of $L$, and use the notation $\tau_{n}(s)=[s/\Delta_{n}]\Delta_{n}$
		as well as $u_{n}(s)=\left\lceil s/\Delta_{n}\right\rceil \Delta_{n}$,
		for $s\geq0.$

		\subsection{Preliminary estimates and decompositions}
		
		In this part we derive some estimates and establish some fundamental
		decompositions that will be used in a later stage. We also introduce a number of auxiliary random variables and processes.
		
		The following estimate for the moments of a homogeneous L\'evy basis
		will play a key role in our analysis. For a proof see \cite{LuschPages08}  and \cite{Turner11} (Corollary 1.2.7).
		
		\begin{lemma}\label{moment_estimates_ahat-1}Let $L$ be a centred
			homogeneous L\'evy basis with characteristic triplet $\mathfrak{t}_{L}=(\gamma,\sigma,\nu)$
			with $\mathbb{E}(\rvert L^{\prime}\rvert^{r})<\infty$ for $r\geq1$.
			Then
			\begin{equation}
				\mathbb{E}(\rvert L(B)\rvert^{r})\leq C_{r,\mathfrak{t}_{L}}\mathrm{Leb}(B)^{r/2}\lor\mathrm{Leb}(B),\,\,\forall\,\ensuremath{B\in\mathcal{B}(\mathbb{R}^{2}),}\label{momentineq}
			\end{equation}
			where $C_{q,\mathfrak{t}_{L}}>0$ only depends on $q$ and $\mathfrak{t}_{L}$.
			Furthermore
			\begin{align*}
				\mathbb{E}(L(A)^{2}) & =\mathrm{Leb}(A)(\sigma^{2}+\int x^{2}\nu(dx)).\\
				\mathbb{E}(L(A)^{3}) & =\mathrm{Leb}(A)\int x^{3}\nu(dx).\\
				\mathbb{E}(L(A)^{4}) & =\mathrm{Leb}(A)\left\{ \mathfrak{K}_{4}(L')+3\mathrm{Leb}(A)(\sigma^{2}+\int x^{2}\nu(dx))^{2}\right\} .
			\end{align*}
			Moreover, if $A\subseteq B\subseteq C\subseteq D$, then
			\[
			\mathbb{E}\left[L(A)L(B)L(C)\right]=\mathbb{E}\left[L(A)^{3}\right],
			\]
			and 
			\begin{align*}
				\mathbb{E}[L(A)L(B)L(C)L(D)]= & \mathbb{E}[L(A)^{4}]+\mathbb{E}[L(A)^{2}]\mathbb{E}[L(C\backslash A)^{2}]\\
				& +2\mathbb{E}[L(A)^{2}]\mathbb{E}[L(B\backslash A)^{2}]
			\end{align*}
			
		\end{lemma}
		
		Next, we observe that the estimation error $\hat{a}$ admits the decomposition
		\begin{equation}
			\begin{aligned}\hat{a}(l\Delta_{n})-a(l\Delta_{n}) & =\tilde{a}(l\Delta_{n})-\mathbb{E}(\tilde{a}(l\Delta_{n}))+\mathbb{E}(\tilde{a}(l\Delta_{n}))-a(l\Delta_{n})\\
				& +\frac{1}{n\Delta_{n}}(X_{n\Delta_{n}}-X_{l\Delta_{n}})\frac{1}{n}S_{n}\\
				= & \tilde{a}(l\Delta_{n})-\mathbb{E}(\tilde{a}(l\Delta_{n}))+(X_{n\Delta_{n}}-X_{l\Delta_{n}})\frac{1}{n^{2}\Delta_{n}}S_{n}\\
				& +\frac{(n-l)}{n}\frac{1}{\Delta_{n}}\int_{l\Delta_{n}}^{(l+1)\Delta_{n}}[a(s)-a(l\Delta_{n})]\mathrm{d}s-\frac{l}{n}a(l\Delta_{n}),
			\end{aligned}
			\label{eq:decompahat_atilde_plusbias}
		\end{equation}
		where, $S_{n}=\sum_{k=0}^{n-1}\tilde{X}_{k\Delta_{n}}$ and
		\begin{equation}
			\tilde{a}(l\Delta_{n}):=-\frac{1}{n\Delta_{n}}\sum_{k=l}^{n-1}\tilde{X}_{(k-l)\Delta_{n}}\delta_{k}\tilde{X},\,\,\,l=0,\ldots,n-1,\label{eq:def_a_tilde}
		\end{equation}
		in which $\tilde{X}_{t}=X_{t}-\mathbb{E}(X_{t})$ and $\delta_{k}\tilde{X}=\tilde{X}_{(k+1)\Delta_{n}}-\tilde{X}_{k\Delta_{n}}$.
		To analyse this error, we introduce several auxiliary random variables based on the L\'evy basis $L$. Let $\tilde{L}(\cdot)=L(\cdot)-\mathbb{E}(L(\cdot)$, define 
		\begin{align}
			\chi_{k}:=\tilde{L}(\cup_{i=0}^{k}\mathcal{P}_{A}^{n}(i,k)), & \,\,\,\varrho_{k}:=\tilde{L}(\cup_{j\geq k+1}\mathcal{P}_{A}^{n}(k+1,j));\label{defchirho}\\
			\alpha_{k,j}:=\tilde{L}(\mathcal{P}_{A}^{n}(k,j)),\,\,\, & \gamma_{k}:=\tilde{L}(\cup_{j\geq n+1}\mathcal{P}_{A}^{n}(k+1,j)),\label{eq:defalpha_gamma}
		\end{align}
		and set
		\begin{equation}
			\begin{aligned}\beta_{k,j,l}^{(1)} & =\tilde{L}(\cup_{i=0}^{k-l}\mathcal{P}_{A}^{n}(i,j));\,\, \,  \beta_{k,j,l}^{(2)} =\tilde{L}(\cup_{i=0}^{k-l}\cup_{m=k-l}^{j}\mathcal{P}_{A}^{n}(i,m));\\
				\beta_{k,l}^{(3)} & =\tilde{L}(\cup_{i=0}^{k-l}\cup_{m\geq n+1}\mathcal{P}_{A}^{n}(i,m)).
			\end{aligned}
			\label{defbetasdecompX}
		\end{equation}
		Using these random variables, we obtain the decomposition
		\begin{equation}
			\begin{aligned}
				\tilde{a}(l\Delta_{n})-\mathbb{E}(\tilde{a}(l\Delta_{n}))=&\frac{1}{n\Delta_{n}}\sum_{k=l}^{n-1}\left[\tilde{X}_{(k-l)\Delta_{n}}\chi_{k}-\mathbb{E}(\tilde{X}_{(k-l)\Delta_{n}}\chi_{k})\right]\\
				&-\frac{1}{n\Delta_{n}}\sum_{k=l}^{n-1}\tilde{X}_{(k-l)\Delta_{n}}\varrho_{k}.
			\end{aligned}\label{eq:first_decomp}
		\end{equation}
		Moreover, by using that 
		\begin{equation}
			\tilde{X}_{(k-l_{n})\Delta_{n}}=\beta_{k,j,l}^{(2)}+\sum_{m=j+1}^{n}\beta_{k,m,l}^{(1)}+\beta_{k,l}^{(3)},\label{eq:X_dec_betas}
		\end{equation}
		we further have
		\begin{equation}
			\tilde{a}(l\Delta_{n})-\mathbb{E}(\tilde{a}(l\Delta_{n}))=\frac{1}{n\Delta_{n}}\sum_{j=l}^{n}\varXi_{j,l,n}+\frac{1}{n\Delta_{n}}\sum_{j=l}^{n}R_{j,l,n},\label{eq:second_decomp_mtgdiff}
		\end{equation}
		where $\varXi_{j,l,n}=\sum_{\ell=1}^{4}\xi_{j,l,n}^{(\ell)}$ and
		$R_{j,l,n}=\sum_{\ell=1}^{3}\rho_{j,l,n}^{(\ell)}$ in which 
		\begin{equation}
			\begin{aligned}\xi_{j,l,n}^{(1)} & =\left[\chi_{j}\beta_{j,j,l}^{(2)}-\mathbb{E}(\chi_{j}\beta_{j,j,l}^{(2)})\right];\,\,\xi_{j,l,n}^{(2)}=\sum_{k=l}^{j-1}\chi_{k}\beta_{k,j,l}^{(1)};\,\,	\xi_{j,l,n}^{(3)} =-\sum_{k=l}^{j-1}\alpha_{k+1,j}\beta_{k,j-1,l}^{(2)};\\
				\xi_{j,l,n}^{(4)}&:=-\sum_{k=l}^{j-1}\beta_{k,j,l}^{(1)}\sum_{m=k+1}^{j}\alpha_{k+1,m}=-\sum_{k=l}^{j-1}\beta_{k,j,l}^{(1)}(\alpha_{k+1,j}+\varsigma_{k,j}),
			\end{aligned}
			\label{errorstpos}
		\end{equation}
		as well as
		\begin{equation}
			\begin{aligned}\rho_{j,l,n}^{(1)}= & \chi_{j}\beta_{j,l}^{(3)};\,\,\rho_{j,l,n}^{(2)}=-\tilde{X}_{(j-l)\Delta_{n}}\gamma_{j};\\
				\rho_{j,l,n}^{(3)}= & -\beta_{j,l}^{(3)}\tilde{L}(\cup_{m=j+1}^{n}\mathcal{\mathcal{P}}_{A}^{n}(j+1,m)).
			\end{aligned}
			\label{eq:rest_negli}
		\end{equation}
		The following $r$-moments estimates were proven in \cite{SauriVeraart23}.
		
		\begin{lemma}\label{moment_estimates_chi's}Let $r\geq2$ and suppose
			that $\mathbb{E}(\rvert L^{\prime}\rvert^{2r})<\infty$. Then, 
			\begin{equation}
				\mathbb{E}(\rvert\xi_{j,l,n}^{(1)}\rvert^{r})\lesssim_{r,a,\mathfrak{t}_{L}}\Delta_{n};\,\,\,\,\mathbb{E}(\rvert\xi_{j,l,n}^{(2)}\rvert^{r})\lesssim_{r,a,\mathfrak{t}_{L}}\Delta_{n}+\Delta_{n}\int_{0}^{n\Delta_{n}}a(s)s^{r/2-1}\mathrm{d}s,\label{eq:momentest_chi_1_2}
			\end{equation}
			and 
			\begin{equation}
				\mathbb{E}(\rvert\xi_{j,l,n}^{(3)}\rvert^{r})\lesssim_{r,a,\mathfrak{t}_{L}}\Delta_{n};\,\,\,\,\mathbb{E}(\rvert\xi_{j,l,n}^{(4)}\rvert^{r})\lesssim_{r,a,\mathfrak{t}_{L}}\Delta_{n}+\Delta_{n}\int_{0}^{n\Delta_{n}}a(s)s^{r/2-1}\mathrm{d}s.\label{eq:momentest_chi_3_4}
			\end{equation}
			
		\end{lemma}
		With the previous estimates at hand we obtain the subsequent bounds.
		
		\begin{lemma}\label{moment_estimates_ahat} Let $\tilde{a}$ be as
			in (\ref{eq:def_a_tilde}) and assume that $\mathbb{E}(\rvert L^{\prime}\rvert^{2r})<\infty$
			for some $r\geq2$. Then, 
			\[
			\mathbb{E}(\rvert\tilde{a}(l\Delta_{n})-\mathbb{E}[\tilde{a}(l\Delta_{n})]\rvert^{r})\leq C_{r,a,\mathfrak{t}_{L}}/(n\Delta_{n})^{r/2},
			\]
			where $C_{r,a,\mathfrak{t}_{L}}>0$ is a constant only depending on
			$r$, $a$, and $\mathfrak{t}_{L}$.
			
		\end{lemma}
		
		\begin{proof}Without loss of generality assume that
			$L$ (and hence $X$) has mean $0$. Define
			\[
			E_{n}(l)=(n\Delta_{n})^{r/2}\mathbb{E}(\rvert\tilde{a}(l\Delta_{n})-\mathbb{E}[\tilde{a}(l\Delta_{n})]\rvert^{r}),\,\,l=0,\ldots,n-1.
			\]
			It suffices to show that $E_n(l)$ is bounded uniformly in $n$ and $l$. By (\ref{eq:first_decomp}) and (\ref{eq:X_dec_betas}) we
			get that 
			\[
			E_{n}(l)\leq4^{r-1}\sum_{\ell=1}^{4}E_{n}^{\ell}(l),
			\]
			where 
			\begin{align*}
				E_{n}^{1}(l) & =\frac{1}{(n\Delta_{n})^{r/2}}\mathbb{E}(\rvert\sum_{j=l}^{n}\xi_{j,l,n}^{(1)}\rvert^{r});\,\,\,E_{n}^{2}(l)=\frac{1}{(n\Delta_{n})^{r/2}}\mathbb{E}(\rvert\sum_{j=l}^{n}\xi_{j,l,n}^{(2)}\rvert^{r});\\
				E_{n}^{3}(l) & =\frac{1}{(n\Delta_{n})^{r/2}}\mathbb{E}(\rvert\sum_{j=l}^{n}\rho_{j,l,n}^{(1)}\rvert^{r});\,\,E_{n}^{4}(l)=\frac{1}{(n\Delta_{n})^{r/2}}\mathbb{E}(\rvert\sum_{j=l}^{n}X_{(j-l)\Delta_{n}}\varrho_{j}\rvert^{r}).
			\end{align*}
			Note that for every $j\geq l\geq0$, $(\xi_{j,l,n}^{(1)},\xi_{j,l,n}^{(2)})$,
			$\rho_{j,l,n}^{(1)}$, and $X_{(j-l)\Delta_{n}}\varrho_{j}$ are martingale
			differences with respect to  $\mathscr{G}_{j}^{n}$, $\ensuremath{\mathscr{G}_{j}^{n}\lor\sigma(\beta_{j^{\prime},l^{\prime}}^{(3)})_{0\leq l^{\prime}\leq j^{\prime}\leq j}},$
			and $\mathscr{F}_{j}^{n}$, respectively. Hence Rosenthal's inequality (see e.g. Theorem 2.12 in \cite{HallHeyde80}) yields
			\begin{equation}
				E_{n}^{4}(l)  \lesssim_{r}\frac{1}{(n\Delta_{n})^{r/2}}\sum_{j=l}^{n}\mathbb{E}(\rvert X_{(j-l)\Delta_{n}}\varrho_{j}\rvert^{r})+\mathbb{E}\left[\left|\frac{1}{n\Delta_{n}}\sum_{j=l}^{n}(X_{(j-l)\Delta_{n}})^{2}\mathbb{E}(\varrho_{j}^{2})\right|^{r/2}\right]
				\label{eq:rosental_Jensen_momentineq}
			\end{equation}
			Using the stationarity of $X$, Jensen's inequality, and \eqref{momentineq} we obtain that 
			\[
			E_n^4(l)
			\lesssim_{r,a,\mathfrak t_L}
			\frac{\mathbb{E}(|X_0|^r)}{(n\Delta_n)^{r/2-1}}
			+
			\mathbb{E}(|X_0|^r) .
			\]
			Similarly, we deduce that
			\[
			E_{n}^{4}(l) \lesssim_{r,a,\mathfrak{t}_{L}}\frac{1}{(n\Delta_{n})^{r/2-1}}+1.
			\]
			For $i=1,2$, a further application of Rosenthal's inequality results in 
			\[
			E_n^i(l)
			\lesssim_r
			\frac{1}{(n\Delta_n)^{r/2}}
			\sum_{j=l}^n
			\mathbb{E}(|\xi^{(i)}_{j,n}|^r)
			+
			E_n^{i\prime}(l).
			\]
			where
			\[
			E_{n}^{i}(l)^{\prime}=\mathbb{E}\left[\left|\frac{1}{n\Delta_{n}}\sum_{j=l}^{n}\mathbb{E}((\xi_{j,n}^{(i)})^{2}\mid\mathscr{G}_{j-1}^{n})\right|^{r/2}\right],
			\]
			Lemma \ref{moment_estimates_chi's} gives
			\[\mathbb{E}(\rvert\xi_{j,n}^{(1)}\rvert^{r})\lesssim_{r,a,\mathfrak{t}_{L}}\Delta_{n},\qquad \mathbb{E}(\rvert\xi_{j,n}^{(2)}\rvert^{r}) \lesssim_{r,a,\mathfrak{t}_{L}}\Delta_{n}+\frac{(\Delta_{n}n)^{r/2}}{n}.\]
			Thus, we are left to bound $E_{n}^{i}(l)^{\prime}$. From Lemma 3 in \cite{SauriVeraart23}
			and \eqref{momentineq} we get that 
			\begin{align*}
				\sum_{j=l}^{n}\mathbb{E}((\xi_{j,n}^{(1)})^{2}\mid\mathscr{G}_{j-1}^{n}) & \lesssim_{a,\mathfrak{t}_{L},r}n\Delta_{n}+\Delta_{n}\sum_{j=l+1}^{n}\left\{ (\beta_{j,j-1,l}^{(2)})^{2}-\mathbb{E}[(\beta_{j,j-1,l}^{(2)})^{2}]\right\} ;\\
				\sum_{j=l}^{n}\mathbb{E}((\xi_{j,n}^{(2)})^{2}\mid\mathscr{G}_{j-1}^{n}) & \lesssim_{a,\mathfrak{t}_{L},r}n\Delta_{n}+\sum_{k=l+1}^{n}\chi_{k}\nu_{k}^{(1)}\\
				& +\sum_{k=l}^{n}[\chi_{k}^{2}-\mathbb{E}(\chi_{k}^{2})]\mathrm{Leb}(\cup_{i=0}^{k-l}\cup_{j=k+1}^{n}\mathcal{P}_{A}^{n}(i,j)),
			\end{align*}
			where $\nu_{k}^{(1)}:=\sum_{k^{\prime}=l}^{k-1}\chi_{k^{\prime}}\mathrm{Leb}(\cup_{i=0}^{k^{\prime}-l}\cup_{j=k+1}^{n-1}\mathcal{P}_{A}^{n}(i,j))$.
			This, together with \eqref{momentineq}, and Jensen's inequality shows that $E_{n}^{1}(l)^{\prime}$ is uniformly bounded.  
			Now, if $r>4$ a double application of the Rosenthal's inequality yields
			\begin{align*}
				E_{n}^{2}(l)^{\prime}\lesssim_{a,\mathfrak{t}_{L},r} &1+\frac{1}{(n\Delta_{n})^{r/2}}\sum_{k=l}^{n}\mathbb{E}(\rvert\chi_{k}^{2}-\mathbb{E}(\chi_{k}^{2})\rvert^{r/2})\\
				&+\frac{1}{(n\Delta_{n})^{r/2}}\left|\sum_{k=l}^{n}\mathbb{E}(\rvert\chi_{k}^{2}-\mathbb{E}(\chi_{k}^{2})\rvert^{2})\right|^{r/4}\\
				&+		\frac{1}{(n\Delta_{n})^{r/2}}\sum_{k=l}^{n}\mathbb{E}(\rvert\chi_{k}\nu_{k}^{(1)}\rvert^{r/2})+\frac{1}{(n\Delta_{n})^{r/4}}\mathbb{E}\left[\left|\frac{1}{n}\sum_{k=l}^{n}(\nu_{k}^{(1)})^{2}\right|^{r/4}\right].
			\end{align*}
			Each term on the right-hand side can be shown to be bounded exactly as above.
			Finally, if $r\leq4$ we use the inequality $\mathbb{E}(\rvert Y\rvert^{r/2})\leq\mathbb{E}(\rvert Y\rvert^{2})^{r/4}$ along with the von Bahr-Esseen inequality (see \cite[9.3.b]{LinBai10}) to deduce that 
			\[ 	E_{n}^{2}(l)^{\prime}\lesssim  1+\frac{1}{(n\Delta_{n})^{r/2}}\left|\sum_{k=l}^{n}\mathbb{E}[(\chi_{k}\nu_{k}^{(1)})^{2}]\right|^{r/4}+\frac{1}{(n\Delta_{n})^{r/2}}\sum_{k=l}^{n}\mathbb{E}(\rvert\chi_{k}^{2}-\mathbb{E}(\chi_{k}^{2})\rvert^{r/2}), \]
			and the expression above is uniformly bounded due to \eqref{momentineq}.
		\end{proof}
		
		We conclude this part with an estimate that will provide a way to
		control the approximation error in our statistics.
		
		\begin{lemma}\label{lemma_error_term.} Let $r\geq1,v\geq0$ and
			set 
			\[
			A^{n}:=\Delta_{n}\sum_{l=0}^{m_{n}-1}a(l\Delta_{n})^{v}\rvert\hat{a}(l\Delta_{n})-a(l\Delta_{n})\rvert^{r},\,\,m_{n}\in\mathbb{N},m_{n}\leq n.
			\]
			Suppose that $\mathbb{E}(\rvert L^{\prime}\rvert^{2r\lor4})<\infty$
			and that $m_{n}\Delta_n\rightarrow\jmath\in(0,+\infty].$ Then, 
			\begin{equation}
				A^{n}=\mathrm{O}_{\mathbb{P}}((n\Delta_{n})^{-r/2}\int_{0}^{m_{n}\Delta_{n}}a(s)^{v}\mathrm{d}s)+\mathrm{o}(1).\label{eq:error_term_taylo1}
			\end{equation}
			If Assumption \ref{as:trawl} holds, we further have that 
			\begin{equation}
				A^{n}=\mathrm{O}_{\mathbb{P}}((n\Delta_{n})^{-r/2}\int_{0}^{m_{n}\Delta_{n}}a(s)^{v}\mathrm{d}s)+\mathrm{O}(\Delta_{n}^{r}).\label{eq:eq:error_term_taylo1}
			\end{equation}

		\end{lemma}
		
		\begin{proof}From (\ref{eq:decompahat_atilde_plusbias})
			\begin{equation}
				\begin{aligned}A^{n}\lesssim_{r} & \Delta_{n}\sum_{l=0}^{m_{n}-1}a(l\Delta_{n})^{v}\rvert\tilde{a}(l\Delta_{n})-\mathbb{E}(\tilde{a}(l\Delta_{n}))\rvert^{r}\\
					& +\rvert\frac{1}{n}S_{n}\rvert^{r}\frac{1}{(n\Delta_{n})^{r}}\Delta_{n}\sum_{l=0}^{m_{n}-1}a(l\Delta_{n})^{v}\rvert X_{n\Delta_{n}}-X_{l\Delta_{n}}\rvert^{r}\\
					& +\frac{1}{(n\Delta_{n})^{r}}\int_{0}^{m_{n}\Delta_{n}}\rvert\tau_{n}(s)a(\tau_{n}(s))^{1+v/r}\rvert^{r}\mathrm{d}s\\
					& +\Delta_{n}\sum_{l=0}^{m_{n}-1}a(l\Delta_{n})^{v}\left(\frac{1}{\Delta_{n}}\int_{l\Delta_{n}}^{(l+1)\Delta_{n}}[a(l\Delta_{n})-a(s)]\mathrm{d}s\right)^{r}.
				\end{aligned}
				\label{eq:error_taylor}
			\end{equation}
			By Lemma \ref{moment_estimates_ahat} 
			\[
			\Delta_{n}\sum_{l=0}^{m_{n}-1}a(l\Delta_{n})^{v}\mathbb{E}(\rvert\tilde{a}(l\Delta_{n})-\mathbb{E}(\tilde{a}(l\Delta_{n}))\rvert^{r})\lesssim\frac{1}{(n\Delta_{n})^{r/2}}\int_{0}^{m_{n}\Delta_{n}}a(\tau_{n}(s))^{v}\mathrm{d}s.
			\]
			Using the stationarity of $X$ we also obtain
			\[
			\frac{1}{(n\Delta_{n})^{r}}\Delta_{n}\sum_{l=0}^{m_{n}-1}a(l\Delta_{n})^{v}\mathbb{E}(\rvert X_{n\Delta_{n}}-X_{l\Delta_{n}}\rvert^{r})\lesssim\frac{1}{(n\Delta_{n})^{r}}\int_{0}^{m_{n}\Delta_{n}}a(\tau_{n}(s))^{v}\mathrm{d}s.
			\]
			Furthermore, by Jensen's inequality and the monotonicity of $a$, we
			deduce that the last summand in \eqref{eq:error_taylor} is bounded by 
			\begin{equation}
				\int_{0}^{m_{n}\Delta_{n}}a(\tau_{n}(s))^{v}[a(\tau_{n}(s))-a(s)]^{r}\mathrm{d}s\lesssim_{a,r} \Delta_{n}+\int_{0}^{\Delta_{n}}a(s)^{v+1}\mathrm{d}s \lesssim_{a} \Delta_{n}.
				\label{eq:error_last_sum}
			\end{equation}
			Moreover, since $a$ is integrable and of bounded variation, it holds that $xa(x)\rightarrow0$ as $x\rightarrow\infty$ (see the proof of Lemma 5.4 in \cite{PakkaPassOSauriVer20}). Consequently
			\[
			\frac{1}{(n\Delta_{n})^{r}}\int_{0}^{m_{n}\Delta_{n}}\rvert\tau_{n}(s)a(\tau_{n}(s))^{1+v/r}\rvert^{r}\mathrm{d}s\lesssim_{r,a}\frac{1}{(n\Delta_{n})^{r}}\int_{0}^{m_{n}\Delta_{n}}a(\tau_{n}(s))^{v}\mathrm{d}s.
			\]
			In view that 
			\[
			\int_{0}^{m_{n}\Delta_{n}}a(\tau_{n}(s))^{\upsilon}\mathrm{d}s\leq a(0)^{\upsilon}\Delta_{n}+\int_{0}^{m_{n}\Delta_{n}}a(s)^{v}\mathrm{d}s,
			\]
			relation (\ref{eq:error_term_taylo1})follows. Finally, if Assumption \ref{as:trawl}
			is satisfied then there is $y_{0}>0$ such that 
			\begin{equation}\label{eq:assumption_ac_estimate}
				\frac{1}{\Delta_{n}}\int_{l\Delta_{n}}^{(l+1)\Delta_{n}}[a(l\Delta_{n})-a(s)]\mathrm{d}s\leq\mathbf{1}_{l\Delta_{n}\leq y_{0}}\Delta_{n}+\mathbf{1}_{l\Delta_{n}>y_{0}}\Delta_{n}(l\Delta_{n})^{-\alpha}.
			\end{equation}
			Thus, the estimate in \eqref{eq:error_last_sum} can be replaced by
			
			\[ \Delta_{n}^{r}(\mathrm{O}(1)+\int_{y_{0}}^{\infty}s^{-r\alpha}\mathrm{d}s)\leq C\Delta_{n}^{r}, \]
			which gives (\ref{eq:eq:error_term_taylo1}).\end{proof}
		
		\subsection{Proof of Theorems  \ref{consitency_Phin}  - \ref{consitency_Lambdan_bar}}
		
		We start by showing the consistency of $\Psi^{n}(g)$. The proof follows
		the idea of the proof of Theorem 9.4.1 in \cite{JacProt11}. For the rest of this section
		$T>0$ is fixed and choose $n$ large enough such that $(n-1)\Delta_{n}>T.$
		
		\begin{proof}[Proof of Theorem \ref{consitency_Phin}] Without loss assume that $g\geq0$  (otherwise analyze
			$g^{+}$ and $g^{-}$ separately). Thus, we only need to check that
			for every $0\leq t\leq T$, $\Psi_{t}^{n}(g)\overset{\mathbb{P}}{\rightarrow}\Psi_{t}(g)$.
			Suppose first that $g$ is bounded. Then 
			\[
			\mathbb{E}(\rvert\Psi_{t}^{n}(g)-\Psi_{t}(g)\rvert)\leq\Delta_{n}+\int_{0}^{t}\mathbb{E}(\rvert g(\hat{a}(s))-g(a(s))\rvert)\mathrm{d}s.
			\]
			Theorem 1 in \cite{SauriVeraart23} along with the boundedness of
			$g$ allow us to apply the Dominated Convergence twice to conclude
			that the right-hand side of the previous inequality goes to $0$ as
			$n\rightarrow\infty$. Suppose now that $g$ is of polynomial growth
			of order $q>0$ (when $q=0$ $g$ is bounded). Let $\psi:\mathbb{R}^{+}\rightarrow[0,1]$
			be a function of class $\mathcal{C}^{\infty}$ such that $\mathbf{1}_{[1,\infty)}\leq\psi\leq\mathbf{1}_{[1/2,\infty)}$,
			and for every $m\in\mathbb{N}$, set $\psi_{m}(x):=1-\psi(\rvert x\rvert/m)$
			as well as $\tilde{\psi}_{m}(x)=\psi(\rvert x\rvert/m).$ Write
			\[
			\Psi_{t}^{n}(g)=\Psi_{t}^{n}(g_{m})+\Psi_{t}^{n}(\tilde{g}_{m}),
			\]
			where $g_{m}=g\psi_{m}$ and $\tilde{g}_{m}=g\tilde{\psi}_{m}$. Since $g_{m}$ is continuous and bounded, the first part
			of the proof gives that $\Psi_{t}^{n}(g_{m})\overset{\mathbb{P}}{\rightarrow}\int_{0}^{t}g_{m}(a(s))\mathrm{d}s$.
			In view that $\int_{0}^{t}g_{m}(a(s))\mathrm{d}s\rightarrow\Psi_{t}(g),$
			as $m\rightarrow\infty$, we are left to show that for every $\eta>0$
			\begin{equation}
				\lim_{m\rightarrow\infty}\limsup_{n\rightarrow\infty}\mathbb{P}(\rvert\Psi_{t}^{n}(\tilde{g}_{m})\rvert>\eta)=0.\label{eq:negligiapprox}
			\end{equation}
			Since $g$ has polynomial growth, it holds that $\rvert\tilde{g}_{m}(x)\rvert\leq C_{q}\rvert x\rvert^{q}\mathbf{1}_{\rvert x\rvert>m}$. Hence, (\ref{eq:negligiapprox})
			follows if 
			\begin{equation}
				\lim_{m\rightarrow\infty}\limsup_{n\rightarrow\infty}\mathbb{P}(\rvert Y_{n}^{m}(t)\rvert>\eta)=0,\label{eq:negligiapprox-1}
			\end{equation}
			where 
			\[
			Y_{n}^{m}(t)=\Delta_{n}\sum_{l=0}^{[t/\Delta_{n}]-1}\rvert\hat{a}(l\Delta_{n})\rvert^{q}\mathbf{1}_{\{\rvert\hat{a}(l\Delta_{n})\rvert>m\}}.
			\]
			From \eqref{eq:decompahat_atilde_plusbias} 
			\[
			Y_{n}^{m}(t)  \lesssim_q\Delta_{n}\sum_{l=0}^{[t/\Delta_{n}]-1}\rvert\tilde{a}(l\Delta_{n})\rvert^{q}\mathbf{1}_{\{\rvert\hat{a}(l\Delta_{n})\rvert>m\}}+\rvert\frac{1}{n}S_{n}\rvert^{q}\frac{1}{(n\Delta_{n})^{q}}\Delta_{n}\sum_{l=0}^{[t/\Delta_{n}]-1}\rvert X_{n\Delta_{n}}-X_{l\Delta_{n}}\rvert^{q}.
			\]
			Denote by $Y_{n}^{m,1}$ and $Y_{n}^{2}$ the last two terms. As in the proof of Lemma \ref{lemma_error_term.}, $Y_{n}^{2}\overset{\mathbb{P}}{\rightarrow}0$.
			Thus, 
			\begin{equation}
				\limsup_{n\rightarrow\infty}\mathbb{P}(\rvert Y_{n}^{m}(t)\rvert>\eta)\leq\limsup_{n\rightarrow\infty}\mathbb{P}(C_{q}\rvert Y_{n}^{m,1}\rvert>\eta/2).\label{eq:limsup_ineq}
			\end{equation}
			Moreover, Lemma \ref{moment_estimates_ahat} and Jensen's inequality
			guarantee that 
			\begin{align*}
				\mathbb{E}(\rvert Y_{n}^{m,1}\rvert)\lesssim_{q} & \Delta_{n}\sum_{l=0}^{[t/\Delta_{n}]-1}\mathbb{E}(\rvert\tilde{a}(l\Delta_{n})-\mathbb{E}[\tilde{a}(l\Delta_{n})]\rvert^{q}\mathbf{1}_{\{\rvert\hat{a}(l\Delta_{n})\rvert>m\}})\\
				& +\Delta_{n}\sum_{l=0}^{[t/\Delta_{n}]-1}\mathbb{E}[\tilde{a}(l\Delta_{n})]^{q}\mathbb{P}(\rvert\hat{a}(l\Delta_{n})\rvert>m)\\
				\lesssim & \frac{t}{(n\Delta_{n})^{q/2}}+\sum_{l=0}^{[t/\Delta_{n}]-1}\int_{l\Delta_{n}}^{(l+1)\Delta_{n}}a(s)^{q}\mathrm{d}s\mathbb{P}(\rvert\hat{a}(l\Delta_{n})\rvert>m).
			\end{align*}
			By Markov's inequality 
			\[
			\mathbb{P}(\rvert\hat{a}(l\Delta_{n})\rvert>m)\leq  \frac{1}{m^q}\mathbb{E}(\rvert\tilde{a}(l\Delta_{n})\rvert^q).
			\]
			Using \eqref{eq:decompahat_atilde_plusbias} together with Lemma \ref{moment_estimates_ahat}, we obtain that $\mathbb{E}(\rvert\tilde{a}(l\Delta_{n})\rvert^q)$ is bounded uniformly in $n$ and $l$. 
			Therefore, 
			\[
			\mathbb{E}(\rvert Y_{n}^{m,1}\rvert)\lesssim\frac{1}{(n\Delta_{n})^{q/2}}+\frac{1}{m^q}\int_{0}^{t}a(s)^{q}\mathrm{d}s.
			\]
			Relation (\ref{eq:negligiapprox-1}) now follows from (\ref{eq:limsup_ineq}),
			Markov's inequality, and the previous estimate.\end{proof}
		
		Our next goal is to unify the proof of Theorems \ref{consitency_Lambdan}
		and \ref{consitency_Lambdan_bar}. To this end, we write  $\mathcal{S}^{n}(g)$
		and $\mathcal{S}(g)$ to denote any of the pairs $(\Psi^n(g),\Psi(g))$, $(\Lambda^n(g),\Lambda(g))$, or $(\bar{\Lambda}^n(g),\bar{\Lambda}(g))$.
		With this notation, 
		\begin{equation}
			\mathcal{S}_{t}^{n}(g)=\Delta_{n}\sum_{l=0}^{n-1}c_{l}^{n}(t)g(\hat{a}(l\Delta_{n})),\,\,\,0\leq t\leq T,\label{eq:simplifiedrepfunctionals}
		\end{equation}
		where 
		\begin{equation}
			c_{l}^{n}(t)=\begin{cases}
				\mathbf{1}_{0\leq l\leq[t/\Delta_{n}]-1} & \text{if }\mathcal{S}^{n}(g)=\Psi^{n}(g);\\
				\mathbf{1}_{[t/\Delta_{n}]\leq l\leq n-1} & \text{if }\mathcal{S}^{n}(g)=\Lambda^{n}(g);\\
				\mathbf{1}_{[t/\Delta_{n}]\leq l\leq N_{n}-1} & \text{if }\mathcal{S}^{n}(g)=\bar{\Lambda}^{n}(g).
			\end{cases}\label{eq:c_sumandsdef}
		\end{equation}
		If $g\in\mathcal{C}^1$, the Mean Value Theorem yields
		\begin{equation}
			\mathcal{S}_{t}^{n}(g)-\mathcal{S}_{t}(g)=U_{t}^{n}+\mathcal{R}_{t}^{n},\quad0\leq t\leq T,\label{Decomp_UR_consistency}  
		\end{equation}
		where 
		\begin{equation}
			U_{t}^{n}=\Delta_{n}\sum_{l=0}^{n-1}c_{l}^{n}(t)g^{\prime}(\theta_{l})\left[\hat{a}(l\Delta_{n})-a(l\Delta_{n})\right],\label{eq:errorU_consistency}
		\end{equation}
		for some $\hat{a}(l\Delta_{n})\land a(l\Delta_{n})\leq\theta_{l}\leq\hat{a}(l\Delta_{n})\lor a(l\Delta_{n})$, and
		\begin{equation}
			\mathcal{R}_{t}^{n}=\Delta_{n}\sum_{l=0}^{n-1}c_{l}^{n}(t)g(a(l\Delta_{n}))-\mathcal{S}_{t}(g).\label{eq:errorR_consistency}
		\end{equation}
		Hence Theorems \ref{consitency_Lambdan} and \ref{consitency_Lambdan_bar} follow once the processes $U^n$ and $\mathcal{R}^n$ are shown to be asymptotically negligible.  In what follows we write $m_{n}=n$ or $m_{n}=N_{n}$, depending on whether $c_{l}^{n}(t)=\mathbf{1}_{[t/\Delta_{n}]\leq l\leq n-1}$
		or $c_{l}^{n}(t)=\mathbf{1}_{[t/\Delta_{n}]\leq l\leq N_{n}-1}$, respectively.
		
		\begin{proof}[Proof of Theorems  \ref{consitency_Lambdan} and \ref{consitency_Lambdan_bar}]
			Let $g\in\mathfrak{C}_{p,q}^{1}$ with $p,q\geq1$ and recall the decomposition \eqref{Decomp_UR_consistency}. From the definition of $\mathcal{R}_{t}^{n}$ we get 
			\begin{equation}	
				\mathcal{R}_{t}^{n}=\int_{t}^{m_{n}\Delta_{n}}[g(a(\tau_{n}(s)))-g(a(s))]\mathrm{d}s+\int_{[t/\Delta_{n}]\Delta_{n}}^{t}g(a(\tau_{n}(s)))\mathrm{d}s-\int_{m_{n}\Delta_{n}}^{\infty}g(a(s))\mathrm{d}s.
				\label{eq:dec_Rn}
			\end{equation}
			Since $g(a(\cdot))$	is continuous and satisfies $\rvert g(a(s))\rvert\leq C_{a(0)}a(s)$,  (see \eqref{eq:boundsforg_gprime}), 
			we have that, uniformly on $t\in[0,T]$,
			\[
			\int_{[t/\Delta_{n}]\Delta_{n}}^{t}g(a(\tau_{n}(s)))\mathrm{d}s-\int_{m_{n}\Delta_{n}}^{\infty}g(a(s))ds\rightarrow0.
			\]
			Similarly, we deduce 
			\[
			\int_{t}^{m_{n}\Delta_{n}}\rvert g(a(\tau_{n}(s)))-g(a(s))\rvert\mathrm{d}s \lesssim_{a}\int_{0}^{m_{n}\Delta_{n}}[a(\tau_{n}(s))-a(s)]\mathrm{d}s
			\lesssim_{a}\Delta_{n},
			\]
			where in the last inequality we further used \eqref{eq:error_last_sum}.
			Hence $\mathcal{R}^{n}\overset{u.c.p}{\rightarrow}0$. Let us now turn our attention to $U^n$. Since $\rvert g^{\prime}(x)\rvert\lesssim\rvert x\rvert^{p}\mathbf{1}_{\rvert x\rvert\leq1}+\rvert x\rvert^{q}\mathbf{1}_{\rvert x\rvert>1}$,
			we obtain
			\begin{equation}
				\begin{aligned}\sup_{t\in[0,T]}\rvert U_{t}^{n}\rvert\lesssim & \Delta_{n}\sum_{l=0}^{m_{n}-1}\rvert\hat{a}(l\Delta_{n})-a(l\Delta_{n})\rvert^{p\land q+1}+\Delta_{n}\sum_{l=0}^{m_{n}-1}\rvert\hat{a}(l\Delta_{n})-a(l\Delta_{n})\rvert^{p\lor q+1}\\
					& +\Delta_{n}\sum_{l=0}^{m_{n}-1}a(l\Delta_{n})^{p\land q}\rvert\hat{a}(l\Delta_{n})-a(l\Delta_{n})\rvert
				\end{aligned}
				\label{eq:error_Ubound_unbounded}
			\end{equation}
			Applying Lemma \ref{lemma_error_term.} yields
			\[
			\sup_{t\in[0,T]}|U_t^n|
			=\mathrm{O}_{\mathbb{P}}(\frac{m_{n}}{n}(n\Delta_{n})^{\frac{1-p\land q}{2}})+\mathrm{O}_{\mathbb{P}}(1/\sqrt{n\Delta_{n}})+\mathrm{o}(1).
			\]
			If $p,q>1$ and $m_{n}=n$ (as in Theorem
			\ref{consitency_Lambdan}), the right-hand side of the previous estimate converges to zero. If $p\land q=1$
			and $m_{n}=N_{n}$ satisfies \eqref{eq:window_assumption} (as in Theorem \ref{consitency_Lambdan_bar}), the bound in \eqref{eq:error_Ubound_unbounded} becomes $\mathrm{O}_{\mathbb{P}}(\frac{N_{n}}{n})$, which also converges to zero by assumption. Thus $U^n$ and $\mathcal R^n$ are asymptotically negligible, completing the proof.\end{proof}
		
		We conclude this section by presenting a proof for Theorem \ref{bias_quadratic}.
		
		\begin{proof}[Proof of Theorem  \ref{bias_quadratic}]Arguing as in
			the preceding proof, we have that
			\[
			\Delta_{n}\sum_{l=0}^{n-1}\hat{a}(l\Delta_{n})^{2}-\int_{0}^{\infty}a(s)^{2}\mathrm{d}s=B_{n}+\mathrm{o}_{\mathbb{P}}(1)
			\]
			where $B_{n}=\Delta_{n}\sum_{l=0}^{n-1}\left[\hat{a}(l\Delta_{n})-a(l\Delta_{n})\right]^{2}$.
			Let 
			\[
			b_{l}^{n}:=\hat{a}(l\Delta_{n})-a(l\Delta_{n})-\left(\tilde{a}(l\Delta_{n})-\mathbb{E}(\tilde{a}(l\Delta_{n}))\right).
			\]
			Arguing exactly as in the proof of Lemma \ref{lemma_error_term.},
			we have that 
			\[
			\mathcal{E}_{n}:=\Delta_{n}\sum_{l=0}^{n-1}(b_{l}^{n})^{2}\overset{\mathbb{P}}{\rightarrow}0.
			\]
			Thus, by setting $B_{n}^{\prime}:=\Delta_{n}\sum_{l=0}^{n-1}\left[\tilde{a}(l\Delta_{n})-\mathbb{E}(\tilde{a}(l\Delta_{n}))\right]^{2}$
			and applying the reverse triangle inequality we deduce that 
			\[
			\left|\sqrt{B_{n}}-\sqrt{B_{n}^{\prime}}\right|\leq\sqrt{\mathcal{E}_{n}}\overset{\mathbb{P}}{\rightarrow}0.
			\]
			Therefore, it suffices to prove that $B_{n}^{\prime}\overset{\mathbb{P}}{\rightarrow}B$.
			Now, using the random variables introduced in \eqref{errorstpos} and \eqref{eq:second_decomp_mtgdiff}, we
			set
			\[
			B_{n}^{\prime\prime}:=\Delta_{n}\sum_{l=0}^{n-1}\left[\frac{1}{n\Delta_{n}}\sum_{j=l}^{n}(\xi_{j,l,n}^{\prime(1)}+\xi_{j,l,n}^{(3)})\right]^{2},
			\]
			where $\xi_{j,l,n}^{\prime(1)}=\chi_{j}\beta_{j,j-1,l}^{(2)}$, and let
			\[
			\mathcal{E}_{n}^{\prime}:=\Delta_{n}\sum_{l=1}^{n-1}\left[\frac{1}{n\Delta_{n}}\sum_{j=l}^{n-1}\left(R_{j,l,n}+\xi_{j,l,n}^{(1)}+\xi_{j,l,n}^{(2)}+\xi_{j,l,n}^{(4)}\right)\right]^{2}.
			\]
			By \eqref{eq:second_decomp_mtgdiff}, \eqref{eq:cond_variance} in Appendix \ref{Appendix1}, 
			and the definition of $R_{j,l,n}$, we get that $\mathbb{E}(\mathcal{E}_{n}^{\prime})\rightarrow0$.
			This shows, just as above, that $B_{n}^{\prime\prime}$ and $B_{n}^{\prime}$
			have the limit (and therefore the same limit as $B_{n}$). Arguing as in Lemma \ref{lemmaavar_approx}  in Appendix \ref{Appendix1}, we conclude that 
			\begin{align*}
				B_{n}^{\prime\prime} & :=\frac{\Delta_{n}}{(n\Delta_{n})^{2}}\sum_{j=1}^{n-1}\sum_{l=1}^{j}\left[(\xi_{j,l,n}^{\prime(1)})^{2}+(\xi_{j,l,n}^{(3)})^{2}+2\xi_{j,l,n}^{\prime(1)}\xi_{j,l,n}^{(3)}\right]\\
				& +\frac{2\Delta_{n}}{(n\Delta_{n})^{2}}\sum_{j=2}^{n-1}\sum_{l=1}^{j-1}(\xi_{j,l,n}^{\prime(1)}+\xi_{j,l,n}^{(3)})\sum_{j^{\prime}=l}^{j-1}(\xi_{j^{\prime},l,n}^{\prime(1)}+\xi_{j^{\prime},l,n}^{(3)})\\
				& =\frac{\Delta_{n}}{(n\Delta_{n})^{2}}\sum_{j=1}^{n-1}\sum_{l=1}^{j}\mathbb{E}\left[(\xi_{j,l,n}^{\prime(1)})^{2}+(\xi_{j,l,n}^{(3)})^{2}+2\xi_{j,l,n}^{\prime(1)}\xi_{j,l,n}^{(3)}\right]+\mathrm{o}_{\mathbb{P}}(1).
			\end{align*}
			The stated convergence now follows by invoking once again \eqref{eq:cond_variance} in Appendix \ref{Appendix1}.\end{proof}
		
		\subsection{Proof of Theorems \ref{CLT_Phin} and \ref{CLT_Lambdan_hatLambdan}}
		We begin by outlining the proof of the central limit theorems. Starting from a decomposition of the estimation error into a leading term and several remainder terms, we first show that the remainder terms are asymptotically negligible. The leading term is then rewritten as a sum of martingale differences plus a negligible error term. The proof proceeds by establishing tightness, reducing the convergence of the finite-dimensional to a martingale central limit theorem via the Cram\'er--Wold device, and verifying the convergence of the conditional variance and the Lyapunov condition. 
		
		Similarly to the previous subsection, fix $T>0$ and choose $n$
		large enough such that $(m_n-1)\Delta_{n}>T$, where, as before
		$m_{n}=n$ or $m_{n}=N_{n}$, depending whether $c_{l}^{n}(t)=\mathbf{1}_{[t/\Delta_{n}]\leq l\leq n-1}$
		or $c_{l}^{n}(t)=\mathbf{1}_{[t/\Delta_{n}]\leq l\leq N_{n}-1}$,
		respectively. 
		
		Once again, we unify the proofs by means of (\ref{eq:simplifiedrepfunctionals}).
		Specifically, for $g\in\mathcal{C}^{2}$, define
		\[
		Z_{t}^{n}:=\sqrt{n\Delta_{n}}\left(\mathcal{S}_{t}^{n}(g)-\mathcal{S}_{t}(g)\right).
		\]
		Then
		\[
		Z_t^n
		=
		V_t^n+\sum_{\ell=1}^5 U_t^{n,\ell},
		\]
		where (see (\ref{eq:c_sumandsdef}) and (\ref{eq:errorR_consistency}))
		\begin{equation}
			V_{t}^{n}:=\sqrt{n\Delta_{n}}\Delta_{n}\sum_{l=0}^{n-1}c_{l}^{n}(t)g^{\prime}(a(l\Delta_{n}))\left[\tilde{a}(l\Delta_{n})-\mathbb{E}(\tilde{a}(l\Delta_{n}))\right],\label{eq:leading_V_1}
		\end{equation}
		is the leading term, and the remainder terms are given by
		\begin{align*}
			U_{t}^{n,1}= & \frac{\sqrt{n\Delta_{n}}}{2}\Delta_{n}\sum_{l=0}^{n-1}c_{l}^{n}(t)g^{\prime\prime}(\theta_{l})\left[\hat{a}(l\Delta_{n})-a(l\Delta_{n})\right]^{2};\\
			U_{t}^{n,2}= & \sqrt{n\Delta_{n}}\sum_{l=0}^{n-1}c_{l}^{n}(t)\frac{(n-l)}{n}\int_{l\Delta_{n}}^{(l+1)\Delta_{n}}g^{\prime}(a(l\Delta_{n}))[a(s)-a(l\Delta_{n})]\mathrm{d}s;\\
			U_{t}^{n,3}= & -\sqrt{\frac{\Delta_{n}}{n}}\sum_{l=0}^{n-1}c_{l}^{n}(t)g^{\prime}(a(l\Delta_{n}))a(l\Delta_{n})l\Delta_{n};\\
			U_{t}^{n,4}= & \frac{1}{n}S_{n}\sqrt{\frac{\Delta_{n}}{n}}\sum_{l=0}^{n-1}c_{l}^{n}(t)g^{\prime}(a(l\Delta_{n}))(X_{n\Delta_{n}}-X_{l\Delta_{n}});\\
			U_{t}^{n,5}= & \sqrt{n\Delta_{n}}\mathcal{R}_{t}^{n}.
		\end{align*}
		The next lemma shows that, under our assumptions, all remainder terms are asymptotically negligible and the asymptotic behaviour is therefore determined by $V^n$.
		\begin{lemma}\label{error1_CLT} Under the assumptions of Theorems
			\ref{CLT_Phin} and \ref{CLT_Lambdan_hatLambdan}, $U^{n,i}\overset{u.c.p}{\rightarrow}0$
			for $i=1,\ldots,5$.
			
		\end{lemma}
		
		\begin{proof}We treat the terms $U^{n,i}$ separately.
			
			\textbf{Case $i=1$.}
			That $U_{t}^{n,1}$ is negligible when $c_{l}^{n}(t)=\mathbf{1}_{0\leq l\leq[t/\Delta_{n}]-1}$ follows exactly as in the proof of Theorems \ref{consitency_Lambdan} and \ref{consitency_Lambdan_bar}
			by replacing $m_{n}$ with $[T/\Delta_{n}]$ in \eqref{eq:error_Ubound_unbounded}.
			Now, if $g\in\mathfrak{C}_{p,q}^{2}$ and either $m_{n}=n$ or $m_{n}=N_{n}$,
			the same argument yields that, uniformly on $t\in[0,T]$,
			\begin{equation}
				\begin{aligned} U_{t}^{n,1}= \mathrm{O}_{\mathbb{P}}(\frac{m_{n}}{n}(n\Delta_{n})^{\frac{1-p\land q}{2}})+\mathrm{O}_{\mathbb{P}}(\int_{0}^{m_{n}\Delta_{n}}a(s)^{p\land q}\mathrm{d}s/\sqrt{n\Delta_{n}})+\mathrm{o}(1).
				\end{aligned} 
				\label{eq:error_Ubound_unbounded-1}
			\end{equation}
			Since $m_{n}/n$ is bounded, we have that $U^{n,1}\overset{u.c.p}{\longrightarrow}0$
			whenever $p,q>1$, or $p\land q=1$ and $m_{n}=N_n$. When $0\leq p\land q<1$, Assumption \ref{as:samplescheme} together with \eqref{eq:restriction_windows_thm} implies that $\sqrt{n\Delta_{n}}\frac{N_{n}}{n}\rightarrow0$. Indeed, if $\kappa\leq1/2$ the claim is obvious. If $1/2<\kappa<\frac{1}{2}(1+\frac{1}{\varpi})$, we obtain that
			\[
			\sqrt{n\Delta_{n}}\frac{N_{n}}{n}\leq C(n\Delta_{n}^{\frac{1}{1-2(1-\kappa)}})^{\frac{1-2(1-\kappa)}{2}}\rightarrow0.
			\]
			which again implies $U^{n,1}\overset{u.c.p}{\to}0$.
			
			\textbf{Case $i=2$.}
			Let $\tilde{m}_{n}=[T/\Delta_{n}]$ when $c_{l}^{n}(t)=\mathbf{1}_{0\leq l\leq[t/\Delta_{n}]-1}$
			and $\tilde{m}_{n}=m_{n}$ otherwise. Using \eqref{eq:assumption_ac_estimate}, we get that
			\[
			\sup_{t\in[0,T]}\rvert U_{t}^{n,2}\rvert\lesssim\sqrt{n\Delta_{n}^{3}}\int_{0}^{\tilde{m}_{n}\Delta_{n}}g^{\prime}(a(\tau_{n}(s)))\phi(s)\mathrm{d}s.
			\]
			Since $\sqrt{n\Delta_{n}^{3}}\rightarrow0,$ and $\phi$ is bounded (being c\`adl\`ag and vanishing at $+\infty$), it suffices to show that $\int_{0}^{m_{n}\Delta_{n}}g^{\prime}[a(\tau_{n}(s))]\mathrm{d}s$
			is bounded over $n$. If $\tilde{m}_{n}=[T/\Delta_{n}]$ this is immediate from the continuity of $g^{\prime}$. If
			$\tilde{m}_{n}=m_{n}$ and $g\in\mathfrak{C}_{p,q}^{2}$, Remark \ref{integrability_C2_pq}
			implies that 
			\begin{equation}
				\int_{0}^{m_{n}\Delta_{n}}\rvert g^{\prime}[a(\tau_{n}(s))]\rvert\mathrm{d}s\lesssim\int_{0}^{m_{n}\Delta_{n}}a(\tau_{n}(s))^{p+1}\mathrm{d}s\lesssim \Delta_{n}+\int_{0}^{\infty}a(s)^{p+1}\mathrm{d}s,\label{eq:int1_lemmaU}
			\end{equation}
			just as required.
			
			\textbf{Case $i=3$.}
			Since $xa(x)\rightarrow0$, as $x\rightarrow\infty$,
			\[
			\sup_{t\in[0,T]}\rvert U_{t}^{n,3}\rvert\lesssim_{a}\frac{1}{\sqrt{n\Delta_{n}}}\int_{0}^{\tilde{m}_{n}\Delta_{n}}g^{\prime}[a(\tau_{n}(s))]\mathrm{d}s\rightarrow0,
			\]
			due to \eqref{eq:int1_lemmaU}. 
			
			\textbf{Case $i=4$.}
			Using the same arguments as above for $i=4$ along with the the stationarity of $X$, we obtain $\sup_{t\in[0,T]}|U_t^{n,4}|
			\overset{\mathbb P}{\to}0.$

			\textbf{Case $i=5$.}
			We first consider the case $c_{l}^{n}(t)=\mathbf{1}_{0\leq l\leq[t/\Delta_{n}]-1}$. Since $g\in\mathcal{C}^{2}$ and $a$ is
			Lipschitz (thanks to Assumption \ref{as:trawl}), we deduce as in \eqref{eq:int1_lemmaU}
			\[	\rvert U_t^{n,5}\rvert\lesssim\sqrt{n\Delta_{n}}\int_{0}^{[t/\Delta_{n}]\Delta_{n}}\left|g(a(\tau_{n}(s)))-g(a(s))\right|\mathrm{d}s+\sqrt{n\Delta_{n}^{3}}\lesssim\sqrt{n\Delta_{n}^{3}}\rightarrow0.
			\]
			Assume now $c_l^n(t)=\mathbf 1_{[t/\Delta_n]\le l\le m_n-1}$  and let
			$g\in\mathfrak{C}_{p,q}^{2}$. Then,
			\[
			\rvert g(a(\tau_{n}(s)))-g(a(s))\rvert\lesssim \Delta_{n}a(\tau_{n}(s)),
			\]
			by the Mean-Value Theorem, the Lipschitz property of $a$, and the fact that $g'(x)=\mathrm{o}(|x|)$ as $x\to0$. Furthermore, from \eqref{eq:boundsforg_gprime} we also have that $\rvert g(a(s))\rvert\leq Ca(s)^{2+p}$. Thus, 
			\[
			\rvert\mathcal{R}_{t}^{n}\rvert\lesssim \Delta_n \int_{0}^{m_{n}\Delta_{n}}a(\tau_{n}(s))\mathrm{d}s+\int_{m_{n}\Delta_{n}}^{\infty}a(s)^{2+p}\mathrm{d}s+\Delta_{n}.
			\]
			Using this, Assumption \ref{as:trawl}, and \eqref{eq:int1_lemmaU} we obtain that, uniformly on $t\in[0,T]$,
			\[
			\sqrt{n\Delta_{n}}\rvert\mathcal{R}_{t}^{n}\rvert=\mathrm{o}(1)+\sqrt{n\Delta_{n}}(m_{n}\Delta_{n})^{1-\alpha(p+2)}.
			\]
			Therefore, if $m_{n}=n$,
			we have $\sqrt{n\Delta_{n}}(m_{n}\Delta_{n})^{1-\alpha(p+2)}=(n\Delta_{n})^{\frac{3}{2}-\alpha(p+2)}\rightarrow0$
			because $\alpha>1$. If instead $m_{n}=N_{n}$, Assumption
			\ref{as:samplescheme} and \eqref{eq:restriction_windows_thm},
			again imply
			\[
			\sqrt{n\Delta_{n}}(m_{n}\Delta_{n})^{1-\alpha(p+2)}\leq C(n\Delta_{n}^{\frac{2((2+p)\alpha-1)-1}{2\kappa((2+p)\alpha-1)-1}})^{-\frac{2\kappa((2+p)\alpha-1)-1}{2}}\rightarrow0,
			\]
			due to the fact that $\kappa>\frac{1}{\varpi}+\frac{\varpi-1}{2\varpi((2+p)\alpha-1)}>\frac{1}{2((2+p)\alpha-1)}.$
			This concludes the proof.\end{proof}
		
		The next step is to write $V^{n}$, defined in \eqref{eq:leading_V_1},
		as a sum of martingale differences plus an error term. Once this is achieved, the proof of Theorems \ref{CLT_Phin} and \ref{CLT_Lambdan_hatLambdan}
		will follow as an application of, for instance, Theorem 6.1 in \cite{HauslerLuschgy15}
		(c.f. Theorem IX 7.19 in \cite{JacShri02}). To do this, we use the decomposition (\ref{eq:second_decomp_mtgdiff}) and write
		\[
		V_{t}^{n}=V_{t}^{\prime n}+U_{t}^{\prime,n},
		\]
		where (see \eqref{eq:c_sumandsdef})
		\begin{equation}
			\begin{aligned}
				V_{t}^{\prime n}:=&\sqrt{\frac{\Delta_{n}}{n}}\sum_{l=0}^{n-1}\sum_{j=l}^{n-1}c_{l}^{n}(t)g^{\prime}(a(l\Delta_{n}))\varXi_{j,l,n};\\
				U_{t}^{\prime,n}:=&\sqrt{\frac{\Delta_{n}}{n}}\sum_{l=0}^{n-1}\sum_{j=l}^{n-1}c_{l}^{n}(t)g^{\prime}(a(l\Delta_{n}))R_{j,l,n}.
			\end{aligned}
			\label{eq:def_Vprime_Uprime}
		\end{equation}
		Note that $\varXi_{j,l,n}$ is $\mathscr{G}_{j}^{n}$-measurable (see
		\eqref{filtrationoverspace}) and, by construction, $\mathbb{E}(\varXi_{j,l,n}\rvert\mathscr{G}_{j-1}^{n})=0$,
		so $(\varXi_{j,l,n}:0\leq l\leq j)_{j\geq0}$ forms a martingale difference sequence with
		respect to $(\mathscr{G}_{j}^{n})_{j\ge0}$. The next lemma shows that
		$U^{\prime,n}$ is asymptotically negligible.
		
		\begin{lemma}\label{error2_CLT} Under the assumptions of Theorems
			\ref{CLT_Phin} and \ref{CLT_Lambdan_hatLambdan}, $U_{t}^{\prime,n}\overset{\mathbb{P}}{\rightarrow}0$. 
			
		\end{lemma}
		
		\begin{proof}As before, we may assume without loss of generality that $L$ has mean zero.
			Since each $\rho_{j,l,n}^{(\ell)}$ (see \eqref{eq:rest_negli}) is
			a martingale difference, we only need to show that for $\ell=1,2,3$
			\[
			\frac{1}{n\Delta_{n}}\mathcal{E}_{n}^{(\ell)}(t):=\frac{1}{n\Delta_{n}}\sum_{j=0}^{n}\mathbb{E}\left[\left(\Delta_{n}\sum_{l=0}^{j}c_{l}^{n}(t)g^{\prime}(a(l\Delta_{n}))\rho_{j,l,n}^{(\ell)}\right)^{2}\right]\rightarrow0.
			\]
			To see this, first note that the quantity $\Delta_{n}\sum_{l=0}^{j}\rvert c_{l}^{n}(t)g^{\prime}(a(l\Delta_{n}))\rvert$
			is uniformly bounded. Indeed, if $c_{l}^{n}(t)=\mathbf{1}_{0\leq l\leq[t/\Delta_{n}]-1}$, the continuity of $g'$ and $a$ implies
			\[
			\Delta_{n}\sum_{l=0}^{j}\rvert c_{l}^{n}(t)g^{\prime}(a(l\Delta_{n}))\rvert\leq\Delta_{n}\sum_{l=0}^{[T/\Delta_{n}]}\lvert g^{\prime}(a(l\Delta_{n}))\rvert\leq CT.
			\]
			If instead $c_{l}^{n}(t)=\mathbf{1}_{[t/\Delta_{n}]\leq l\leq n-1}$
			or $c_{l}^{n}(t)=\mathbf{1}_{[t/\Delta_{n}]\leq l\leq N_{n}-1}$ and
			$g\in\mathfrak{C}_{p,q}^{2}$, the same bound follows from
			\eqref{eq:int1_lemmaU}.	Hence, by Jensen's inequality, 
			\[
			\mathcal{E}_{n}^{(\ell)}(t)\lesssim\Delta_{n}\sum_{j=0}^{n}\sum_{l=0}^{j}\rvert c_{l}^{n}(t)g^{\prime}(a(l\Delta_{n}))\rvert\mathbb{E}[(\rho_{j,l,n}^{(\ell)})^{2}].
			\]
			Moreover, Lemma \ref{moment_estimates_ahat-1} yields
			\begin{align*}
				\mathbb{E}[(\rho_{j,l,n}^{(\ell)})^{2}] & \leq a(0)\Delta_{n}\int_{(n-j)\Delta_{n}}^{\infty}a(s)\mathrm{d}s;\,\ell=1,3;\\
				\mathbb{E}[(\rho_{j,l,n}^{(2)})^{2}] &\leq \mathbb{E}[X_{0}^{2}]\int_{j\Delta_{n}}^{(j+1)\Delta_{n}}a(n\Delta_{n}-s)\mathrm{d}s.
			\end{align*}
			Therefore,
			\[
			\frac{1}{n\Delta_{n}}\mathcal{E}_{n}^{(2)}(t)\lesssim\frac{1}{n\Delta_{n}}\sum_{j=0}^{n-1}\int_{j\Delta_{n}}^{(j+1)\Delta_{n}}a(n\Delta_{n}-s)\mathrm{d}s\leq\mathrm{Leb}(A)/n\Delta_{n}\rightarrow0,
			\]
			and, by the Dominated Convergence Theorem, 
			\[
			\frac{1}{n\Delta_{n}}\mathcal{E}_{n}^{(\ell)}(t)\leq C\frac{1}{\Delta_{n}n}\int_{0}^{n\Delta_{n}}\int_{r}^{\infty}a(s)\mathrm{d}s\mathrm{d}r\rightarrow0,\,\,\ell=1,3,
			\]
			as required.
		\end{proof}
		
		We are now ready to prove the CLTs for our statistics. In view of the
		previous two lemmas and the definition of $\tilde a$ (see
		(\ref{eq:def_a_tilde})), we may and do assume that
		$L$ has mean $0$. Using the decomposition $V_t^n = V_t^{\prime n} +
		U_t^{\prime,n}$, the proof reduces to applying a martingale central
		limit theorem to the leading term $V^{\prime n}$.
		\begin{proof}[Proof of Theorems  \ref{CLT_Phin} and \ref{CLT_Lambdan_hatLambdan}]
			We remind the reader that  $\mathcal{S}^{n}(g)$
			and $\mathcal{S}(g)$ denote any of the pairs $(\Psi^n(g),\Psi(g))$, $(\Lambda^n(g),\Lambda(g))$, or $(\bar{\Lambda}^n(g),\bar{\Lambda}(g))$. Recall also the decomposition
			\[
			\sqrt{n\Delta_{n}}\left(\mathcal{S}_{t}^{n}(g)-\mathcal{S}_{t}(g)\right)
			=
			V_t^n+\sum_{\ell=1}^{5}U_t^{n,\ell}.
			\]
			By Lemma \ref{error1_CLT} and Remark \ref{Rmk_cont_Z}, it is enough to show that $V^{ n}\rightarrow Z$ in $\mathcal D([0,T])$, where $Z$ is the limit process stated in the theorems. We proceed in three steps. First we establish tightness of $V^n$. Next we identify the finite-dimensional limits using the Cram\'er-Wold device. Finally we verify the martingale CLT conditions.
			
			\textbf{Step 1: Tightness.} From Lemma \ref{moment_estimates_ahat}, we deduce that for all $0\leq s\leq t\leq T$
			\[
			\mathbb{E}(\rvert V_{t}^{n}-V_{s}^{n}\rvert^{2})^{1/2}\leq C\Delta_{n}\sum_{l=[s/\Delta_{n}]}^{[t/\Delta_{n}]-1}\rvert g(a(l\Delta_{n}))\rvert\leq C\Delta_{n}([t/\Delta_{n}]-[s/\Delta_{n}]),
			\]
			Therefore, by Theorem 13.5 and equation
			(13.14) in  \cite{Billingsley99}, the sequence $V^{n}$ is tight on $\mathcal D([0,T])$.
			
			\textbf{Step 2: Cram\'er-Wold reduction.}
			Fix $M\in\mathbb{N}$, $0=t_{0}\leq t_{1}\leq\cdots\leq t_{M}\leq T$
			, and $z_{1},z_{2},\ldots,z_{M}\in\mathbb{R}$. By Lemma \ref{error2_CLT}, convergence of the finite-dimensional distributions of $V^n$ towards those of $Z$ reduces to showing $\sum_{k=1}^{M}z_{k}V_{t_{k}}^{\prime n}\overset{d}{\rightarrow}\sum_{k=1}^{M}z_{k}Z_{t_{k}}$, with $V^{\prime n}$ as in \eqref{eq:def_Vprime_Uprime}. Define 
			\begin{equation}
				\overline{c}_{l}^{n}=\sum_{k=1}^{M}z_{k}c_{l}^{n}(t_{k})g^{\prime}(a(l\Delta_{n})), \quad\zeta_{j}^{n}=\Delta_{n}\sum_{l=0}^{j}\overline{c}_{l}^{n}\varXi_{j,l,n},\label{eq:zi_def}
			\end{equation}
			in which $\varXi_{j,l,n}=\sum_{\ell=1}^{4}\xi_{j,l,n}^{(\ell)}$ (see
			\eqref{errorstpos}). Then
			\[
			\sum_{k=1}^{M}z_{k}V_{t_{k}}^{\prime n}=\frac{1}{\sqrt{n\Delta_{n}}}\sum_{j=0}^{n}\zeta_{j}^{n}.
			\]
			Since $\zeta_{j}^{n}$ is a $(\mathscr{G}_{j}^{n})$-martingale difference (see our discussion before Lemma \ref{error2_CLT}), the CLT of Theorem 6.1 in \cite{HauslerLuschgy15} (c.f. Theorem IX 7.19 in \cite{JacShri02}) applies provided that for some $r>2$ 
			\begin{align}
				1. \,\sigma^2_n:=\frac{1}{n\Delta_{n}}\sum_{j=1}^{n}\mathbb{E}[(\zeta_{j}^{n})^{2}\mid\mathscr{G}_{j-1}^{n}] & \overset{\mathbb{P}}{\rightarrow}\sum_{k,k^{\prime}}z_{k}z_{k^{\prime}}\mathbb{E}(Z_{t_{k}}Z_{t_{k^{\prime}}}),\label{asymptoticvar}\\
				2.\,\frac{1}{(n\Delta_{n})^{r/2}}\sum_{j=1}^{n}\mathbb{E}(\mid\zeta_{j}^{n}\mid^{r}) & \overset{\mathbb{P}}{\rightarrow}0.\label{linderbergfellercond}
			\end{align}
			
			\textbf{Step 3: Verification of \eqref{asymptoticvar} and \eqref{linderbergfellercond}.}
			
			\noindent{}	\textit{Asymptotic Variance:} In Lemmas \ref{lemmaavar_approx}
			and \ref{lemmaavar_limit} in Appendix \ref{Appendix1}, it is shown
			that
			
			\begin{equation}
				\sigma^2_n=  \Delta_{n}^{2}\sum_{l=0}^{n-1}\sum_{l^{\prime}=0}^{n-1}\overline{c}_{l}^{n}\overline{c}_{l^{\prime}}^{n}\left(\sum_{\ell,\ell^{\prime}=1}^{4}F_{n}^{(\ell,\ell^{\prime})}(t_{l},t_{l^{\prime}})\right)+\mathrm{o}_{\mathbb{P}}(1),
				\label{eq:eq:AVARapprox}
			\end{equation}
			where $F_{n}^{(\ell,\ell^{\prime})}$ are bounded functions vanishing outside $[0,n\Delta_{n}]$ and satisfying 
			\[ \sum_{\ell,\ell^{\prime}=1}^{4}F_{n}^{(\ell,\ell^{\prime})}(\tau_{n}(s),\tau_{n}(r))\rightarrow\Sigma_{a}(s,r), \text{ as } n\rightarrow\infty.\] 
			Thus, if $c_{l}^{n}(t)=\mathbf{1}_{[t/\Delta_{n}]\leq l\leq n-1}$, the continuity of $g^{\prime}$ immediately yields
			\[ 	\sigma^2_n\rightarrow \sum_{\ell,\ell^{\prime}=1}^{4}\sum_{q,q^{\prime}=1}^{M}z_{q}z_{q^{\prime}}\int_{0}^{t_{q}}\int_{0}^{t_{q^{\prime}}}g^{\prime}(a(s))g^{\prime}(a(r))\Sigma_{a}(s,r)\mathrm{d}r\mathrm{d}s,\]
			which coincides with \eqref{asymptoticvar}. If instead $c_{l}^{n}(t)=\mathbf{1}_{[t/\Delta_{n}]\leq l\leq n-1}$
			or $c_{l}^{n}(t)=\mathbf{1}_{[t/\Delta_{n}]\leq l\leq N_{n}-1}$ 
			\begin{equation}
				\sigma^2_n=\sum_{\ell,\ell^{\prime}=1}^{4}\sum_{q,q^{\prime}=1}^{M}z_{q}z_{q^{\prime}}\int_{t_{q}}^{m_{n}\Delta_{n}}\int_{t_{q^{\prime}}}^{m_{n}\Delta_{n}}g^{\prime}(a(\tau_{n}(s)))g^{\prime}(a(\tau_{n}(r)))F_{n}^{(\ell,\ell^{\prime})}(t_{l},t_{l^{\prime}})\mathrm{d}r\mathrm{d}s+\mathrm{o}_\mathbb{P}(1),\label{eq:AVAR_unbounded}
			\end{equation}
			Hence, to achieve \eqref{asymptoticvar}, it suffices to justify the use of the generalized dominated convergence theorem. To see this is the case, observe that from \eqref{eq:boundsforg_gprime}
			we have that 
			\begin{equation}
				\rvert g^{\prime}(a(\tau_{n}(s)))\rvert\lesssim\mathbf{1}_{s\leq\Delta_{n}}+a(s-\Delta_{n})\mathbf{1}_{s>\Delta_{n}},\label{eq:boundg_prime_GDCT}
			\end{equation}
			and the right-hand side of this converges pointwise to $a$ and is bounded. Hence the generalized dominated convergence applies.

			\noindent{}	\textit{Lyapunov's condition: }
			Arguing as in the proof of Lemma \ref{error2_CLT}, the quantity $\Delta_{n}\sum_{l=0}^{j}\rvert\overline{c}_{l}^{n}\rvert$ is uniformly bounded. Choose $2<r<\frac{r_{0}}{2}\land2\alpha$. Then Jensen's inequality, Lemma
			\ref{moment_estimates_chi's}, and Assumption \ref{as:trawl} imply that
			\[
			\mathbb{E}(\mid\zeta_{j}^{n,(\ell)}\mid^{r})\lesssim\Delta_{n}\sum_{l=0}^{j}\rvert\overline{c}_{l}^{n}\rvert\mathbb{E}(\mid\xi_{j,l,n}^{(\ell)}\mid^{r})\lesssim C\Delta_{n},
			\]
			Hence, condition \eqref{linderbergfellercond} holds concluding this the proof. \end{proof}

		\appendix
		\appendixpage

		\section{Approximations of the Asymptotic Variance}
		\label{Appendix1}
		
		In this subsection we justify the approximation \eqref{eq:eq:AVARapprox}.
		Throughout the analysis we repeatedly use the moment estimates and
		identities from Lemma \ref{moment_estimates_ahat-1}, as well as the
		fact that the variance of a sum of martingale differences equals the
		sum of the individual variances. To avoid unnecessary repetition,
		these arguments will not be explicitly referenced each time they are
		applied. Recall that $m_n=n$ or $m_n=N_n$, depending on whether
		$c_l^n(t)=\mathbf{1}_{[t/\Delta_n]\le l\le n-1}$ or
		$c_l^n(t)=\mathbf{1}_{[t/\Delta_n]\le l\le N_n-1}$.
		
		\begin{lemma}\label{lemmaavar_approx}Let $\zeta_{j}^{n}=\sum_{\ell=1}^{4}\zeta_{j}^{n,(\ell)}$
			with $\zeta_{j}^{n,(\ell)}$ as in \eqref{eq:zi_def}. Under the assumptions
			of Theorems \ref{CLT_Phin} and \ref{CLT_Lambdan_hatLambdan}, it
			holds that 
			\[
			\frac{1}{n\Delta_{n}}\sum_{j=0}^{n}\mathbb{E}[(\zeta_{j}^{n})^{2}\mid\mathscr{G}_{j-1}^{n}]=\frac{1}{n\Delta_{n}}\sum_{j=0}^{n}\mathbb{E}(\zeta_{j,n}^{2})+\mathrm{o}_{\mathbb{P}}(1).
			\]
			
		\end{lemma}
		
		\begin{proof}Plainly 
			\[
			\frac{1}{n\Delta_{n}}\sum_{j=0}^{n}\mathbb{E}[(\zeta_{j}^{n})^{2}\mid\mathscr{G}_{j-1}^{n}]=\sum_{\ell,\ell^{\prime}=1}^{4}\Delta_{n}^{2}\sum_{l=0}^{n-1}\sum_{l^{\prime}=0}^{n-1}\overline{c}_{l}^{n}\overline{c}_{l^{\prime}}^{n}\frac{1}{n\Delta_{n}}\sum_{j=l\lor l^{\prime}}^{n}\mathbb{E}(\xi_{j,l,n}^{(\ell)}\xi_{j,l^{\prime},n}^{(\ell^{\prime})}\mid\mathscr{G}_{j-1}^{n}).
			\]
			By the independent scattered property of $L$ we have that 
			\begin{equation}
				\begin{aligned}\mathbb{E}(\xi_{j,l,n}^{(1)}\xi_{j,l^{\prime},n}^{(1)}\mid\mathscr{G}_{j-1}^{n})= & \mathbb{E}(\chi_{j}^{2}\beta_{j,j,l}^{(1)}\beta_{j,j,l^{\prime}}^{(1)})+\beta_{j,j-1,l^{\prime}}^{(2)}\mathbb{E}(\chi_{j}^{2}\beta_{j,j,l}^{(1)})\\
					& +\beta_{j,j-1,l}^{(2)}\mathbb{E}(\chi_{j}^{2}\beta_{j,j,l^{\prime}}^{(1)})+\beta_{j,j-1,l}^{(2)}\beta_{j,j-1,l^{\prime}}^{(2)}\mathbb{E}(\chi_{j}^{2})\\
					& -\mathbb{E}(\chi_{j}\beta_{j,j,l}^{(2)})\mathbb{E}(\chi_{j}\beta_{j,j,l^{\prime}}^{(2)});\\
					\mathbb{E}(\xi_{j,l,n}^{(2)}\xi_{j,l^{\prime},n}^{(2)}\mid\mathscr{G}_{j-1}^{n})= & \sum_{k=l}^{j-1}\sum_{k^{\prime}=l^{\prime}}^{j-1}\chi_{k}\chi_{k^{\prime}}\mathbb{E}[\beta_{k,j,l}^{(1)}\beta_{k^{\prime},j,l^{\prime}}^{(1)}];\\
					\mathbb{E}(\xi_{j,l,n}^{(3)}\xi_{j,l^{\prime},n}^{(3)}\mid\mathscr{G}_{j-1}^{n})= & \sum_{k=l\lor l^{\prime}}^{j-1}\beta_{k,j-1,l}^{(2)}\beta_{k,j-1,l^{\prime}}^{(2)}\mathbb{E}[\alpha_{k+1,j}^{2}];\\
					\mathbb{E}(\xi_{j,l,n}^{(4)}\xi_{j,l^{\prime},n}^{(4)}\mid\mathscr{G}_{j-1}^{n})= & \sum_{k=l\lor l^{\prime}}^{j-1}\mathbb{E}(\alpha_{k+1,j}^{2})\mathbb{E}(\beta_{k,j,l}^{(1)}\beta_{k,j,l^{\prime}}^{(1)})\\
					& +\sum_{k=l}^{j-1}\sum_{k^{\prime}=l^{\prime}}^{j-1}\varsigma_{k^{\prime},j}\varsigma_{k,j}\mathbb{E}(\beta_{k,j,l}^{(1)}\beta_{k^{\prime},j,l^{\prime}}^{(1)}),
				\end{aligned}
				\label{eq:cond_variance}
			\end{equation}
			where $\varsigma_{k,j}:=\sum_{m=k+1}^{j-1}\alpha_{k+1,m}$. Furthermore,
			\begin{equation}
				\begin{aligned}\mathbb{E}(\xi_{j,l,n}^{(1)}\xi_{j,l^{\prime},n}^{(2)}\mid\mathscr{G}_{j-1}^{n})= & \sum_{k=l^{\prime}}^{j-1}\chi_{k}\beta_{j,j-1,l}^{(2)}\mathbb{E}[(\beta_{k,j,l^{\prime}}^{(1)})^{2}]\\
					& +\sum_{k=l^{\prime}}^{j-1}\chi_{k}\mathbb{E}(\beta_{k,j,l^{\prime}}^{(1)}\beta_{j,j,l}^{(1)}\chi_{j});\\
					-\mathbb{E}(\xi_{j,l,n}^{(1)}\xi_{j,l^{\prime},n}^{(3)}\mid\mathscr{G}_{j-1}^{n})= & \sum_{k=l^{\prime}}^{j-1}\beta_{j,j-1,l}^{(2)}\beta_{k,j-1,l^{\prime}}^{(2)}\mathbb{E}(\alpha_{k+1,j}\chi_{j})\\
					& +\sum_{k=l^{\prime}}^{j-1}\beta_{k,j-1,l^{\prime}}^{(2)}\mathbb{E}(\alpha_{k+1,j}\chi_{j}\beta_{j,j,l}^{(1)});\\
					-\mathbb{E}(\xi_{j,l,n}^{(1)}\xi_{j,l^{\prime},n}^{(4)}\mid\mathscr{G}_{j-1}^{n})= & \sum_{k=l^{\prime}}^{j-1}\mathbb{E}(\beta_{j,j,l}^{(1)}\chi_{j}\alpha_{k+1,j}\beta_{k,j,l^{\prime}}^{(1)})\\
					& +\sum_{k=l^{\prime}}^{j-1}\varsigma_{k,j}\mathbb{E}(\beta_{j,j,l}^{(1)}\chi_{j}\beta_{k,j,l^{\prime}}^{(1)})\\
					& +\sum_{k=l^{\prime}}^{j-1}\varsigma_{k,j}\beta_{j,j-1,l}^{(2)}\mathbb{E}(\chi_{j}\beta_{k,j,l^{\prime}}^{(1)});
				\end{aligned}
				\label{eq:cond_cov}
			\end{equation}
			as well as 
			\begin{equation}
				\begin{aligned}-\mathbb{E}(\xi_{j,l,n}^{(2)}\xi_{j,l^{\prime},n}^{(3)}\mid\mathscr{G}_{j-1}^{n})= & \sum_{k=l}^{j-1}\sum_{k^{\prime}=l^{\prime}}^{j-1}\chi_{k}\beta_{k^{\prime},j-1,l^{\prime}}^{(2)}\mathbb{E}[\alpha_{k^{\prime}+1,j}^{2}]\mathbf{1}_{k^{\prime}+1\leq k-l};\\
					-\mathbb{E}(\xi_{j,l,n}^{(2)}\xi_{j,l^{\prime},n}^{(4)}\mid\mathscr{G}_{j-1}^{n})= & \sum_{k=l}^{j-1}\sum_{k^{\prime}=l^{\prime}}^{j-1}\chi_{k}\varsigma_{k^{\prime},j}\mathbb{E}[\beta_{k,j,l}^{(1)}\beta_{k^{\prime},j,l^{\prime}}^{(1)}];\\
					\mathbb{E}(\xi_{j,l,n}^{(3)}\xi_{j,l^{\prime},n}^{(4)}\mid\mathscr{G}_{j-1}^{n})= & \sum_{k=l}^{j-1}\sum_{k^{\prime}=l^{\prime}}^{j-1}\beta_{k,j-1,l}^{(2)}\varsigma_{k^{\prime},j}\mathbb{E}[\alpha_{k+1,j}\beta_{k^{\prime},j,l^{\prime}}^{(1)}]\mathbf{1}_{k+1\leq k^{\prime}-l^{\prime}}.
				\end{aligned}
				\label{eq:cond_cov2}
			\end{equation}
			We now separately analyse the terms 
			\begin{equation}
				\mathcal{V}_{n}^{(\ell,\ell^{\prime})}:=\Delta_{n}^{2}\sum_{l=0}^{n-1}\sum_{l^{\prime}=0}^{n-1}\overline{c}_{l}^{n}\overline{c}_{l^{\prime}}^{n}\frac{1}{n\Delta_{n}}\sum_{j=l\lor l^{\prime}}^{n}\mathbb{E}(\xi_{j,l,n}^{(\ell)}\xi_{j,l^{\prime},n}^{(\ell^{\prime})}\mid\mathscr{G}_{j-1}^{n}).\label{eq:varcov_def}
			\end{equation}
			We repeatedly use that $\Delta_{n}\sum_{l=0}^{n-1}\rvert\overline{c}_{l}^{n}\rvert$
			is uniformly bounded, as shown in the proof of Lemma \ref{error2_CLT}. Our aim is to verify that
			\begin{equation}
				\mathcal{V}_{n}^{(\ell,\ell^{\prime})}=\mathbb{E}(\mathcal{V}_{n}^{(\ell,\ell^{\prime})})+\Delta_{n}^{2}\sum_{l=1}^{n-1}\sum_{l^{\prime}=0}^{n-1}\overline{c}_{l}^{n}\overline{c}_{l^{\prime}}^{n}\frac{\mathcal{S}_{l,l^{\prime}}^{n}}{n\Delta_{n}},\label{eq:decomp_sll_AVAR}
			\end{equation}
			where $\mathcal S_{l,l'}^n$ can be written as 
			\[
			\mathcal S_{l,l'}^n=\sum_{\iota=1}^{N}\mathcal S_{l,l'}^{n,\iota},
			\qquad N\in\mathbb N \text{ fixed},
			\] and each summand satisfies either
			\begin{equation}
				\mathbb{E}(\rvert\mathcal{S}_{l,l^{\prime}}^{n,\iota}\rvert^{2})\lesssim n\Delta_{n}\upsilon_{n}, \quad\upsilon_{n}=\mathrm{o}(n\Delta_{n}),\label{eq:bound1_neglig_AVAR}
			\end{equation}
			uniformly on $l,l^{\prime}$, or
			\begin{equation}
				\mathbb{E}(\rvert\mathcal{S}_{l,l^{\prime}}^{n,\iota}\rvert^{2})\lesssim n\Delta_{n}[\Delta_{n}(l\lor l^{\prime})+C],\label{eq:bound2_neglig_AVAR}
			\end{equation}
			for some constant $C\geq0$. These bounds are sufficient for our purposes. Indeed, by Jensen's inequality
			\[
			\mathbb{E}(\rvert\mathcal{V}_{n}^{(\ell,\ell^{\prime})}-\mathbb{E}(\mathcal{V}_{n}^{(\ell,\ell^{\prime})})\rvert^{2})\lesssim\sum_{\iota=1}^{\mathrm{N}}\Delta_{n}^{2}\sum_{l=0}^{n-1}\sum_{l^{\prime}=0}^{n-1}\rvert\overline{c}_{l}^{n}\rvert\rvert\overline{c}_{l^{\prime}}^{n}\rvert\frac{1}{(n\Delta_{n})^{2}}\mathbb{E}(\rvert\mathcal{S}_{l,l^{\prime}}^{n,\iota}\rvert^{2}).
			\]
			Therefore, for all $\iota=1,\ldots,\mathrm{N}$, either
			\[
			\Delta_{n}^{2}\sum_{l=0}^{n-1}\sum_{l^{\prime}=0}^{n-1}\rvert\overline{c}_{l}^{n}\rvert\rvert\overline{c}_{l^{\prime}}^{n}\rvert\frac{1}{(n\Delta_{n})^{2}}\mathbb{E}(\rvert\mathcal{S}_{l,l^{\prime}}^{n,\iota}\rvert^{2})\lesssim\upsilon_{n}/n\Delta_{n}\rightarrow0,\,\,n\rightarrow\infty,
			\]
			or 
			\[
			\Delta_{n}^{2}\sum_{l=0}^{n-1}\sum_{l^{\prime}=0}^{n-1}\rvert\overline{c}_{l}^{n}\rvert\rvert\overline{c}_{l^{\prime}}^{n}\rvert\frac{1}{(n\Delta_{n})^{2}}\mathbb{E}(\rvert\mathcal{S}_{l,l^{\prime}}^{n,\iota}\rvert^{2})\lesssim\frac{1}{n}\sum_{l=0}^{n-1}\rvert\overline{c}_{l}^{n}\rvert t_{l}+1/(n\Delta_{n}).
			\]
			In the latter case,  as $n\rightarrow\infty$
			\begin{enumerate}
				\item If $c_{l}^{n}(t)=\mathbf{1}_{0\leq l\leq[t/\Delta_{n}]-1},$ by the
				continuity of $g^{\prime}$ 
				\[
				\frac{1}{n}\sum_{l=0}^{n-1}\rvert\overline{c}_{l}^{n}\rvert t_{l}\lesssim T/n\Delta_{n}\rightarrow0.
				\]
				\item Otherwise, arguing as in \eqref{eq:int1_lemmaU}, we deduce that 
				\[
				\frac{1}{n}\sum_{l=0}^{n-1}\rvert\overline{c}_{l}^{n}\rvert t_{l}\lesssim\mathrm{O}(1/(n\Delta_{n}))+\frac{m_{n}}{n}\times\frac{1}{m_{n}\Delta_{n}}\int_{0}^{m_{n}\Delta_{n}}a(s)^{p+1}s\mathrm{d}s\rightarrow0,
				\]
				thanks to Assumption \ref{as:trawl} and \eqref{eq:window_assumption}. 
			\end{enumerate}
			Therefore, in any situation, it holds that $\mathbb{E}(\rvert\mathcal{V}_{n}^{(\ell,\ell^{\prime})}-\mathbb{E}(\mathcal{V}_{n}^{(\ell,\ell^{\prime})})\rvert^{2})\rightarrow0$,
			as $n\rightarrow\infty$, which is exactly the conclusion of this
			lemma. In what is left of the proof, for each $\mathcal{V}_{n}^{(\ell,\ell^{\prime})}$ we identify a decomposition of the form \eqref{eq:decomp_sll_AVAR}. Then, we verify that the corresponding summands satisfy either \eqref{eq:bound1_neglig_AVAR} or \eqref{eq:bound2_neglig_AVAR}.
			
			For convenience, when $\ell=\ell'$ we refer to $\mathcal V_n^{(\ell,\ell')}$ as the conditional variance, and otherwise as the conditional covariance. We also stress that the notation $\mathcal S_{l,l'}^n$ will be reused from case to case: only the decomposition
			$\mathcal S_{l,l'}^n=\sum_{\iota=1}^{N}\mathcal S_{l,l'}^{n,\iota}$
			is common, while the precise definition of $\mathcal S_{l,l'}^{n,\iota}$ depends on the term under consideration.
			
			\subsubsection*{Conditional Variances}
			\begin{description}
				\item [{$\ell=1:$}] Set $\mathfrak{K}_{3}=\int x^{3}\nu(dx)$ and write
				\begin{align*}
					\mathcal{S}_{l,l^{\prime}}^{n}:= & \mathfrak{K}_{3}\int_{t_{l}}^{t_{l+1}}a(s)\mathrm{d}s\sum_{j=l\lor l^{\prime}}^{n-1}\beta_{j,j-1,l^{\prime}}^{(2)} +\mathfrak{K}_{3}\int_{t_{l}}^{t_{l+1}}a(s)\mathrm{d}s\sum_{j=l\lor l^{\prime}}^{n-1}\beta_{j,j-1,l}^{(2)}\\
					& +\int_{0}^{\Delta_{n}}a(s)\mathrm{d}s\sum_{j=l\lor l^{\prime}}^{n-1}[\beta_{j,j-1,l}^{(2)}\beta_{j,j-1,l^{\prime}}^{(2)}-\mathbb{E}(\beta_{j,j-1,l}^{(2)}\beta_{j,j-1,l^{\prime}}^{(2)})]
					=:  \sum_{\iota=1}^{3}\mathcal{S}_{l,l^{\prime}}^{n,\iota}.
				\end{align*}
				Since the sequence $\beta_{j,j-1,l^{\prime}}^{(2)}$
				is $l^{\prime}$-dependent and $\mathbb{E}[(\beta_{j,j-1,l^{\prime}}^{(2)})^{2}]\leq\mathrm{Leb}(A),$
				it holds that
				\[
				\mathbb{E}\left[\left(\sum_{j=l\lor l^{\prime}}^{n-1}\beta_{j,j-1,l^{\prime}}^{(2)}\right)^{2}\right]\leq Cn(l^{\prime}+1).
				\]
				Hence, $\mathbb{E}(\rvert\mathcal{S}_{l,l^{\prime}}^{n,1}\rvert^{2})\lesssim n\Delta_{n}(l^{\prime}\Delta_{n}+C).$
				The same argument applies to $\mathcal S_{l,l'}^{n,2}$. Moreover, since $\beta_{j,j-1,l}^{(2)}\beta_{j,j-1,l^{\prime}}^{(2)}$
				is $l\lor l^{\prime}$-dependent, the same bound also holds for
				$\mathcal S_{l,l'}^{n,3}$. This is exactly \eqref{eq:bound2_neglig_AVAR}.
				\item [{$\ell=2:$}] By interchanging the order of summation we may write
				\[
				\sum_{j=l\lor l^{\prime}}^{n}\mathbb{E}(\xi_{j,l,n}^{(2)}\xi_{j,l^{\prime},n}^{(2)}\mid\mathscr{G}_{j-1}^{n})=\sum_{k=l,k^{\prime}=l^{\prime}}^{n-1}\chi_{k}\chi_{k^{\prime}}\mathrm{Leb}(\cup_{i=0}^{(k-l)\land(k^{\prime}-l^{\prime})}\cup_{j=(k\lor k^{\prime})+1}^{n-1}\mathcal{P}_{A}^{n}(i,j)).
				\]
				Thus, (\ref{eq:decomp_sll_AVAR}) holds with 
				\[
				\mathcal{S}_{l,l^{\prime}}^{n}:= \sum_{k=l\lor l^{\prime}}^{n-1}[\chi_{k}^{2}-\mathbb{E}(\chi_{k}^{2})]b_{k,l,l^{\prime}}+\sum_{k=l}^{n-1}\chi_{k}\nu_{k,l^{\prime},l} +\sum_{k^{\prime}=l^{\prime}}^{n-1}\chi_{k^{\prime}}\tilde{\nu}_{k^{\prime},l,l^{\prime}}
				=: \sum_{\iota=1}^{3}\mathcal{S}_{l,l^{\prime}}^{n,\iota},
				\]
				where 
				\begin{align*}
					b_{k,l,l^{\prime}}:= & \mathrm{Leb}(\cup_{i=0}^{(k-l)\land(k-l^{\prime})}\cup_{j=k+1}^{n-1}\mathcal{P}_{A}^{n}(i,j));\\
					\nu_{k,l,l^{\prime}}:= & \sum_{k^{\prime}=l^{\prime}}^{n-1}\chi_{k^{\prime}}\mathrm{Leb}(\cup_{i=0}^{(k-l)\land(k^{\prime}-l^{\prime})}\cup_{j=k+1}^{n-1}\mathcal{P}_{A}^{n}(i,j))\mathbf{1}_{k^{\prime}\leq k-1};\\
					\tilde{\nu}_{k^{\prime},l,l^{\prime}}:= & \sum_{k=l}^{n-1}\chi_{k}\mathrm{Leb}(\cup_{i=0}^{(k-l)\land(k^{\prime}-l^{\prime})}\cup_{j=k^{\prime}+1}^{n-1}\mathcal{P}_{A}^{n}(i,j))\mathbf{1}_{k\leq k^{\prime}-1}
				\end{align*}
				Note that $b_{k,l,l^{\prime}}\leq\mathrm{Leb}(A)$ and $\mathbb{E}(\rvert\chi_{k}^{2}-\mathbb{E}(\chi_{k}^{2})\rvert^{2})\lesssim\Delta_{n}$. Hence, 
				\[
				\mathbb{E}(\rvert\mathcal{S}_{l,l^{\prime}}^{n,1}\rvert^{2})=\sum_{k=l}^{n-1}\mathbb{E}(\rvert\chi_{k}^{2}-\mathbb{E}(\chi_{k}^{2})\rvert^{2})b_{k,l,l^{\prime}}^{2}\leq\mathrm{Leb}(A)^{2}n\Delta_{n}.
				\]
				Furthermore
				\begin{align*}
					\mathbb{E}(\rvert\chi_{k}\nu_{k,l^{\prime},l}\rvert^{2})\lesssim & \Delta_{n}^{2}\sum_{k^{\prime}=l^{\prime}}^{n-1}\mathrm{Leb}(\cup_{i=0}^{k^{\prime}-l^{\prime}}\cup_{j\geq k}\mathcal{P}_{A}^{n}(i,j))^{2}\mathbf{1}_{k^{\prime}\leq k}\\
					\lesssim & \Delta_{n}\int_{0}^{n\Delta_{n}}\left(\int_{s}^{\infty}a(u)\mathrm{d}u\right)^{2}\mathrm{d}s
					\lesssim  \Delta_{n}(C_{\alpha}+(n\Delta_{n})^{1-2(\alpha-1)}),
				\end{align*}
				where in the last inequality we also applied Assumption \ref{as:trawl}.
				The same bound holds for $\mathbb{E}(\rvert\chi_{k^{\prime}}\tilde{\nu}_{k,l,l^{\prime}}\rvert^{2})$.
				In view that $\chi_{k}\nu_{k,l^{\prime},l}$
				and $\chi_{k^{\prime}}\tilde{\nu}_{k,l,l^{\prime}}$ are martingale
				differences, it follows that 
				\[
				\mathbb{E}(\rvert\mathcal{S}_{l,l^{\prime}}^{n,\iota}\rvert^{2})\lesssim n\Delta_{n}(C_{\alpha}+(n\Delta_{n})^{1-2(\alpha-1)}), \quad \iota=2,3.
				\]
				Therefore \eqref{eq:bound1_neglig_AVAR} is satisfied
				with $\upsilon_{n}=C_{\alpha}+(n\Delta_{n})^{1-2(\alpha-1)}.$
				\item [{$\ell=3:$}] By symmetry we may write 
				\[
				\mathcal{V}_{n}^{(3,3)}  =\mathbb{E}(\mathcal{V}_{n}^{(3,3)})+\Delta_{n}^{2}\sum_{l=0}^{n-1}(\overline{c}_{l}^{n})^{2}\frac{\mathcal{S}_{l,l}^{n}}{n\Delta_{n}}
				+2\Delta_{n}^{2}\sum_{l=1}^{n-1}\sum_{l^{\prime}=0}^{l-1}\overline{c}_{l}^{n}\overline{c}_{l^{\prime}}^{n}\frac{\mathcal{S}_{l,l^{\prime}}^{n}}{n\Delta_{n}},
				\]
				where
				\[
				\mathcal{S}_{l,l^{\prime}}^{n}:=\sum_{j=l}^{n}\sum_{k=l}^{j-1}[\beta_{k,j-1,l}^{(2)}\beta_{k,j-1,l^{\prime}}^{(2)}-\mathbb{E}(\beta_{k,j-1,l}^{(2)}\beta_{k,j-1,l^{\prime}}^{(2)})]\mathbb{E}[\alpha_{k+1,j}^{2}].
				\]
				
				Using
				\begin{equation}
					\beta_{k,j-1,l}^{(2)}=\sum_{m=k-l}^{j-1}\beta_{k,m,l}^{(1)},\label{eq:beta2_as_sumbeta1}
				\end{equation}
				and reordering the sums yields the decomposition $\mathcal{S}_{l,l^{\prime}}^{n}=\sum_{\imath=1}^{7}\mathcal{S}_{l,l^{\prime}}^{n,\imath}$, in which
				\begin{align*}
					\mathcal{S}_{l,l^{\prime}}^{n,1}:= & \sum_{m=0}^{n-1}\sum_{k=l\lor m+1}^{(n-1)\land(m+l^{\prime})}[\beta_{k,m,l}^{(1)}\beta_{k,m,l^{\prime}}^{(1)}-\mathbb{E}(\beta_{k,m,l}^{(1)}\beta_{k,m,l^{\prime}}^{(1)})]\sum_{j=k+1}^{n}\mathbb{E}[\alpha_{k+1,j}^{2}]\\
					\mathcal{S}_{l,l^{\prime}}^{n,2}:= & \sum_{m=l-l^{\prime}}^{n-1}\sum_{m^{\prime}=(l-l^{\prime})\lor(m-l^{\prime})+1}^{m-1}\sum_{k=l\lor m+1}^{(n-1)\land(m^{\prime}+l^{\prime})}\beta_{k,m,l}^{(1)}\beta_{k,m^{\prime},l^{\prime}}^{(1)}\sum_{j=k+1}^{n}\mathbb{E}[\alpha_{k+1,j}^{2}]\\
					\mathcal{S}_{l,l^{\prime}}^{n,3}:= & \sum_{m=0}^{n-1}\sum_{k=l\lor m+1}^{(n-1)\land(m+l)}\sum_{m^{\prime}=m+1}^{k-1}\beta_{k,m,l}^{(1)}\beta_{k,m^{\prime},l^{\prime}}^{(1)}\sum_{j=k+1}^{n}\mathbb{E}[\alpha_{k+1,j}^{2}]\\
					\mathcal{S}_{l,l^{\prime}}^{n,4}:= & \sum_{m^{\prime}=l+1}^{n-1}\sum_{m=0}^{m^{\prime}-1}\sum_{k=l\lor m+1}^{m^{\prime}}\beta_{k,m,l}^{(1)}\beta_{k,m^{\prime},l^{\prime}}^{(1)}\sum_{j=m^{\prime}+1}^{n}\mathbb{E}[\alpha_{k+1,j}^{2}]\\
					\mathcal{S}_{l,l^{\prime}}^{n,5}:= & \sum_{m=l}^{n-1}\sum_{m^{\prime}=l-l^{\prime}}^{m-1}\sum_{k=l}^{m\land(m^{\prime}+l^{\prime})}\beta_{k,m,l}^{(1)}\beta_{k,m^{\prime},l^{\prime}}^{(1)}\sum_{j=m+1}^{n}\mathbb{E}[\alpha_{k+1,j}^{2}]\\
					\mathcal{S}_{l,l^{\prime}}^{n,6}:= & \sum_{m=l}^{n-1}\sum_{k=l}^{m}[\beta_{k,m,l}^{(1)}\beta_{k,m,l^{\prime}}^{(1)}-\mathbb{E}(\beta_{k,m,l}^{(1)}\beta_{k,m,l^{\prime}}^{(1)})]\sum_{j=m+1}^{n}\mathbb{E}[\alpha_{k+1,j}^{2}]\\
					\mathcal{S}_{l,l^{\prime}}^{n,7}:= & \sum_{m^{\prime}=l+1}^{n-1}\sum_{m=l}^{m^{\prime}-1}\sum_{k=l}^{m}\beta_{k,m,l}^{(1)}\beta_{k,m^{\prime},l^{\prime}}^{(1)}\sum_{j=m^{\prime}+1}^{n}\mathbb{E}[\alpha_{k+1,j}^{2}]
				\end{align*}
				It remains to verify that each $\mathcal{S}_{l,l^{\prime}}^{n,\iota}$ satisfies \eqref{eq:bound2_neglig_AVAR}. We only focus on $\mathcal{S}_{l,l^{\prime}}^{n,1}$ and $\mathcal{S}_{l,l^{\prime}}^{n,2}$
				since the other terms can be analysed in the same way. Set $\theta_{k}:=\sum_{j=k+1}^{n}\mathbb{E}[\alpha_{k+1,j}^{2}]$. Then,
				\begin{align*}
					\mathbb{E}(\rvert\mathcal{S}_{l,l^{\prime}}^{n,1}\rvert^{2})\leq & \sum_{m=0}^{n-1}\sum_{k=l\lor m+1}^{(n-1)\land(m+l)}\sum_{p=l\lor m+1}^{(n-1)\land(m+l)}\mathbb{E}(\beta_{k,m,l}^{(1)}\beta_{k,m,l^{\prime}}^{(1)}\beta_{p,m,l}^{(1)}\beta_{p,m,l^{\prime}}^{(1)})\theta_{k}\theta_{p}
				\end{align*}
				By Lemma \ref{moment_estimates_ahat-1}, 
				\[
				\mathbb{E}(\beta_{k,m,l}^{(1)}\beta_{k,m,l^{\prime}}^{(1)}\beta_{p,m,l}^{(1)}\beta_{p,m,l^{\prime}}^{(1)})\lesssim_{a,\mathfrak{t}_{L}}\mathrm{Leb}(\cup_{i=0}^{k\land p-l}\mathcal{P}_{A}^{n}(i,m)),
				\]
				and $\rvert\theta_{k}\rvert\leq a(0)\Delta_{n}$. Hence,
				\[
				\mathbb{E}(\rvert\mathcal{S}_{l,l^{\prime}}^{n,1}\rvert^{2})\lesssim\Delta_{n}\sum_{m=0}^{n-1}\sum_{k=l\lor m+1}^{(n-1)\land(m+l)}\theta_{k}\sum_{p=l\lor m+1}^{k}\mathrm{Leb}(\cup_{i=0}^{p-l}\mathcal{P}_{A}^{n}(i,m)).
				\]
				Plainly, 
				\[
				\sum_{p=l\lor m+1}^{k}\mathrm{Leb}(\cup_{i=0}^{p-l}\mathcal{P}_{A}^{n}(i,m))=\sum_{p=l\lor m+1}^{k}\int_{t_{m}+t_{l}-t_{p}}^{t_{m+1}+t_{l}-t_{p}}a(s)\mathrm{d}s\leq\mathrm{Leb}(A),
				\]
				and 
				\begin{align*}
					\sum_{k=l\lor m+1}^{(n-1)\land(m+l)}\theta_{k} & =\sum_{k=l\lor m+1}^{(n-1)\land(m+l)}\sum_{j=k+1}^{(n-1)\land(m+l)}\mathbb{E}[\alpha_{k+1,j}^{2}]+\sum_{k=l\lor m+1}^{(n-1)\land(m+l)}\sum_{j=(n-1)\land(m+l)}^{n}\mathbb{E}[\alpha_{k+1,j}^{2}]\\
					& \leq l\Delta_{n}+\mathrm{Leb}(A),
				\end{align*}
				where in the second inequality we also used (\ref{eq:reprA}). This
				shows that $\mathcal{S}_{l,l^{\prime}}^{n,1}$ satisfies \eqref{eq:bound2_neglig_AVAR}.
				Similarly, arguments give that
				\begin{align*}
					\mathbb{E}(\rvert\mathcal{S}_{l,l^{\prime}}^{n,2}\rvert^{2})=  \sum_{m=l-l^{\prime}}^{n-1}\sum_{m^{\prime}=(l-l^{\prime})\lor(m-l^{\prime})+1}^{m-1}&\sum_{k,p=l\lor m+1}^{(n-1)\land(m^{\prime}+l^{\prime})}\mathrm{Leb}(\cup_{i=0}^{k\land p-l}\mathcal{P}_{A}^{n}(i,m))\\
					&\times\mathrm{Leb}(\cup_{i=0}^{k\land p-l^{\prime}}\mathcal{P}_{A}^{n}(i,m^{\prime}))\theta_{k}\theta_{p}\\
					\lesssim  \Delta_{n}\sum_{m=l-l^{\prime}}^{n-1}\sum_{k=l\lor m+1}^{(n-1)\land(m+l^{\prime})}\theta_{k}&\lesssim n\Delta_{n}(l^{\prime}\Delta_{n}+\mathrm{Leb}(A)).
				\end{align*}
				\item [{$\ell=4:$}] Recall that $\varsigma_{k,j}:=\sum_{m=k+1}^{j-1}\alpha_{k+1,m}$.
				Arguing as above we may write 	
				\[\mathcal{V}_{n}^{(4,4)}  =\mathbb{E}(\mathcal{V}_{n}^{(4,4)})+\Delta_{n}^{2}\sum_{l=0}^{n-1}(\overline{c}_{l}^{n})^{2}\frac{\mathcal{S}_{l,l}^{n}}{n\Delta_{n}}
				+2\Delta_{n}^{2}\sum_{l=1}^{n-1}\sum_{l^{\prime}=0}^{l-1}\overline{c}_{l}^{n}\overline{c}_{l^{\prime}}^{n}\frac{\mathcal{S}_{l,l^{\prime}}^{n}}{n\Delta_{n}},\]
				
				where 
				\begin{align*}
					\mathcal{S}_{l,l^{\prime}}^{n}= &  \sum_{k=l}^{n-1}\sum_{j=k+1}^{n}[\varsigma_{k,j}^{2}-\mathbb{E}(\varsigma_{k,j}^2)]\theta_{k,l,k,l^{\prime}}^{j} +\sum_{k=l+1}^{n-1}\sum_{k^{\prime}=l^{\prime}}^{k-1}\sum_{j=k+1}^{n}\varsigma_{k^{\prime},j}\varsigma_{k,j}\theta_{k,l,k^{\prime},l^{\prime}}^{j}\\
					& +\sum_{k^{\prime}=l+1}^{n-1}\sum_{k=l}^{k^{\prime}-1}\sum_{j=k^{\prime}+1}^{n}\varsigma_{k^{\prime},j}\varsigma_{k,j}\theta_{k,l,k^{\prime},l^{\prime}}^{j} =:\sum_{\imath=1}^{3}\mathcal{S}_{l,l^{\prime}}^{n,\imath},
				\end{align*}
				wherein we have let
				\begin{equation}
					\theta_{k,l,k^{\prime},l^{\prime}}^{j}:=\mathrm{Leb}(\cup_{i=0}^{(k-l)\land (k^{\prime}-l)^{\prime}}\mathcal{P}_{A}^{n}(i,j)). \label{theta_kl_def}
				\end{equation}
				
				In this case we show that \eqref{eq:bound1_neglig_AVAR}
				holds with $\upsilon_{n}=C_{\alpha}+(n\Delta_{n})^{1-2(\alpha-1)}.$
				By construction $\varsigma_{k,j}$ is $\mathscr{F}_{k+1}^{n}$-measurable
				and independent of $\mathscr{F}_{k}^{n}$, reason why each $\mathcal{S}_{l,l^{\prime}}^{n,\imath}$
				is a sum of martingale differences. This, along
				with the fact that $\sum_{j=k\lor k^{\prime}+1}^{n}\theta_{k,l,k^{\prime},l^{\prime}}^{j}\leq\int_{\rvert t_{k}-t_{k^{\prime}}\rvert}^{\infty}a(s)\mathrm{d}s$
				imply that
				\begin{align*}
					\mathbb{E}(\rvert\mathcal{S}_{l,l^{\prime}}^{n}\rvert^{2})\lesssim & \Delta_{n}^{2}\sum_{k=1}^{n-1}\sum_{k^{\prime}=0}^{k-1}(\int_{t_{k}-t_{k^{\prime}}}^{\infty}a(s)\mathrm{d}s)^{2}\leq n\Delta_{n}\int_{0}^{n\Delta_{n}}(\int_{u}^{\infty}a(s)\mathrm{d}s)^{2}\mathrm{d}u.
				\end{align*}
				We now can follow the same arguments as in the case $\ell=2$.
			\end{description}
			
			\subsubsection*{Conditional Covariances}
			
			By symmetry, we only need to focus on $\mathcal{V}_{n}^{(\ell,\ell^{\prime})}$
			for $\ell=1,2,3,4$ and $\ell^{\prime}\geq\ell+1$.
			\begin{description}
				\item [{$\ell=1,\ell^{\prime}=2:$}] Set 
				\begin{align*}
					\mathcal{S}_{l,l^{\prime}}^{n} & =\mathfrak{K}_{3}\sum_{k=l^{\prime}}^{n-1}\chi_{k}\sum_{j=l\lor k+1}^{n-1}\mathrm{Leb}(\cup_{i=0}^{(k-l^{\prime})\land(j-l)}\mathcal{P}_{A}^{n}(i,j))\\
					&+\sum_{j=l\lor l^{\prime}}^{n-1}\sum_{k=l^{\prime}}^{j-1}[\beta_{j,j-1,l}^{(2)}\chi_{k}-\mathbb{E}(\beta_{j,j-1,l}^{(2)}\chi_{k})]\mathbb{E}[(\beta_{k,j,l^{\prime}}^{(1)})^{2}]\\
					& =:\mathcal{S}_{l,l^{\prime}}^{n,1}+\mathcal{S}_{l,l^{\prime}}^{n,2}
				\end{align*}
				Since
				\begin{equation}
					\sum_{j=l\lor k+1}^{n-1}\mathrm{Leb}(\cup_{i=0}^{(k-l^{\prime})\land(j-l)}\mathcal{P}_{A}^{n}(i,j))\leq\mathrm{Leb}(A),
					\label{estimate_sumLeb}
				\end{equation}
				
				we conclude, as in previous cases, that $\mathcal{S}_{l,l^{\prime}}^{n,1}$ satisfies \eqref{eq:bound1_neglig_AVAR}. The second term requires a separate argument. Decompose
				$\mathcal{S}_{l,l^{\prime}}^{n,2}=\mathcal{S}_{l,l^{\prime}}^{\star,n}+\tilde{\mathcal{S}}_{l,l^{\prime}}^{n,2}$, where 
				\[
				\mathcal{S}_{l,l^{\prime}}^{\star,n}:=\sum_{j=l\lor l^{\prime}}^{l+l^{\prime}}\sum_{k=l^{\prime}}^{j-1}[\beta_{j,j-1,l}^{(2)}\chi_{k}-\mathbb{E}(\beta_{j,j-1,l}^{(2)}\chi_{k})]\mathbb{E}[(\beta_{k,j,l^{\prime}}^{(1)})^{2}].
				\]
				Arguing as in the proof of Lemma 3 in \cite{SauriVeraart23}
				(specifically the analysis done for $\mathbb{E}(\xi_{j,n}^{(2)}\xi_{j,n}^{(3)}\mid\mathscr{G}_{j-1}^{n})$
				under the notation introduced in that paper) we conclude that $\tilde{\mathcal{S}}_{l,l^{\prime}}^{n,2}$
				satisfy (\ref{eq:bound2_neglig_AVAR}) and that
				\begin{equation}
					\mathbb{E}(\rvert\mathcal{S}_{l,l^{\prime}}^{\star,n}\rvert)\lesssim\Delta_{n}((l\lor l^{\prime})+1).\label{eq:bound2_neglig_AVAR_prime}
				\end{equation}
				With \eqref{eq:bound2_neglig_AVAR_prime} in place of \eqref{eq:bound2_neglig_AVAR}, the same argument yields
				$\mathbb{E}(|\mathcal{V}_{n}^{(1,2)}-\mathbb{E}(\mathcal{V}_{n}^{(1,2)})|
				)\to0$, as $n\to\infty $.
				
				\item [{$\ell=1,\ell^{\prime}=3:$}] By applying \eqref{eq:beta2_as_sumbeta1}, we obtain that \eqref{eq:decomp_sll_AVAR} holds with
				\begin{align*}
					\mathcal{S}_{l,l^{\prime}}^{n}  =&-\sum_{m=0}^{n-1}\sum_{k=l^{\prime}}^{(n-l-1)\land(m+l^{\prime})}\beta_{k,m,l^{\prime}}^{(1)}\theta_{m,k}^{l,l^{\prime}}\\
					&-\sum_{j=l\lor (l^{\prime}+1)}^{n-1}\sum_{k=l^{\prime}}^{j-1}[\beta_{j,j-1,l}^{(2)}\beta_{k,j-1,l^{\prime}}^{(2)}-\mathbb{E}(\beta_{j,j-1,l}^{(2)}\beta_{k,j-1,l^{\prime}}^{(2)})]\mathbb{E}(\alpha_{k+1,j}^{2})\\
					& =:\mathcal{S}_{l,l^{\prime}}^{n,1}+\mathcal{S}_{l,l^{\prime}}^{n,2},
				\end{align*}
				where $\theta_{m,k}^{l,l^{\prime}}:=\sum_{j=m\lor(k+l)+1}^{n-1}\mathbb{E}(\alpha_{k+1,j}^{3})$.
				As in the case $\ell=\ell^{\prime}=3$, one
				deduces that $\mathcal{S}_{l,l^{\prime}}^{n,1}$ satisfies \eqref{eq:bound2_neglig_AVAR}. Assume without loss of generality
				that $l\geq l^{\prime}$, and put 
				\begin{align*}
					\vartheta_{j,l,l^{\prime}}^{(1)} & =\sum_{m=0}^{j-1-l}\sum_{i=0}^{m}\alpha_{i,m}\sum_{k=i+l^{\prime}}^{m+l^{\prime}}\mathbb{E}(\alpha_{k+1,j}^{2});\\
					\vartheta_{j,l,l^{\prime}}^{(2)} & =\sum_{m=j-l}^{j-1-l^{\prime}}\sum_{i=0}^{m}\alpha_{i,m}\sum_{k=i+l^{\prime}}^{m+l^{\prime}}\mathbb{E}(\alpha_{k+1,j}^{2});\\
					\vartheta_{j,l,l^{\prime}}^{(3)} & =\sum_{i=0}^{j-1-l^{\prime}}\sum_{m=j-l^{\prime}}^{j-1}\alpha_{i,m}\sum_{k=i+l^{\prime}}^{j-1}\mathbb{E}(\alpha_{k+1,j}^{2}).
				\end{align*}
				Then $\mathcal{S}_{l,l^{\prime}}^{n,2}=\sum_{\imath=1}^{3}\mathcal{\tilde{S}}_{l,l^{\prime}}^{n,\imath}$,
				where $\mathcal{\tilde{S}}_{l,l^{\prime}}^{n,1}:=-\sum_{j=l\lor l^{\prime}+1}^{n-1}\beta_{j,j-1,l}^{(2)}\vartheta_{j,l,l^{\prime}}^{(1)}$,
				and for $\imath=2,3$
				\[
				\mathcal{\tilde{S}}_{l,l^{\prime}}^{n,\imath}:=-\sum_{j=l\lor l^{\prime}+1}^{n-1}[\beta_{j,j-1,l}^{(2)}\vartheta_{j,l,l^{\prime}}^{(\imath)}-\mathbb{E}(\beta_{j,j-1,l}^{(2)}\vartheta_{j,l,l^{\prime}}^{(\imath)})].
				\]
				Note that $\beta_{j,j-1,l}^{(2)}\vartheta_{j,l,l^{\prime}}^{(\imath)}$
				are $l$-uncorrelated. Furthermore, by Rosenthal's inequality, Lemma
				\ref{moment_estimates_ahat-1}, and the bound
				\[ \sum_{k=i+l^{\prime}}^{m+l^{\prime}}\mathbb{E}(\alpha_{k+1,j}^{2})\lesssim\int_{t_{j}-t_{m+l^{\prime}}}^{t_{j+1}-t_{m+l^{\prime}}}a(s)\mathrm{d}s,\quad m\leq j-1-l^{\prime},\]
				we obtain
				\[
				\mathbb{E}(\rvert\vartheta_{j,l,l^{\prime}}^{(\imath)}\rvert^{4})\lesssim\Delta_{n}^{4},\,\,\imath=1,2.
				\]
				Likewise, $\mathbb{E}(\rvert\vartheta_{j,l,l^{\prime}}^{(3)}\rvert^{4})\lesssim\Delta_{n}^{4}$. Consequently, arguing as in the case $\ell=\ell^{\prime}=1$, we conclude
				that $\mathcal{S}_{l,l^{\prime}}^{n,2}$ fulfils \eqref{eq:bound2_neglig_AVAR}.
				\item [{$\ell=1,\ell^{\prime}=4:$}] We start by noting that (\ref{eq:decomp_sll_AVAR})
				is satisfied with 
				\begin{align*}
					\mathcal{S}_{l,l^{\prime}}^{n} & =-\mathfrak{K}_{3}\sum_{k=l^{\prime}}^{n-1}\sum_{j=l\lor k+1}^{n-1}\varsigma_{k,j}\mathrm{Leb}(\cup_{i=0}^{(k-l^{\prime})\land(j-l)}\mathcal{P}_{A}^{n}(i,j))\\
					& -\sum_{j=l\lor l^{\prime}}^{n-1}\sum_{k=l^{\prime}}^{j-1}[\beta_{j,j-1,l}^{(2)}\varsigma_{k,j}-\mathbb{E}(\beta_{j,j-1,l}^{(2)}\varsigma_{k,j})]\mathrm{Leb}(\cup_{i=0}^{k-l^{\prime}}\mathcal{P}_{A}^{n}(i,j))=:\mathcal{S}_{l,l^{\prime}}^{n,1}+\mathcal{S}_{l,l^{\prime}}^{n,2}.
				\end{align*}
				Using \eqref{estimate_sumLeb}
				and $\mathbb{E}(\varsigma_{k,p}\varsigma_{k,j})\lesssim\Delta_{n}$,
				it follows that $\mathbb{E}(\rvert\mathcal{S}_{l,l^{\prime}}^{n,1}\rvert^{2})\lesssim n\Delta_{n}$. Next, decompose 
				$\mathcal{S}_{l,l^{\prime}}^{n,2}=\mathcal{S}_{l,l^{\prime}}^{\star,n}+\mathcal{S}_{l,l^{\prime}}^{\star\star,n}$,
				where 
				\[
				\mathcal{S}_{l,l^{\prime}}^{\star,n}:=\sum_{j=l\lor l^{\prime}}^{l+l^{\prime}+1}\sum_{k=l^{\prime}}^{j-1}[\beta_{j,j-1,l}^{(2)}\varsigma_{k,j}-\mathbb{E}(\beta_{j,j-1,l}^{(2)}\varsigma_{k,j})]\mathrm{Leb}(\cup_{i=0}^{k-l^{\prime}}\mathcal{P}_{A}^{n}(i,j)).
				\]
				By setting 
				\begin{align*}
					\vartheta_{j,l,l^{\prime}}^{(1)} & =\sum_{m=l^{\prime}+1}^{j-l-1}\sum_{k=l^{\prime}}^{m-1}\alpha_{k+1,m}\mathrm{Leb}(\cup_{i=0}^{k-l^{\prime}}\mathcal{P}_{A}^{n}(i,j));\\
					\vartheta_{j,l,l^{\prime}}^{(2)} & =\sum_{m=j-l}^{j-1}\sum_{k=l^{\prime}}^{m-1}\alpha_{k+1,m}\mathrm{Leb}(\cup_{i=0}^{k-l^{\prime}}\mathcal{P}_{A}^{n}(i,j)).
				\end{align*}
				we can further decompose $\mathcal{S}_{l,l^{\prime}}^{\star\star,n}=\mathcal{\tilde{S}}_{l,l^{\prime}}^{n,1}+\mathcal{\tilde{S}}_{l,l^{\prime}}^{n,2}$,
				where 
				\[ \mathcal{\tilde{S}}_{l,l^{\prime}}^{n,1}:=-\sum_{j=l+ l^{\prime}+2}^{n-1}\beta_{j,j-1,l}^{(2)}\vartheta_{j,l,l^{\prime}}^{(1)}, \]
				and 
				\[
				\mathcal{\tilde{S}}_{l,l^{\prime}}^{n,2}:=-\sum_{j=l+ l^{\prime}+2}^{n-1}[\beta_{j,j-1,l}^{(2)}\vartheta_{j,l,l^{\prime}}^{(2)}-\mathbb{E}(\beta_{j,j-1,l}^{(2)}\vartheta_{j,l,l^{\prime}}^{(2)})].
				\]
				As in the previous case, $\beta_{j,j-1,l}^{(2)}\vartheta_{j,l,l^{\prime}}^{(\imath)}$
				are $(l+1)$-correlated and $\mathbb{E}(\rvert\vartheta_{j,l,l^{\prime}}^{(\imath)}\rvert^{4})\lesssim\Delta_{n}^{4}$,
				reason why $\mathbb{E}(\rvert\mathcal{S}_{l,l^{\prime}}^{\star\star,n}\rvert^{2})\lesssim n\Delta_{n}((l+1)\Delta_{n})$.
				It remains to treat $\mathcal{S}_{l,l^{\prime}}^{\star,n}$. By  the Cauchy-Swarchz inequality and the fact that $\mathrm{Leb}(\cup_{i=0}^{k-l^{\prime}}\mathcal{P}_{A}^{n}(i,j))=\int_{t_{j}-t_{k-l^{\prime}}}^{t_{j+1}-t_{k-l^{\prime}}}a(s)\mathrm{d}s$
				we deduce that
				\[
				\mathbb{E}(\rvert\mathcal{S}_{l,l^{\prime}}^{\star,n}\rvert)\lesssim\Delta_{n}\sum_{j=l\lor l^{\prime}}^{l+l^{\prime}+1}\left(\sum_{k=l^{\prime}}^{j-1}\int_{t_{j}-t_{k-l^{\prime}}}^{t_{j+1}-t_{k-l^{\prime}}}a(s)\mathrm{d}s\right)^{1/2}\lesssim\Delta_{n}((l\lor l^{\prime})+1).
				\]
				Thus, $\mathcal{S}_{l,l^{\prime}}^{\star,n}$ satisfies \eqref{eq:bound2_neglig_AVAR_prime}, which, as explained above, suffices.
				\item [{$\ell=2,\ell^{\prime}=3:$}] Another application of \eqref{eq:beta2_as_sumbeta1}
				yields
				\begin{align*}
					\mathcal{S}_{l,l^{\prime}}^{n}= & \sum_{k=l^{\prime}+l+1}^{n-2}\chi_{k}\sum_{k^{\prime}=l^{\prime}}^{k-l-1}\beta_{k^{\prime},k-1,l^{\prime}}^{(2)}\theta_{k,k^{\prime}}+\sum_{k=l^{\prime}+l+1}^{n-2}\sum_{k^{\prime}=l^{\prime}}^{k-l-1}[\chi_{k}\beta_{k^{\prime},k,l^{\prime}}^{(1)}-\mathbb{E}(\chi_{k}\beta_{k^{\prime},k,l^{\prime}}^{(1)})]\theta_{k,k^{\prime}}\\
					& +\sum_{m=l^{\prime}+l+2}^{n-2}\sum_{k=l^{\prime}+l+1}^{m-1}\chi_{k}\sum_{k^{\prime}=l^{\prime}}^{k-l-1}\beta_{k^{\prime},m,l^{\prime}}^{(1)}\theta_{m,k^{\prime}}
					=:  \sum_{\imath=1}^{3}\mathcal{S}_{l,l^{\prime}}^{n,\imath},
				\end{align*}
				where $\theta_{k,k^{\prime}}:=\sum_{j=k+1}^{n-1}\mathbb{E}[\alpha_{k^{\prime}+1,j}^{2}]$.
				Once again, each $\mathcal{S}_{l,l^{\prime}}^{n,\imath}$ is a sum
				of martingale differences. By reasoning as in the proof
				of Lemma 3 in \cite{SauriVeraart23} (see the treatment of $\frac{1}{n\Delta_{n}}\sum_{j=l_{n}}^{n}\mathbb{E}(\xi_{j,n}^{(3)}\xi_{j,n}^{(4)}\mid\mathscr{G}_{j-1}^{n})$),
				the second moment of each summand in $\mathcal{S}_{l,l^{\prime}}^{n,\imath}$ is
				$\mathrm{O}(\Delta_n)$. Hence, $\mathcal{S}_{l,l^{\prime}}^{n,\imath}$ satisfies \eqref{eq:bound1_neglig_AVAR}, for every $\imath=1,2,3$.
				\item [{$\ell=2,\ell^{\prime}=4:$}] Using \eqref{theta_kl_def}, we have
				\[
				\mathcal{S}_{l,l^{\prime}}^{n}=-\sum_{j=l\lor l^{\prime}}^{n-1}\sum_{k=l}^{j-1}\sum_{k^{\prime}=l^{\prime}}^{j-1}[\chi_{k}\varsigma_{k^{\prime},j}-\mathbb{E}(\chi_{k}\varsigma_{k^{\prime},j})]\theta_{k,l,k^{\prime},l^{\prime}}^{j}.
				\]
				Assume $l\geq l^{\prime}$ and note that for $k\geq k^{\prime}+1$,
				it holds
				\[
				\chi_{k}=L(\cup_{i=0}^{k^{\prime}}\mathcal{P}_{A}^{n}(i,k))+\sum_{i=k^{\prime}+1}^{k}\alpha_{i,k}=:\lambda_{k^{\prime},k}+\sum_{i=k^{\prime}+1}^{k}\alpha_{i,k}.
				\]
				Then $\mathcal{S}_{l,l^{\prime}}^{n}=\sum_{\imath=1}^{4}\mathcal{S}_{l,l^{\prime}}^{n,\imath}$,
				where 
				\begin{align*}
					\mathcal{S}_{l,l^{\prime}}^{n,1} & :=-\sum_{k^{\prime}=l}^{n-2}\sum_{m=k^{\prime}+1}^{n-2}\alpha_{k^{\prime}+1,m}\sum_{k=l}^{k^{\prime}}\chi_{k}\bar{\theta}_{m,k^{\prime},l^{\prime}};\\
					\mathcal{S}_{l,l^{\prime}}^{n,2} & :=-\sum_{i=l+1}^{n-2}\sum_{k^{\prime}=l}^{i-1}\sum_{j=i}^{n-1}\sum_{k=i}^{j-1}[\alpha_{i,k}\varsigma_{k^{\prime},j}-\mathbb{E}(\varsigma_{k^{\prime},j}\alpha_{i,k})]\theta_{k,l,k^{\prime},l^{\prime}}^{j};\\
					\mathcal{S}_{l,l^{\prime}}^{n,3} & :=-\sum_{k^{\prime}=l}^{n-3}\sum_{m=k^{\prime}+1}^{n-2}\alpha_{k^{\prime}+1,m}\sum_{k=k^{\prime}+1}^{n-2}\lambda_{k^{\prime},k}\bar{\theta}_{m\lor k,k^{\prime},l^{\prime}};\\
					\mathcal{S}_{l,l^{\prime}}^{n,4} & :=-\sum_{j=l}^{n-1}\sum_{k=l}^{j-1}\sum_{k^{\prime}=l^{\prime}}^{l-1}[\chi_{k}\varsigma_{k^{\prime},j}-\mathbb{E}(\chi_{k}\varsigma_{k^{\prime},j})]\theta_{k,l,k^{\prime},l^{\prime}}^{j},
				\end{align*}
				with
				\[ \bar{\theta}_{m,k,l,k^{\prime},l^{\prime}}:=\mathrm{Leb}(\cup_{i=0}^{(k^{\prime}-l^{\prime})\land(k-l)}\cup_{j=m+1}^{n-1}\mathcal{P}_{A}^{n}(i,j))\leq\mathrm{Leb}(A). \]
				The same line of arguments used in the cases $\ell=\ell^{\prime}=2$
				and $\ell=\ell^{\prime}=4$, result in $\mathbb{E}(\rvert\mathcal{S}_{l,l^{\prime}}^{n,\imath}\rvert^{2})\lesssim n\Delta_{n}(C_{\alpha}+(n\Delta_{n})^{1-2(\alpha-1)})$,
				for $\imath=1,2$. Next, note that $\mathcal{S}_{l,l^{\prime}}^{n,3}$
				is once again a sum of martingale differences. For $\imath=3$,
				using independence of $(\alpha_{k^{\prime}+1,m})_{k^{\prime}+1\leq m\leq n-2}$
				and $(\lambda_{k^{\prime},k})_{k^{\prime}+1\leq k\leq n-2}$ for fixed $k^{\prime}$, we deduce
				\begin{align*}
					\mathbb{E}(\rvert\mathcal{S}_{l,l^{\prime}}^{n,3}\rvert^{2})= & \sum_{k^{\prime}=l}^{n-3}\sum_{m=k^{\prime}+1}^{n-2}\mathbb{E}(\alpha_{k^{\prime}+1,m}^{2})\sum_{k=k^{\prime}+1}^{n-2}\mathbb{E}(\lambda_{k^{\prime},k}^{2})(\bar{\theta}_{m\lor k,k^{\prime},l^{\prime}})^{2}\\
					\lesssim & \sum_{k^{\prime}=l}^{n-3}\sum_{m=k^{\prime}+1}^{n-2}\mathbb{E}(\alpha_{k^{\prime}+1,m}^{2})\lesssim n\Delta_{n}.
				\end{align*}
				After changing the order of summations, we get
				that 
				\begin{align*}
					\mathcal{S}_{l,l^{\prime}}^{n,4}= & -\sum_{k^{\prime}=l^{\prime}}^{l-1}\sum_{k=l}^{n-2}[\alpha_{k^{\prime}+1,k}^{2}-\mathbb{E}(\alpha_{k^{\prime}+1,k}^{2})]\bar{\theta}_{k,k^{\prime},l^{\prime}}\\
					& -\sum_{k^{\prime}=l^{\prime}}^{l-1}\sum_{k=l}^{n-2}\sum_{m=k^{\prime}+1}^{n-2}\alpha_{k^{\prime}+1,k}\alpha_{k^{\prime}+1,m}\bar{\theta}_{m\lor k,k^{\prime},l^{\prime}}\mathbf{1}_{m\geq k+1}\\
					& -\sum_{k^{\prime}=l^{\prime}}^{l-1}\sum_{k=l}^{n-2}\sum_{m=k^{\prime}+1}^{n-2}\alpha_{k^{\prime}+1,k}\alpha_{k^{\prime}+1,m}\bar{\theta}_{m\lor k,k^{\prime},l^{\prime}}\mathbf{1}_{m\leq k-1}\\
					& -\sum_{k^{\prime}=l^{\prime}}^{l-1}\sum_{m=k^{\prime}+1}^{n-2}\alpha_{k^{\prime}+1,m}\sum_{i=0}^{k^{\prime}}\sum_{k=l}^{n-2}\alpha_{i,k}\bar{\theta}_{m\lor k,k^{\prime},l^{\prime}}\\
					& -\sum_{k^{\prime}=l^{\prime}}^{l-1}\sum_{m=k^{\prime}+1}^{n-2}\alpha_{k^{\prime}+1,m}\sum_{k=l}^{n-2}\sum_{i=k^{\prime}+2}^{k}\alpha_{i,k}\bar{\theta}_{m\lor k,k^{\prime},l^{\prime}}.
				\end{align*}
				Using this we conclude that 
				\begin{align*}
					\mathbb{E}(\rvert\mathcal{S}_{l,l^{\prime}}^{n,4}\rvert^{2})\lesssim & \mathrm{Leb}(A)+\sum_{k^{\prime}=l^{\prime}}^{l-1}\sum_{k=l}^{n-2}\sum_{m=k^{\prime}+1}^{k}\mathbb{E}(\alpha_{k^{\prime}+1,k}^{2})\mathbb{E}(\alpha_{k^{\prime}+1,m}^{2})\\
					& +\sum_{k^{\prime}=l^{\prime}}^{l-1}\sum_{m=k^{\prime}+1}^{n-2}\mathbb{E}(\alpha_{k^{\prime}+1,m}^{2})\sum_{i=0}^{k^{\prime}}\sum_{k=l}^{n-2}\mathbb{E}(\alpha_{i,k}^{2})\\
					\lesssim & \mathrm{Leb}(A)(1+\Delta_{n}+l\Delta_{n}).
				\end{align*}
				The case $l<l^{\prime}$ follows analogously.
				\item [{$\ell=3,\ell^{\prime}=4:$}]
				This case follows from $(\ell,\ell^{\prime})=(2,3)$ by swapping $k$ and $k^{\prime}$ and replacing $\chi_{k}$ with $\varsigma_{k^{\prime},j}$. Hence, \eqref{eq:decomp_sll_AVAR} holds and each $\mathcal{S}_{l,l^{\prime}}^{n,\iota}$ satisfies \eqref{eq:bound1_neglig_AVAR}.
		\end{description}\end{proof}
		
		In the proof of the next lemma we will often use the following local
		approximation of $a$: 
		\begin{equation}
			a_{n}(s):=\int_{0}^{1}a(\Delta_{n}x+s)\mathrm{d}x.\label{eq:an_def}
		\end{equation}
		Clearly, $a_{n}\leq a$ and $a_{n}\rightarrow a$ uniformly on compacts,
		thanks to Assumption \ref{as:trawl}. We recall the reader that $\tau_{n}(s):=[s/\Delta_{n}]\Delta_{n}$,
		$s\geq0$.
		
		\begin{lemma}\label{lemmaavar_limit}Let Assumptions of Theorems
			\ref{CLT_Phin} and \ref{CLT_Lambdan_hatLambdan} hold. Then, for all $\ell,\ell^{\prime}=1,2,3,4$
			\begin{equation}
				\Sigma_{n}^{(\ell,\ell^{\prime})}(l,l^{\prime}):=\frac{1}{n\Delta_{n}}\sum_{j=l\lor l^{\prime}}^{n}\mathbb{E}(\xi_{j,l,n}^{(\ell)}\xi_{j,l^{\prime},n}^{(\ell^{\prime})})=F_{n}^{(\ell,\ell^{\prime})}(t_{l},t_{l^{\prime}})+\mathrm{o}(1),\,\,\label{eq:approx_avar_exp}
			\end{equation}
			uniformly on $l,l^{\prime}$, where $F_{n}^{(\ell,\ell^{\prime})}$
			is a sequence of bounded functions such that for all $s,r\geq0$,
			$F_{n}^{(\ell,\ell^{\prime})}(\tau_{n}(s),\tau_{n}(r))\rightarrow F^{(\ell,\ell^{\prime})}(s,r)$
			as $n\rightarrow\infty$. Moreover, 
			\begin{equation}
				\sum_{\ell,\ell^{\prime}=1}^{4}\Sigma_{n}^{(\ell,\ell^{\prime})}(\tau_{n}(s),\tau_{n}(r))\rightarrow\Sigma_{a}(s,r), \,\, 	\forall \, s,r\geq0,\label{final_eq_AVAR} 
			\end{equation}
			where $\Sigma_a$ as in \eqref{eq:AVAR_ALL}.
			
		\end{lemma}
		
		\begin{proof}As in the proof of Lemma \ref{lemmaavar_approx}, we
			analyse $\Sigma_{n}^{(\ell,\ell^{\prime})}(l,l^{\prime})$ case by case.
			In view of $\Sigma_{n}^{(\ell^{\prime},\ell)}(l,l^{\prime})=\Sigma_{n}^{(\ell,\ell^{\prime})}(l^{\prime},l)$,
			it is enough to consider $\ell^{\prime}\geq\ell$. Furthermore, for each pair $(\ell,\ell')$, we construct a function $F_{n}^{(\ell,\ell')}$ such that \eqref{eq:approx_avar_exp} holds. We extend $F_{n}^{(\ell,\ell')}$ by zero outside $[0,n\Delta_n]$. Unless stated otherwise, all error terms $\mathrm{O}(\cdot)$ are uniform in $l,l'$.
			\begin{description}
				\item [{$\ell=\ell^{\prime}=1:$}] Using (\ref{eq:cond_variance}) and
				Lemma \ref{moment_estimates_ahat-1}, we get that 
				\begin{align*}
					\Sigma_{n}^{(1,1)}(l,l^{\prime})= & \frac{\mathfrak{K}_{4}}{\Delta_{n}n}\sum_{j=l\lor l^{\prime}}^{n-1}\mathrm{Leb}(\cup_{i=0}^{j-l\lor l^{\prime}}\mathcal{P}_{A}^{n}(i,j))\\
					& +a_{n}(0)\frac{1}{n}\sum_{j=l\lor l^{\prime}}^{n-1}\mathrm{Leb}(\cup_{i=0}^{j-l\lor l^{\prime}}\cup_{m=j-l\land l^{\prime}}^{j-1}\mathcal{P}_{A}^{n}(i,m))\\
					& -\frac{1}{\Delta_{n}n}\sum_{j=l\lor l^{\prime}}^{n-1}\mathrm{Leb}(\cup_{i=0}^{j-l}\mathcal{P}_{A}^{n}(i,j))\mathrm{Leb}(\cup_{i=0}^{j-l^{\prime}}\mathcal{P}_{A}^{n}(i,j))+\mathrm{O}(\Delta_{n})
				\end{align*}
				Thus,
				\[
				\Sigma_{n}^{(1,1)}(l,l')
				= F_{n,1}^{(1,1)}(t_l,t_{l'}) + F_{n,2}^{(1,1)}(t_l,t_{l'}) + \mathrm{O}(\Delta_n),
				\]
				where for $s,r\leq n\Delta_{n}$ 
				\[
				F_{n,1}^{(1,1)}(s,r)
				=\mathfrak{K}_{4}a_{n}(s\lor r)\Big(1-\frac{s\lor r}{n\Delta_n}\Big),
				\]
				and
				\[
				F_{n,2}^{(1,1)}(s,r)
				=a_{n}(0)\int_{|s-r|}^{s\lor r} a(u)\,\mathrm{d}u\Big(1-\frac{s\lor r}{n\Delta_n}\Big).
				\]
				Since $a_n\leq a\leq a(0)$, $F_{n}^{(1,1)}:=F_{n,1}^{(1,1)}+F_{n,2}^{(1,1)}$ is bounded. Moreover,
				\[
				F_{n}^{(1,1)}(\tau_n(s),\tau_n(r))
				\to \mathfrak{K}_{4}a(s\lor r)
				+ a(0)\int_{|s-r|}^{s\lor r} a(u)\,\mathrm{d}u
				=: F^{(1,1)}(s,r).
				\]
				\item [{$\ell=\ell^{\prime}=2:$}] Just as above we get that 
				\[
				\Sigma_{n}^{(2,2)}(l,l^{\prime})= a_{n}(0)\frac{1}{n}\sum_{j=l\lor l^{\prime}+1}^{n}\int_{t_{l+1}\lor t_{l^{\prime}+1}}^{t_{j+1}}a(u)\mathrm{d}u.
				\]
				Thus, (\ref{eq:approx_avar_exp}) holds with 
				\begin{equation}
					\begin{aligned}F_{n}^{(2,2)}(s,r) & =a_{n}(0)\frac{1}{n}\sum_{j=[s/\Delta_{n}]\lor[r/\Delta_{n}]+1}^{n}\int_{s\lor r}^{t_{j}}a(u)\mathrm{d}u.\end{aligned}
				\end{equation}
				Furthermore, by the DCT
				
				\[F_{n}^{(2,2)}(\tau_{n}(s),\tau_{n}(r))\rightarrow a(0)\int_{s\lor r}^{\infty}a(u)\mathrm{d}u=:F^{(2,2)}(s,r)\]

				\item [{$\ell=\ell^{\prime}=3:$}] In this situation we have due to (\ref{eq:cond_variance})
				that 
				\begin{equation}\label{Sigma33}
					\Sigma_{n}^{(3,3)}(l,l^{\prime})=\frac{1}{n\Delta_{n}}\sum_{j=l\lor l^{\prime}+1}^{n}\sum_{k=l\lor l^{\prime}}^{j-1}\mathrm{Leb}(\mathcal{\mathcal{P}}_{A}^{n}(k+1,j))\int_{t_{l}\lor t_{l^{\prime}}-t_{l}\land t_{l^{\prime}}}^{t_{j}-t_{k}+t_{l}\lor t_{l^{\prime}}}a(u)\mathrm{d}u,
				\end{equation}
				For $m\geq l\lor l'+1$, set $
				A_m^{l,l'}:=\int_{t_{m-1}-t_l\lor t_{l'}}^{t_m-t_l\lor t_{l'}} a(s)\mathrm{d}s$,
				and let \(m=j-k+l\lor l'\). Then the inner sum in \eqref{Sigma33} equals
				\begin{align*}
					\sum_{m=l\lor l^{\prime}+1}^{j}\int_{t_{l}\lor t_{l^{\prime}}-t_{l}\land t_{l^{\prime}}}^{t_{m}}a(u)\mathrm{d}u[A_{m}^{l,l^{\prime}}-A_{m+1}^{l,l^{\prime}}] 
					=&\Delta_{n}a_{n}(0)\int_{t_{l}\lor t_{l^{\prime}}-t_{l}\land t_{l^{\prime}}}^{t_{l+1}\lor t_{l^{\prime}+1}}a(u)\mathrm{d}u\\
					& +\Delta_{n}\int_{t_{l+1}\lor t_{l^{\prime}+1}}^{t_{j}}a_{n}(\tau_{n}(u))a(u-t_{l}\lor t_{l^{\prime}})\mathrm{d}u\\
					& +\mathrm{O}\left(\int_{t_{j}}^{t_{j+1}}a(u-t_{l}\lor t_{l^{\prime}})\mathrm{d}u\right).
				\end{align*}
				Hence, $\Sigma_{n}^{(3,3)}(l,l^{\prime})=F_{n}^{(3,3)}(t_{l},t_{l^{\prime}})+\mathrm{O}(1/(n\Delta_{n})),$
				where 
				\begin{align*}
					F_{n}^{(3,3)}(s,r):=&\frac{1}{n\Delta_{n}}\int_{s\lor r}^{n\Delta_{n}}\int_{s\lor r}^{y}a_{n}(\tau_{n}(u))a(u-s\lor r)\mathrm{d}u\mathrm{d}y\\
					&+a(0)\int_{s\lor r-s\land r}^{s\lor r}a(u)\mathrm{d}u\left(1-\frac{s\lor r}{n\Delta_{n}}\right).
				\end{align*}
				Invoking again the inequality $a_{n}\leq a\leq a(0)$, we obtain that
				$\rvert F_{n}^{(3)}(s,r)\rvert\leq3a(0)\mathrm{Leb}(A)$ and just
				as for $F_{n}^{(2,2)}$ we also get 
				\[
				F_{n}^{(3,3)}(\tau_{n}(s),\tau_{n}(r))\rightarrow\int_{s\lor r}^{\infty}a(u)a(u-s\lor r)\mathrm{d}u+a(0)\int_{\rvert s-r\rvert}^{s\lor r}a(u)\mathrm{d}u=:F^{(3,3)}(s,r).
				\]
				\item [{$\ell=\ell^{\prime}=4:$}] Using \eqref{eq:cond_variance} and making a change of variable, we obtain
				\[
				\Sigma_{n}^{(4,4)}(l,l^{\prime}) 
				=\frac{1}{n}\sum_{j=l\lor l^{\prime}+1}^{n}\sum_{m=l\lor l^{\prime}+1}^{j-1}\int_{t_{m}-t_{l}\lor t_{l^{\prime}}}^{t_{m+1}-t_{l}\lor t_{l^{\prime}}}[a_{n}(0)-a(u)]a_{n}(t_{m})\mathrm{d}u.
				\]
				
				As a result, \eqref{eq:approx_avar_exp} is satisfied with 
				\[
				F_{n}^{(4,4)}(s,r):=\frac{1}{n\Delta_{n}}\int_{s\lor r}^{n\Delta_{n}}\int_{s\lor r}^{y}[a_{n}(0)-a(u-s\lor r)]a_{n}(\tau_{n}(u))\mathrm{d}u\mathrm{d}y,\,\,0\leq r,s\leq n\Delta_{n}.
				\]
				In addition, 
				\[
				F_{n}^{(4,4)}(\tau_{n}(s),\tau_{n}(r))\rightarrow\int_{s\lor r}^{\infty}[a(0)-a(u-s\lor r)]a(u)\mathrm{d}u=:F^{(4,4)}(s,r),
				\]
				thanks to the DCT.
				\item [{$\ell=1,\ell^{\prime}=2:$}] In this case we use relation (\ref{eq:cond_cov})
				to get that 
				\[
				\Sigma_{n}^{(1,2)}(l,l^{\prime})=\frac{1}{n\Delta_{n}}\sum_{j=l\lor l^{\prime}+1}^{n-1}\sum_{k=l^{\prime}}^{j-1}\mathbb{E}(\chi_{k}\beta_{j,j-1,l}^{(2)})\mathbb{E}[(\beta_{k,j,l^{\prime}}^{(1)})^{2}].
				\]
				But 
				\begin{align*}
					\sum_{k=l^{\prime}}^{j-1}\mathbb{E}(\chi_{k}\beta_{j,j-1,l}^{(2)})\mathbb{E}[(\beta_{k,j,l^{\prime}}^{(1)})^{2}] & =\sum_{k=l^{\prime}}^{j-1}\mathbb{E}[(\beta_{j,k,l}^{(1)})^{2}]\mathbb{E}[(\beta_{k,j,l^{\prime}}^{(1)})^{2}]\mathbf{1}_{k\geq j-l}\\
					& =\Delta_{n}\sum_{m=l^{\prime}+1}^{j-(j-l-l^{\prime})^{+}}\int_{t_{m}}^{t_{m+1}}a(u)\mathrm{d}ua_{n}(t_{l}+t_{l^{\prime}}-t_{m}).
				\end{align*}
				Hence, $\Sigma_{n}^{(1,2)}(l,l^{\prime})=F_{n}^{(1,2)}(t_{l},t_{l^{\prime}})+\mathrm{O}(1/n),$
				where
				\[
				F_{n}^{(1,2)}(s,r):=\frac{1}{n}\sum_{j=[s/\Delta_{n}]\lor[r/\Delta_{n}]}^{n}\sum_{m=[r/\Delta_{n}]+1}^{j-(j-[s/\Delta_{n}]-[r/\Delta_{n}])^{+}}\int_{t_{m}}^{t_{m+1}}a(u)a_{n}(s+r-\tau_{n}(u))\mathrm{d}u.
				\]
				The bound $a_n\leq a\leq a(0)$ yields boundedness of $F_n^{(1,2)}$. For fixed $s,r\geq0$, we further have 
				\[
				F_{n}^{(1,2)}(s,r)=\int_{r}^{s+r}a(u)a_{n}(s+r-\tau_{n}(u))\mathrm{d}u\left(1-\frac{\tau_{n}(r)+\tau_{n}(s)}{n\Delta_{n}}\right)+\mathrm{O}(\frac{s\lor r}{n\Delta_{n}}).
				\]
				Consequently, as $n\rightarrow\infty$ 
				\[
				F_{n}^{(1,2)}(\tau_{n}(s),\tau_{n}(r))\rightarrow\int_{r}^{s+r}a(u)a(s+r-u)\mathrm{d}u=:F^{(1,2)}(s,r),
				\]
				by the DCT. 
				\item [{$\ell=1,\ell^{\prime}=3:$}] Exactly as above 
				\[
				-\Sigma_{n}^{(1,3)}(l,l^{\prime})=\frac{1}{n\Delta_{n}}\sum_{j=l\lor l^{\prime}+1}^{n-1}\sum_{k=l^{\prime}}^{j-1}\mathbb{E}(\beta_{j,j-1,l}^{(2)}\beta_{k,j-1,l^{\prime}}^{(2)})\mathbb{E}(\alpha_{k+1,j}\chi_{j}).
				\]
				Write
				\[
				\sum_{k=l'}^{j-1}\mathbb{E}(\beta_{j,j-1,l}^{(2)}\beta_{k,j-1,l'}^{(2)})
				\mathbb{E}(\alpha_{k+1,j}\chi_j)
				=\mathfrak{S}_{j,l,l'}^{n,1}+\mathfrak{S}_{j,l,l'}^{n,2},
				\]
				where \(\mathfrak{S}_{j,l,l'}^{n,1}\) and \(\mathfrak{S}_{j,l,l'}^{n,2}\) correspond to the cases
				\(k\leq j-(l-l')\) and \(k>j-(l-l')\), respectively. 
				where 
				\begin{align*}
					\mathfrak{S}_{j,l,l^{\prime}}^{n,1} & :=\sum_{i=0}^{j-(l^{\prime}+1)\lor l}\mathrm{Leb}(\cup_{m=j-l}^{j-1}\mathcal{\mathcal{P}}_{A}^{n}(i,m))\sum_{k=i+l^{\prime}}^{j-1-(l-l^{\prime}+1)^{+}}\mathrm{Leb}(\mathcal{\mathcal{P}}_{A}^{n}(k+1,j)),\\
					\mathfrak{S}_{j,l,l^{\prime}}^{n,2} & :=\sum_{m=j-l+1}^{j-1}\mathrm{Leb}(\cup_{i=0}^{j-l}\mathcal{\mathcal{P}}_{A}^{n}(i,m))\sum_{k=l^{\prime}}^{j-1}\mathrm{Leb}(\mathcal{\mathcal{P}}_{A}^{n}(k+1,j))\mathbf{1}_{m+l^{\prime}\geq k\geq j-(l-l^{\prime})+1}.
				\end{align*}
				Simple algebraic manipulations result in 
				\begin{align*}
					\mathfrak{S}_{j,l,l^{\prime}}^{n,1}=&\sum_{i=0}^{j-(l^{\prime}+1)\lor l}\mathrm{Leb}(\cup_{m=j-l}^{j-1}\mathcal{\mathcal{P}}_{A}^{n}(i,m))\sum_{k=i+l^{\prime}}^{j-1-(l-l^{\prime}+1)^{+}}\mathrm{Leb}(\mathcal{\mathcal{P}}_{A}^{n}(k+1,j))\\
					&= \Delta_{n}\sum_{m=l\lor(l^{\prime}+1)}^{j-1}\int_{t_{m}}^{t_{m+1}}[a(u-t_{l})-a(u)][a_{n}((t_{l}-t_{l^{\prime}+1})^{+})-a_{n}(t_{m}-t_{l^{\prime}})]\mathrm{d}u\\
					& +\mathrm{O}\left(\Delta_{n}\int_{t_{j}-t_{l}}^{\infty}a(u)\mathrm{d}u\right)
				\end{align*}
				and 
				\[\mathfrak{S}_{j,l,l^{\prime}}^{n,2}  =\Delta_{n}\sum_{m=1}^{l-1}\int_{t_{m}}^{t_{m+1}}a(u)[a_{n}((t_{l}-t_{l^{\prime}+1}-t_{m})^{+})-a_{n}((t_{l}-t_{l^{\prime}+1})^{+})]\mathrm{d}u.\]
				Hence, (\ref{eq:approx_avar_exp}) holds with $-F_{n}^{(1,3)}(s,r)$ equal to
				\begin{align*}
					\frac{1}{n\Delta_{n}}\int_{s\lor r}^{n\Delta_{n}}\int_{s\lor r}^{y}[a(u-s)&-a(u)][a_{n}((s-r-\Delta_{n})^{+})-a_{n}(\tau_{n}(u)-s)]\mathrm{d}u\mathrm{d}y\\
					+\int_{0}^{s}a(u)[a_{n}((s-r-\Delta_{n}-\tau_{n}(u))^{+})&-a_{n}((s-r-\Delta_{n})^{+})]\mathrm{d}u\left(1-\frac{s\lor r}{n\Delta_{n}}\right).
				\end{align*}
				In addition,
				\begin{align*}
					F_{n}^{(1,3)}(\tau_{n}(s),\tau_{n}(r))\rightarrow & -\int_{s\lor r}^{\infty}[a(u-s)-a(u)][a((s-r)^{+})-a(u-s)]\mathrm{d}u\\
					& -\int_{0}^{s}a(u)[a((s-r-u)^{+})-a((s-r)^{+})]\mathrm{d}u,
				\end{align*}
				once again by the DCT.
				\item [{$\ell=1,\ell^{\prime}=4:$}] We start by noting that from Lemma
				\ref{moment_estimates_ahat-1}		
				\[\mathbb{E}(\beta_{j,j,l}^{(1)}\chi_{j}\alpha_{k+1,j}\beta_{k,j,l^{\prime}}^{(1)})  =\mathbb{E}(\beta_{k,l^{\prime},j}^{(1)}\beta_{j,l,j}^{(1)})\mathbb{E}(\alpha_{k+1,j}^{2})(1+\mathbf{1}_{k\leq j-l-1}).
				\]
				Thus, 
				\[
				\sum_{k=l^{\prime}}^{j-1}\mathbb{E}(\beta_{j,j,l}^{(1)}\chi_{j}\alpha_{k+1,j}\beta_{k,j,l^{\prime}}^{(1)})\leq\sum_{i=0}^{j-l^{\prime}-1}\mathrm{Leb}(\mathcal{\mathcal{P}}_{A}^{n}(i,j))\sum_{k=i+l^{\prime}}^{j-1}\mathrm{Leb}(\mathcal{\mathcal{P}}_{A}^{n}(k+1,j))\leq(a(0)\Delta_{n})^{2}.
				\]
				Therefore, in this case we have that 
				\[
				\Sigma_{n}^{(1,4)}(l,l^{\prime})=\frac{1}{n\Delta_{n}}\sum_{j=l\lor l^{\prime}+1}^{n-1}\sum_{k=l^{\prime}}^{j-1}\mathbb{E}(\varsigma_{k,j}\beta_{j,j-1,l}^{(2)})\mathrm{Leb}(\cup_{i=0}^{k-l^{\prime}}\mathcal{P}_{A}^{n}(i,j))+\mathrm{O}(\Delta_{n}).
				\]
				Since, 
				\[
				\mathbb{E}(\varsigma_{k,j}\beta_{j,l}^{(4)})=\mathrm{Leb}(\cup_{m=j-l}^{j-1}\mathcal{\mathcal{P}}_{A}^{n}(k+1,m))\mathbf{1}_{k\leq j-l-1},
				\]
				we conclude that for $j\geq l^{\prime}+l+1$	
				\[	
				\sum_{k=l^{\prime}}^{j-1}\mathbb{E}(\varsigma_{k,j}\beta_{j,j-1,l}^{(2)})\mathrm{Leb}(\cup_{i=0}^{k-l^{\prime}}\mathcal{P}_{A}^{n}(i,j))= 
				\Delta_{n}\sum_{m=l}^{j-l^{\prime}-1}\int_{t_{m}}^{t_{m+1}}\left[a(u-t_{l})-a(u)\right]\mathrm{d}ua_{n}(t_{m}+t_{l^{\prime}+1}).\]
				Therefore, \eqref{eq:approx_avar_exp} is fulfilled,
				with
				\[
				-F_{n}^{(1,4)}(s,r)=\frac{1}{n\Delta_{n}}\int_{s}^{n\Delta_{n}-r}\int_{s}^{y}[a(u-s)-a(u)]a_{n}(\tau_{n}(u)+r+\Delta_{n})]\mathrm{d}u\mathrm{d}y.
				\]
				Finally, as above, we obtain 
				\begin{align*}
					F_{n}^{(1,4)}(\tau_{n}(s),\tau_{n}(r))\rightarrow & -\int_{s}^{\infty}[a(u-s)-a(u)]a(u+r)\mathrm{d}u=:F^{(1,4)}(s,r).
				\end{align*}
				\item [{$\ell=2,\ell^{\prime}=3:$}] Plainly 
				\[-\Sigma_{n}^{(2,3)}(l,l^{\prime})=\frac{1}{n\Delta_n} \sum_{j=l^{\prime}+l^{\prime}+2}^{n}\sum_{k=l}^{j-1}\sum_{k^{\prime}=l^{\prime}}^{j-1}\mathbb{E}(\chi_{k}\beta_{k^{\prime},j-1,l^{\prime}}^{(2)})\mathbb{E}(\alpha_{k^{\prime}+1,j}^{2})\mathbf{1}_{k^{\prime}\leq(k-l-1)\land(k+l^{\prime})}\mathbf{1}_{k\geq l+l^{\prime}+1}.\]
				Rearranging the sums yields
				\begin{align*}
					-\Sigma_{n}^{(2,3)}(l,l^{\prime})&=\frac{1}{n\Delta_{n}}\sum_{k=l+l^{\prime}+1}^{n-1}\sum_{k^{\prime}=l^{\prime}}^{k-l-1}\mathrm{Leb}(\cup_{i=0}^{k^{\prime}-l^{\prime}}\mathcal{\mathcal{P}}_{A}^{n}(i,k))\sum_{j=k+1}^{n}\mathrm{Leb}(\mathcal{\mathcal{P}}_{A}^{n}(k^{\prime}+1,j))\\
					= & \frac{1}{n}\sum_{k=l+l^{\prime}+1}^{n-1}\sum_{k^{\prime}=l^{\prime}}^{k-l-1}a_{n}(t_{k}-t_{k^{\prime}}+t_{l^{\prime}})\int_{t_{k}-t_{k^{\prime}}}^{t_{k+1}-t_{k^{\prime}}}a(u)\mathrm{d}u+\mathrm{O}(\int_{n\Delta_{n}}^{\infty}a(u)\mathrm{d}u).
				\end{align*}
				This mean that (\ref{eq:approx_avar_exp}) is fulfilled with 
				\begin{align*}
					-F_{n}^{(2,3)}(s,r) & =\frac{1}{n\Delta_{n}}\int_{s+r}^{n\Delta_{n}}\int_{s}^{y-r}a(u)a_{n}(\tau_{n}(u)+r)]\mathrm{d}u\mathrm{d}y,
				\end{align*}
				satisfying
				\[
				F_{n}^{(2,3)}(\tau_{n}(s),\tau_{n}(r))\rightarrow-\int_{s}^{\infty}a(u)a(u+r)\mathrm{d}u=:F^{(2,3)}(s,r).
				\]
				\item [{$\ell=2,\ell^{\prime}=4:$}] From \eqref{eq:cond_cov2},
				\begin{align*}
					-\Sigma_{n}^{(2,4)}(l,l^{\prime})= & \frac{1}{n\Delta_{n}}\sum_{j=l\lor l^{\prime}+1}^{n-1}\sum_{k=l}^{j-1}\sum_{k^{\prime}=l^{\prime}}^{j-1}\mathbb{E}(\chi_{k}\varsigma_{k^{\prime},j})\theta_{k,l,k^{\prime},l^{\prime}}^{j}.
				\end{align*}
				where $\theta_{k,l,k^{\prime},l^{\prime}}^{j}$ as in \eqref{theta_kl_def}. Moreover,
				\begin{align*}
					\sum_{k=l,k^{\prime}=l^{\prime}}^{j-1}\mathbb{E}(\chi_{k}\varsigma_{k^{\prime},j})\theta_{k,l,k^{\prime},l^{\prime}}^{j}
					&=  \sum_{k=l,k^{\prime}=l^{\prime}}^{j-1}\mathrm{Leb}(\cup_{i=0}^{k-l}\mathcal{\mathcal{P}}_{A}^{n}(i,j))\mathrm{Leb}(\mathcal{\mathcal{P}}_{A}^{n}(k^{\prime}+1,k))\mathbf{1}_{k-1\geq k^{\prime}\geq k-(l-l^{\prime})}\\
					& +\sum_{k=l,k^{\prime}=l^{\prime}}^{j-1}\mathrm{Leb}(\cup_{i=0}^{k^{\prime}-l^{\prime}}\mathcal{\mathcal{P}}_{A}^{n}(i,j))\mathrm{Leb}(\mathcal{\mathcal{P}}_{A}^{n}(k^{\prime}+1,k))\mathbf{1}_{k^{\prime}+1+(l-l^{\prime})^{+}\leq k}\\
					= & \Delta_{n}[a_{n}(0)-a_{n}((t_{l}-t_{l^{\prime}})^{+})]\int_{t_{l+1}}^{t_{j+1}}a(u)\mathrm{d}u\\
					& +\Delta_{n}\sum_{m=l\lor l^{\prime}+2}^{j}\int_{t_{m}}^{t_{m+1}}a(u)\mathrm{d}u[a_{n}((t_{l}-t_{l^{\prime}})^{+})-a_{n}(t_{m}-t_{l^{\prime}+1})].
				\end{align*}
				Therefore, in this situation we may choose $-F^{(2,4)}_n(s,r)$ as
				\begin{align*}
					&\frac{1}{n\Delta_{n}}\int_{s\lor r}^{n\Delta_{n}}\int_{s\lor r}^{y}a(u)[a_{n}((s-r)^{+})
					-a_{n}(\tau_{n}(u)-r-\Delta_{n})]\mathrm{d}u\mathrm{d}y\\
					&+[a_{n}(0)-a_{n}((s-r)^{+})]\frac{1}{n\Delta_{n}}\int_{s\lor r}^{n\Delta_{n}}\int_{s}^{y}a(u)\mathrm{d}u\mathrm{d}y,	
				\end{align*}
				and 
				\[
				-F^{(2,4)}(s,r)=[a(0)-a((s-r)^{+})]\int_{s}^{\infty}a(u)\mathrm{d}u+\int_{s\lor r}^{\infty}a(u)[a((s-r)^{+})-a(u-r)]\mathrm{d}u.\]
				\item [{$\ell=3,\ell^{\prime}=4:$}] Finally, for all $l,l^{\prime}\geq0$,
				$k\geq l$ and $k^{\prime}\geq l^{\prime}$
				
				\[	\mathbb{E}[\beta_{k,j-1,l}^{(2)}\varsigma_{k^{\prime},j}]\mathbb{E}[\alpha_{k+1,j}\beta_{k^{\prime},j,l^{\prime}}^{(1)}]\mathbf{1}_{k+1\leq k^{\prime}-l^{\prime}}=\mathbb{E}[\beta_{k,j-1,l}^{(2)}\varsigma_{k^{\prime},j}]\mathbb{E}[\alpha_{k+1,j}^{2}]\mathbf{1}_{k^{\prime}+l+1\leq k\leq k^{\prime}-l^{\prime}-1}=0.\]
				Therefore, $\Sigma_{n}^{(3,4)}(l,l^{\prime})=0$.
			\end{description}
			We are left to show that \eqref{final_eq_AVAR} holds. Using the relation $\Sigma_{n}^{(\ell^{\prime},\ell)}(l,l^{\prime})=\Sigma_{n}^{(\ell,\ell^{\prime})}(l^{\prime},l)$
			and the first part of the proof, we obtain
			\begin{equation}
				\sum_{\ell,\ell^{\prime}=1}^{4}\Sigma_{n}^{(\ell,\ell^{\prime})}(\tau_{n}(s),\tau_{n}(r))\rightarrow\sum_{\ell=1}^{4}F^{(\ell,\ell)}(s,r)+\sum_{\ell=1}^{3}\sum_{\ell^{\prime}=\ell+1}^{4}\widetilde{F}^{(\ell,\ell^{\prime})}(s,r). \label{semi_expl_Sigma_a} 
			\end{equation}

			Using the explicit expressions obtained above gives that
			\[
			\sum_{\ell=1}^{4}F^{(\ell,\ell)}(s,r)=\mathfrak{K}_{4}a(s\lor r)+2a(0)\int_{\rvert s-r\rvert}^{\infty}a(u)\mathrm{d}u,
			\]
			as well as 
			\[
			F^{(1,2)}(s,r)+F^{(1,4)}(s,r)+F^{(2,3)}(s,r)=\sigma_{3}(s,r),
			\]
			and
			\[
			F^{(1,3)}(s,r)+F^{(2,4)}(s,r)
			=
			\sigma_{2}(s,r)-a(0)\int_{|s-r|}^{\infty}a(u)\,\mathrm{d}u.
			\]
			Plugging in the previous equations into \eqref{semi_expl_Sigma_a} completes the proof.\end{proof}

		\bibliographystyle{plain}
		\bibliography{bibSept24}
		
	\end{document}